\documentclass[reqno]{amsart}
\usepackage{microtype}
\usepackage[sortcites=true]{biblatex}
\usepackage{links}
\usepackage{graphicx}
\usepackage{amsmath,amssymb,amsfonts}
\usepackage{amsthm}
\usepackage[title]{appendix}
\usepackage{xcolor}
\usepackage{textcomp}
\usepackage{booktabs}
\usepackage{standard}
\usepackage{mathtools}
\usepackage[nameinlink]{cleveref}
\usepackage{subcaption}

\newcommand{\ER}{\R \cup \{+\infty\}}
\newcommand{\trace}{\gamma}

\DeclareMathOperator{\li}{Li}
\DeclareMathOperator{\dist}{dist}

\newtheorem{theorem}{Theorem}
\newtheorem{lemma}{Lemma}
\newtheorem{proposition}{Proposition}
\newtheorem{assumption}{Assumption}
\newtheorem{corollary}{Corollary}

\newtheorem{remark}{Remark}

\newtheorem{definition}{Definition}

\Crefname{equation}{}{}
\crefname{equation}{}{}

\title[Analysis of Bregman Proximal Point for the Obstacle Problem]{Analysis of Bregman Proximal Point \\ for the Obstacle Problem}
\author{Brendan Keith, Haojun Qin, Noe Reyes Rivas}

\begin{document}

\begin{abstract}
We study the Bregman proximal point method for the obstacle problem, a
fundamental variational inequality arising in contact mechanics, optimal
design, and mathematical finance. Each Bregman proximal step regularizes the
energy through a Bregman divergence generated by a Legendre function, leading
to a semilinear elliptic subproblem. We first establish a well-posedness and
strict feasibility theory for these subproblems. We then analyze the
convergence of the resulting iteration in the \(H^1\)-norm. Our convergence
analysis is based on a sequential strict local minimality inequality. We show
abstractly how the order of minimality determines the convergence rate.
Assuming that the initial guess lies above the exact solution and a one-sided
Bregman growth condition holds, we establish sequential strict local minimality
of order $s \geq 2$ and derive sublinear convergence rates in the $H^1$-norm of
the form $\mathcal{O}(k^{-\zeta})$, where $\zeta\in(1/2,1]$ depends on the
choice of Legendre function. For the Shannon and Tsallis entropies, we show
that these rates are sharp in a uniform worst-case sense. Under additional
assumptions on the free boundary and the obstacle, we establish sequential strict local minimality of order $s = 1$ with the Shannon and Spence entropies. This
yields linear convergence of the form $\mathcal{O}(\varrho^k)$ for some
$\varrho\in(0,1)$.
\end{abstract}


\maketitle

\section{Introduction}

Variational inequalities with pointwise constraints arise naturally in contact
mechanics \cite{kikuchi1988contact,wriggers2007computational}, mathematical
image processing \cite{ambrosio1990approximation}, mathematical finance
\cite{cont2003financial}, and many other areas of applied analysis. A
prototypical example is the obstacle problem, where one minimizes an energy
functional over a closed convex set of functions constrained not to pass through
a prescribed obstacle \cite{bartels2015numerical}. Even for the quadratic
Dirichlet energy, this inequality constraint introduces a free boundary and
destroys the direct Euler--Lagrange equation structure available for
unconstrained elliptic problems \cite{kinderlehrer2000introduction}. As a
consequence, both the design and analysis of solution algorithms require tools
that are sensitive to the geometry of the feasible set.

Bregman proximal point methods provide an elegant framework for solving
variational inequalities with pointwise constraints. The basic idea is to
replace the original constrained problem by a sequence of regularized
subproblems in which the feasible set's geometry is incorporated into a suitable
Legendre function \cite{rockafellar1967conjugates}. This viewpoint has recently
been used in numerical algorithms for variational problems, including the
proximal Galerkin method \cite{keith2024proximal,keith2025priori}.

For the obstacle problem with feasible set
\begin{equation*}
    K = \{u \in H_g^1(\Omega) \mid u \geq \phi \text{ a.e.\ in } \Omega\},
\end{equation*}
a canonical choice of Legendre function is the shifted Shannon entropy
\begin{equation}
    \label{eq:GeneralizedShannonEntropy}
    R(u) = (u - \phi) \log(u - \phi) - (u - \phi).
\end{equation}
The precise definition of a Legendre function and its associated Bregman
divergence will be given in \Cref{sec:problem-setting}, below. With the choice
\eqref{eq:GeneralizedShannonEntropy}, the Bregman proximal subproblems take the
form of semilinear elliptic partial differential equations (PDEs). For example,
to minimize the Dirichlet energy over the feasible set,
\begin{equation*}
    E(u)
    = \frac{1}{2} \int_{\Omega} |\nabla u|^2 \diff x - \int_{\Omega} fu \diff x,
    \quad u \in K,
\end{equation*}
the formal optimality condition at the \(k\)-th step is
\begin{equation}
    \label{eq:entropic-poisson}
    (\nabla u^k,\nabla v) + \alpha_k^{-1} (\log(u^k - \phi),v)
    = (f,v) + \alpha_k^{-1} (\log(u^{k-1} - \phi),v)
    \quad \forall v \in H_0^1(\Omega).
\end{equation}
We refer to \eqref{eq:entropic-poisson} as \textit{the entropic Poisson
equation}. The logarithmic singularity in \eqref{eq:entropic-poisson} keeps the
iterates strictly feasible, while still allowing the limiting solution to make
contact with the obstacle.

The entropic Poisson equation is attractive because it transforms the
constrained problem into a sequence of (locally) smooth semilinear elliptic
PDEs. However, its rigorous analysis is subtle. The logarithm is
meaningful only when \(u^k-\phi\) is bounded away from zero, a property that
cannot be assumed a priori because the solution of the obstacle problem may touch
the obstacle on sets of positive measure. Thus, a fundamental task is to prove
that each proximal subproblem is well-posed and that its solution is strictly
separated from the obstacle in the \(L^\infty\) sense. In this paper, we
establish this property by introducing a regularized entropy, deriving uniform
\(L^\infty\) bounds for the corresponding entropy variables, and then passing
back to the original logarithmic equation.

Our second goal is to quantify the convergence rate of the resulting Bregman
proximal point iteration. Recent work shows that these rates can depend
strongly on the geometry induced on the feasible set by the underlying Legendre
function. In the finite-dimensional setting, \cite{azizian2024rate} studies
last-iterate convergence of Bregman proximal methods for variational
inequalities through the associated local geometry. Their Legendre exponent
quantifies the local growth of the Bregman divergence near the solution and
leads to rates that reflect this geometry. In the present infinite-dimensional
setting, the analogous role will be played by a one-sided Bregman growth
condition adapted to both the positive invariance of the iterates and the PDE
formulation.

For the obstacle problem specifically, \cite{keith2024proximal} introduces and
analyzes a Bregman proximal point method based on the Shannon entropy. The
estimate
\begin{equation*}
    \|u^k-u^*\|_{H^1(\Omega)}
    \lesssim \left(\sum_{j=1}^k\alpha_j\right)^{-1/2}
\end{equation*}
is established there, yielding the baseline rate \(\mathcal O(k^{-1/2})\) under
constant step sizes. This worst-case estimate follows from a classical proximal
point argument \cite{chen1993convergence} and does not exploit finer properties
of the underlying Legendre function, the Lagrange multiplier, or the free
boundary. Numerical experiments reported in \cite{keith2024proximal} suggest
faster worst-case convergence for the Shannon entropy, while
\cite[Remark~4.18]{keith2024proximal} conjectures that the rate can be
(further) improved under a strict complementarity assumption. These
observations leave open how the Legendre geometry can be used to obtain sharper
algebraic rates in the infinite-dimensional obstacle setting and whether
additional free-boundary structure leads to even faster convergence.

The main contributions of this work are therefore threefold. First, we give a
self-contained well-posedness and regularity theory for the generalized
entropic Poisson equations associated with the Bregman proximal point method.
Second, we identify a convergence mechanism that relates the convergence rate
to the order of minimality. Under a one-sided Bregman growth condition
with exponent \(\theta\in(0,1]\), we establish sequential strict local minimality
of order \(s=2/\theta\), yielding the dimension-independent \(H^1\)-convergence
rate \(\mathcal O(k^{-1/(2-\theta)})\) for constant step sizes and sharpening
the estimate in \cite[Theorem 4.13]{keith2024proximal}. For the Shannon and
Tsallis entropies, we also construct families of admissible obstacle problems
to prove that the resulting sublinear rates are sharp in a uniform worst-case
sense. Third, we show that additional assumptions on the obstacle and free
boundary yield sequential strict local minimality of order $s = 1$ for the
Shannon and Spence entropies, and consequently linear convergence of both the
Bregman divergence and the natural energy error.

The rest of the paper is organized as follows. \Cref{sec:problem-setting}
introduces the obstacle problem including the energy form and free boundary,
the class of Legendre functions considered in this work, and the Bregman
proximal point iteration. \Cref{sec:main-results} states the well-posedness
results, the a priori properties of the iterates, and the convergence
estimates. \Cref{sec:wellposedness-regularity} proves existence, uniqueness,
and some a priori estimates of the generalized entropic Poisson equation.
\Cref{sec:convergence-rate} contains the proof of the sublinear and linear
convergence estimates. We conclude in \Cref{sec:conclusion}. The appendices
verify the structural assumptions for the model Bregman divergences used
throughout the paper and collect several auxiliary results.

\section{Algorithm and Problem Setting}
\label{sec:problem-setting}

Let \(\Omega \subset \R^d\) be a bounded Lipschitz domain. We denote by
$H_0^1(\Omega)$ the homogeneous space of admissible variations. When
nonhomogeneous Dirichlet data are prescribed, we write
\begin{equation*}
    H_g^1(\Omega)
    \coloneqq \{v \in H^1(\Omega) \mid \trace v = g \text{ on } \partial\Omega\}
\end{equation*}
for the associated affine space, where $g \in H^{1/2}(\partial \Omega)$ and
\(\trace\) denotes the trace operator \(\trace \colon H^1(\Omega) \to
H^{1/2}(\partial\Omega)\). Let $\phi \in H^1(\Omega)$ be a prescribed obstacle.
The feasible set is
\begin{equation*}
    K
    \coloneqq  \{v \in H_g^1(\Omega) \mid v \geq \phi \text{ a.e.\ in } \Omega\}.
\end{equation*}
Throughout the paper, we assume that \(K\) is nonempty. More precisely, we
impose the strict boundary compatibility condition that there exists a constant
\(\delta_0 > 0\) such that
\begin{equation}
\label{eq:boundary-compatibility}
    g - \trace \phi > \delta_0
    \quad \text{ a.e.\ on } \partial\Omega.
\end{equation}

\subsection{Obstacle problems}

We consider the obstacle problem
\begin{equation}
\label{eq:obstacle-problem}
    u^* \in \argmin_{u \in K} E(u),
\end{equation}
where \(E\) is a quadratic energy of divergence form:
\begin{equation}
\label{eq:energy-general-A}
    E(u)
    = \frac{1}{2} \int_\Omega A(x) \nabla u \cdot \nabla u \diff x - F(u).
\end{equation}
Here \(F\) is a bounded linear functional, and \(A(x) = (a_{ij}(x))_{i,j=1}^d\)
is a symmetric, uniformly elliptic, \(C^{1,\alpha}\) matrix field on
\(\overline \Omega\). More precisely, we assume that there exists a constant $0
< \Lambda < \infty$ such that
\begin{equation}
\label{eq:uniform-ellipticity}
    \frac{1}{\Lambda} |\xi|^2
    \leq A(x) \xi \cdot \xi
    \leq \Lambda |\xi|^2
    \quad \text{ for a.e. } x \in \Omega,\, \xi \in \R^d.
\end{equation}
We write
\begin{equation*}
    a(u,v)
    \coloneqq  \int_\Omega A(x) \nabla u \cdot \nabla v \diff x.
\end{equation*}
Then
\begin{equation}
\label{eq:energy-general}
    E(u)
    = \frac{1}{2} a(u,u) - F(u).
\end{equation}
The ellipticity condition \eqref{eq:uniform-ellipticity} implies that
\(a(\cdot,\cdot)\) is continuous and coercive on \(H_0^1(\Omega)\). Consequently, since
\(K\) is closed, convex, and nonempty, the obstacle problem
\eqref{eq:obstacle-problem} admits a unique minimizer \(u^* \in K\).

The minimizer \(u^* \in K\) is equivalently characterized by the first-order
optimality condition
\begin{equation}
\label{eq:abstract-obstacle-vi}
    \langle E'(u^*),v - u^*\rangle
    \geq 0
    \quad \forall v \in K ;
\end{equation}
cf.\ \cite[Theorem 6.2]{kinderlehrer2000introduction}. For the quadratic energy
\eqref{eq:energy-general}, this condition becomes the variational inequality
\begin{equation}
\label{eq:obstacle-vi}
    a(u^*,v - u^*) \geq F(v - u^*)
    \quad \forall v \in K .
\end{equation}
When \(F(v) = (f,v)\), the distributional residual
$
\lambda^* \coloneqq -\nabla \cdot (A \nabla u^*) - f
$
may be interpreted as the obstacle multiplier. Formally, the variational
inequality corresponds to the complementarity system
\begin{equation*}
    u^* \geq \phi,
    \qquad \lambda^* \geq 0,
    \qquad \lambda^*(u^* - \phi) = 0.
\end{equation*}
This formulation separates the domain into the inactive and active regions.
Formally, the inactive set is
\begin{equation*}
    \Omega_+
    \coloneqq  \{x \in \Omega \mid u^*(x) > \phi(x)\},
\end{equation*}
where the constraint is inactive and hence \(\lambda^* = 0\). On this set, the
solution satisfies the (unconstrained) Euler--Lagrange equation
\begin{equation*}
    -\nabla \cdot (A \nabla u^*)
    = f.
\end{equation*}
The active set is
\begin{equation*}
    \Omega_0
    \coloneqq  \{x \in \Omega \mid u^*(x) = \phi(x)\},
\end{equation*}
where the solution touches the obstacle and the multiplier \(\lambda^*\) may be
nonzero. The interface between these two regions,
\begin{equation*}
    \Gamma
    \coloneqq \partial\Omega_0 \cap \Omega,
\end{equation*}
is the free boundary; see \Cref{fig:free-boundary}. The presence of this free
boundary is one of the main features that distinguishes the obstacle problem
from a standard elliptic boundary value problem.

\begin{figure}[htbp]
    \centering
    \includegraphics[width=0.5\textwidth]{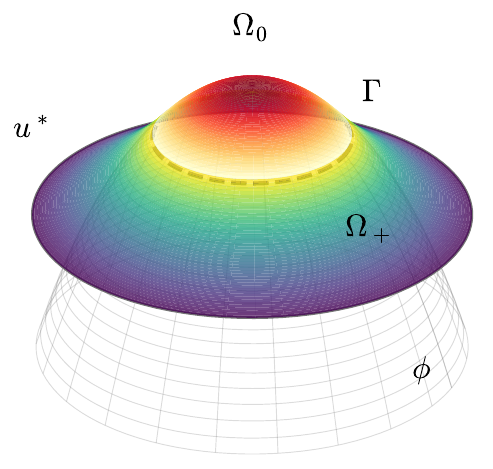}
    \caption{Exact solution $u^*$ and radial obstacle $\phi$ on the unit disk
    $\Omega = \{x \in \R^2 \mid |x| < 1\}$. The central region is the active
    set \(\Omega_0\), the outer region is the inactive set \(\Omega_+\), and
    the dashed curve denotes the free boundary \(\Gamma\).}
    \label{fig:free-boundary}
\end{figure}

\subsection{Legendre functions and Bregman divergence}
\label{subsec:leg-func-breg-div}

We next introduce the Legendre functions used in the Bregman proximal point
method to encode the obstacle constraint.

\begin{definition}[Legendre function]
Let the essential domain of a function \(R \colon \R \to \ER\) be defined by
\begin{equation*}
    \dom R
    \coloneqq \{a \in \R \colon R(a) < +\infty\}.
\end{equation*}
We call a proper, lower-semicontinuous, convex function \(R\) a Legendre
function if
\begin{itemize}
    \item \(\intr(\dom R) \neq \emptyset\);
    \item \(R\) is differentiable on \(\intr(\dom R)\);
    \item for every \(a \in \partial(\dom R)\) and every \(b \in \intr(\dom R)\),
    \begin{equation*}
        \lim_{t \to 0^+} R'(a + t(b-a)) (b-a)
        = -\infty;
    \end{equation*}
    \item \(R\) is strictly convex on \(\intr(\dom R)\).
\end{itemize}
\end{definition}

This class of convex functions was introduced by Rockafellar
\cite{rockafellar1967conjugates,rockafellar1997convex}. We denote by \(R^*\)
the convex conjugate of \(R\), namely
\begin{equation*}
    R^*(a^*)
    = \sup_{a\in\R}\, \{aa^* - R(a)\}.
\end{equation*}
For Legendre functions, the gradient map
\begin{equation*}
    R' \colon \intr(\dom R) \to \intr(\dom R^*)
\end{equation*}
is a topological isomorphism, and its inverse is the gradient of the convex
conjugate:
\begin{equation}
\label{eq:legendre-duality}
    (R^*)'=(R')^{-1};
\end{equation}
see \cite[Theorem~1]{rockafellar1967conjugates}. Thus, the inverse gradient
maps the dual domain into the interior of the feasible domain. We refer the
interested reader to \cite{keith2024proximal,dokken2025latent} for the
associated latent-variable interpretation.

The Legendre function \(R\) induces a non-symmetric notion of distance within
\(\dom R\) via its Bregman divergence \cite{bregman1967relaxation}:
\begin{equation*}
    D_R(a,b)
    = R(a) - R(b) - R'(b)(a-b),
    \quad a \in \dom R,
    \quad b \in \intr(\dom R).
\end{equation*}
This quantity measures the error in the first-order Taylor expansion of \(R\).
Since \(R\) is strictly convex,
\begin{equation*}
    D_R(a,b) \geq 0,
    \qquad
    D_R(a,b) = 0 \iff a = b.
\end{equation*}
We also recall the three-point identity \cite[Lemma~3.1]{chen1993convergence}:
\begin{equation}
\label{eq:three-point-identity}
    D_R(a,b)-D_R(a,c)+D_R(b,c)
    = \bigl(R'(b)-R'(c)\bigr)(b-a),
\end{equation}
which will be used repeatedly in the convergence analysis.

Four scalar Legendre functions considered in this paper are summarized in
\Cref{tab:four-legendre-functions}.
\begin{table}[h]
\scriptsize
    \centering
    \renewcommand{\arraystretch}{2.0}
    \begin{tabular}{llll}
        \toprule
        Type & \(r(s)\) & \(r'(s)\) & \((r^*)'(\psi) = (r')^{-1}(\psi)\) \\
        \midrule
        Shannon entropy
            & \(s\log s - s\)
            & \(\log s\)
            & \(\exp(\psi)\)
            \\
        Tsallis entropy, \(1 < q < 2\)
            & \(\displaystyle
                \frac{s^{2-q} - (2-q)s + (1-q)}{(q-1)(q-2)}
              \)
            & \(\displaystyle
                \frac{s^{1-q} - 1}{1-q}
              \)
            & \(\displaystyle
                (1 - (q-1)\psi)^{-1/(q-1)}
              \)
            \\
        Kaniadakis entropy, \(0 < \kappa < 1\)
            & \(\displaystyle
                \frac{s^{\kappa+1}}{2\kappa(\kappa+1)}
                - \frac{s^{1-\kappa}}{2\kappa(1-\kappa)}
              \)
            & \(\displaystyle
                \frac{s^\kappa - s^{-\kappa}}{2\kappa}
              \)
            & \(\displaystyle
                \left(\sqrt{1 + \kappa^2\psi^2} + \kappa\psi\right)^{1/\kappa}
              \)
            \\
        Spence entropy
            & \(\displaystyle
                \frac{1}{2}s^2 + \li_2(e^{-s}) - \frac{\pi^2}{6}
              \)
            & \(\log(e^s - 1)\)
            & \(\log(1 + e^\psi)\)
            \\
        \bottomrule
    \end{tabular}
    \caption{Four model Legendre functions. The shifted functions are
    \(R_i(x, u) = r_i(u - \phi(x))\).}
    \label{tab:four-legendre-functions}
\end{table}

\begin{remark}[Legendre functions singular at $s=0$]
The Legendre functions in \Cref{tab:four-legendre-functions} admit finite
continuous extensions to $s=0$. This property is important because the
exact solution may satisfy $u^*-\phi=0$ on the active set. By contrast, the
Burg entropy \(r_{\mathrm B}(s)=-\log s\) and the superlinear Burg-type
entropy \(r_{\mathrm{SB}}(s)= -\log(s) + s^2\) satisfy \(r(0)=+\infty\).
The associated Bregman divergence \(D_R(u^*,u^k)\) need not be finite when
\(u^*\) touches the obstacle. Consequently, these Legendre functions are
not covered by the present convergence framework.
\end{remark}

The scalar Legendre functions $r$ in \Cref{tab:four-legendre-functions}
satisfy: \(\dom r = [0,\infty)\) with \(r''(s) > 0\) for all \(s \in
(0,\infty)\).

An additional growth condition is often imposed on Legendre functions---that
is, superlinearity \cite{dokken2025latent,keith2024proximal,keith2025priori}.
There are many characterizations of superlinear functions; for our purposes, a
function $r : \R \to \ER$ is said to be superlinear if $\lim_{|a| \to \infty}
r(a)/|a| = \infty$. For Legendre functions, superlinearity allows its convex
conjugate to be finite-valued everywhere, i.e., $\dom r^* = \R$, which allows
for simpler analysis.

In this work, we do not assume the Legendre functions under consideration are
superlinear, as is the case with the Tsallis entropy. Instead, we weaken this
condition: in addition to our prior assumptions $\dom r = [0,\infty)$ and
$r''(s) > 0$ for $s \in (0,\infty)$, we assume $\lim_{a \to \infty} r(a)/a \geq
0$. Superlinear functions trivially satisfy this condition; for the Tsallis
entropy, $\lim_{a \to \infty} r_q(a)/a = 1/(q - 1) > 0$ for $1 < q < 2$.
Moreover, we highlight that the condition $\lim_{a \to \infty} r(a)/a \geq 0$
is not implied by the definition of Legendre functions or the assumption $\dom
r = [0,\infty)$. Indeed, the twice-differentiable Legendre function $a \mapsto
-\sqrt{a} - a$ satisfies $\dom r = [0,\infty)$, yet $\lim_{a \to \infty} r(a)/a
= -1$.

In the following lemma, we state several characterizations of the condition
$\lim_{a \to \infty} r(a)/a \geq 0$ for Legendre functions.
\begin{lemma}
\label{lem:weak-superlineraity}
    Let $r : \R \to \ER$ be a Legendre function satisfying $\dom r =
    [0,\infty)$. Then
    \begin{equation*}
        \lim_{a \to \infty} \frac{r(a)}{a} \geq 0
        \iff \lim_{a \to \infty} r'(a) \geq 0
        \iff (-\infty,0) \subset \ran r'.
    \end{equation*}
\end{lemma}

With a slight abuse of notation we extend these concepts to Carath\'{e}odory
functions,
\begin{equation*}
    R \colon \Omega \times \R \to \ER,
\end{equation*}
to capture the geometry of the feasible set $K$. In particular, we assume the
map \(R(x,\cdot)\) is a Legendre function with
\begin{equation*}
    \dom R(x,\cdot)
    = [\phi(x),\infty)
    \quad \text{for a.e.\ } x \in \Omega.
\end{equation*}
In the sequel, all model choices are of the shifted scalar form
\begin{equation}
\label{eq:shifted-legendre-form}
    R(x,y) = r(y - \phi(x)),
    \quad y \geq \phi(x).
\end{equation}

Allowing a further abuse of notation, we use the same symbol \(R\) for the
corresponding superposition operator,
\begin{equation*}
    R(u)(x)
    = R(x,u(x)),
    \quad u \in K.
\end{equation*}
For appropriate $u \in K$, we denote
\begin{equation*}
    R'(u)(x)
    = \partial_u R(x,u(x)).
\end{equation*}
For the shifted form \eqref{eq:shifted-legendre-form}, this reduces to
\begin{equation*}
    R'(u)(x)
    = r'(u(x)-\phi(x)).
\end{equation*}

The superposition operator \(R\) also induces the pointwise Bregman divergence
\begin{equation*}
    D_R(u,v)(x)
    \coloneqq  R(x,u(x)) - R(x,v(x)) - \partial_v R(x,v(x))(u(x) - v(x)),
\end{equation*}
defined for \(u(x)\geq \phi(x)\) and \(v(x)>\phi(x)\). Typically, we suppress the explicit \(x\)-dependence and write
\begin{equation*}
    D_R(u,v)
    = R(u) - R(v) - R'(v)(u - v).
\end{equation*}

\subsection{Bregman proximal point iteration}

We now formally define the Bregman proximal point iteration associated with the
obstacle problem. Let
\begin{equation*}
    K^\circ
    \coloneqq \{u \in K \cap L^\infty(\Omega) \mid \essinf_{\Omega} (u - \phi) > 0\}
\end{equation*}
denote the set of strictly feasible functions. Let \(u^0 \in K^\circ\) be an
initial iterate and let \(\{\alpha_k\}_{k \geq 1}\) be a sequence of positive
proximity parameters. The Bregman proximal point iteration is
\begin{equation}
\label{eq:general-bpp-step}
    u^k
    \in \argmin_{u \in K}
    \left\{
        E(u) + \frac{1}{\alpha_k} \int_\Omega D_R(u,u^{k-1}) \diff x
    \right\},
    \qquad k=1,2,\ldots.
\end{equation}
The Bregman term acts as an adaptive regularization whose geometry is determined
by the Legendre function \(R\). In particular, the singularity of \(R'\) at the
boundary of the feasible set encodes the obstacle constraint in the proximal
subproblem.

As demonstrated below, we may differentiate \eqref{eq:general-bpp-step} along
arbitrary directions \(v \in H_0^1(\Omega)\) if $u^k \in K^\circ$. This gives the formal
optimality condition
\begin{equation}
\label{eq:general-bpp-weak-form-A}
    \alpha_k a(u^k,v)
    +
    (R'(u^k),v)
    =
    \alpha_k F(v)
    +
    (R'(u^{k-1}),v)
    \qquad \forall v\in H_0^1(\Omega).
\end{equation}
When \(F(v)=(f,v)\), the corresponding formal strong form is
\begin{equation}
\label{eq:generalized-entropic-poisson-strong}
    -\alpha_k\nabla\cdot(A\nabla u^k)
    +
    R'(u^k)
    =
    \alpha_k f
    +
    R'(u^{k-1})
    \qquad \text{in }\Omega,
\end{equation}
with boundary condition
\begin{equation*}
    u^k = g
    \quad \text{on } \partial\Omega.
\end{equation*}
We refer to \eqref{eq:generalized-entropic-poisson-strong} as \textit{the
generalized entropic Poisson equation} associated with the proximal step $k$.
Taking the Shannon entropy, \(R(u) = (u-\phi)\log(u-\phi)-(u-\phi)\), and
setting \(A=I\), reduces \eqref{eq:general-bpp-weak-form-A} to the entropic
Poisson equation \eqref{eq:entropic-poisson}.

\section{Main Results}
\label{sec:main-results}

Here, we state the main well-posedness, a priori estimates, and convergence
results.

\subsection{Well-posedness and a priori estimates}

At this stage, equations \eqref{eq:general-bpp-weak-form-A} and
\eqref{eq:generalized-entropic-poisson-strong} should be understood as formal
optimality conditions. Their rigorous validity requires proving that the
minimizer of \eqref{eq:general-bpp-step} is strictly separated from the
obstacle, so that \(R'(u^k)\) is a well-defined \(L^\infty\)-function. This
positivity and regularity theory is developed below.

We first establish the well-posedness of each Bregman proximal subproblem
\eqref{eq:general-bpp-step}, including the strict feasibility needed to justify
its variational formulation \eqref{eq:general-bpp-weak-form-A}. We then record
several a priori properties of the resulting iterates that will provide the
basic structural control for the convergence analysis.

\subsubsection{Existence and uniqueness}

The existence and uniqueness of solutions to \eqref{eq:general-bpp-step} can be
argued using the direct method of the calculus of variations and convexity.
However, rigorously deriving \eqref{eq:general-bpp-weak-form-A} requires
establishing additional regularity of the iterate \(u^{k}\) since
\begin{equation*}
    u \mapsto \int_\Omega D_R(u,u^{k-1}) \diff x
    ,
    \qquad u,\, u^{k-1} \in K^\circ
    ,
\end{equation*}
is generally not continuously Fr\'{e}chet differentiable with respect to the
\(H^1(\Omega)\) norm topology. One of our first main results is that \(u^{k}
\in K^\circ\), which allows us to differentiate the functional in the
\(L^\infty(\Omega) \cap H^1(\Omega)\) norm topology, leading to
\eqref{eq:general-bpp-weak-form-A}.

In the case of the Shannon entropy, the authors in \cite{keith2024proximal} used
a truncation argument to show that \(u^k \in K^\circ\) for all \(k \geq 1\);
cf.\ \cite[Theorem 4.7]{keith2024proximal}. Here, we take a different approach
that can be generalized to the large family of Legendre
functions in \Cref{subsec:leg-func-breg-div}. In particular, we \emph{regularize} the function \(r\) through its
second-order Taylor series expansion. To this end, recall
\eqref{eq:shifted-legendre-form} and define a map \(T \colon \R \times \intr(\dom r)
\to \R\) by
\begin{equation*}
    T(a;b)
    = r(b) + r'(b) (a - b) + \frac{r''(b)}{2} (a - b)^2.
\end{equation*}
This map \(T\) is the second-order Taylor approximation of \(r\) centered at
\(b\in \intr(\dom r) = (0,\infty)\). We now define a regularized entropy
function \(r_N \colon \R \to \R\) for a given \(N > 1\),
\begin{equation*}
    r_N(a) =
    \begin{cases}
        r(a) & \text{if } a \in [1/N,N], \\
        T(a; 1/N) & \text{if } a < 1/N, \\
        T(a; N) & \text{if } a > N,
    \end{cases}
\end{equation*}
and set $R_N(u)(x) \coloneqq r_N(u(x) - \phi(x))$.

In \Cref{fig:reg-shannon}, we provide plots of the regularized Shannon entropy
\(r_N\) and its derivative \(r_N'\).
\begin{figure}
    \centering
    \includegraphics[height=0.35\textheight]{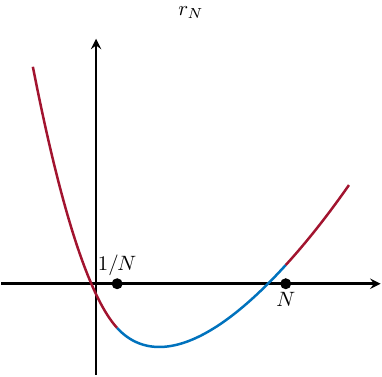}
    \includegraphics[height=0.35\textheight]{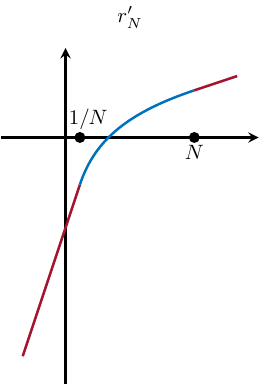}
    \caption{Plots of the regularized Shannon entropy and its derivative for
    given \(N\). The blue region denotes the region where \(r_N = r\) and $r_N'
    = r'$, and the red region denotes the region of regularization.}
    \label{fig:reg-shannon}
\end{figure}

Classical techniques in the calculus of variations can be used to establish the
existence and uniqueness of solutions \(u_N \in H_g^1(\Omega)\) to the regularized form of
\eqref{eq:general-bpp-weak-form-A},
\begin{equation}
    \label{eq:entropic-pde-reg-ve}
        \alpha_k \langle E'(u_N), v\rangle + (R_N'(u_N), v) = (R_N'(u^{k-1}),v) \quad \forall v \in H_0^1(\Omega)
    .
\end{equation}
As the arguments are standard, we omit the proof and refer the reader to
\cite[Chapter~3.4]{Dacorogna2008} for details.

The reason for studying~\eqref{eq:entropic-pde-reg-ve} lies in the fact that
\(\|R_N'(u_N)\|_{L^\infty(\Omega)}\) is bounded uniformly in \(N\), as stated in
the following lemma, which is proven in \Cref{sec:wellposedness-regularity}.

\begin{lemma}
    \label{lem:rn-prime-un-l-infty}
    Let \(\alpha_k > 0\) and \(u^{k-1} \in K^\circ\). Suppose \(f \in
    L^\infty(\Omega)\), \(g \in H^{1/2}(\partial \Omega)\), and \(\dv(A \nabla
    \phi) \in L^\infty(\Omega)\) with \(\esssup_{\partial \Omega} (g - \phi) <
    \infty\) and \(\essinf_{\partial \Omega} (g - \phi) > 0\). Finally, suppose
    $r$ satisfies $\dom r = [0,\infty)$ and $\lim_{a \to \infty} r(a)/a \geq
    0$. If \(u_N \in H_g^1(\Omega)\) solves \eqref{eq:entropic-pde-reg-ve}, then
    \begin{multline*}
        \|R_N'(u_N)\|_{L^\infty(\Omega)}
        \leq \max\bigl\{r_N'(\textstyle\esssup_{\partial \Omega} (g - \phi)),
             -r_N'(\textstyle\essinf_{\partial \Omega} (g - \phi)), \\
             \|\alpha_k f + R'(u^{k-1}) + \alpha_k \dv(A \nabla \phi)\|_{L^\infty(\Omega)}\bigr\}.
    \end{multline*}
    In particular, \(\|R_N'(u_N)\|_{L^\infty(\Omega)}\) is bounded uniformly in \(N\).
\end{lemma}

\Cref{lem:rn-prime-un-l-infty} shows that there exists a critical value of
\(N\) beyond which all solutions of \eqref{eq:entropic-pde-reg-ve} coincide.
Consequently, their common solution also solves
\eqref{eq:general-bpp-weak-form-A}. This leads to the main well-posedness
result, also proven in \Cref{sec:wellposedness-regularity}.

\begin{theorem}
    \label{thm:entropic-pde}
    Under the assumptions of \Cref{lem:rn-prime-un-l-infty}, there exists a
    unique \(u^k \in K^\circ\) that satisfies
    \eqref{eq:general-bpp-weak-form-A}. Moreover, \(u^k\) is the unique
    minimizer of the energy functional
    \begin{equation*}
        K \ni v \mapsto E(v) + \alpha_k^{-1} \int_\Omega D_R(v,u^{k-1}) \diff x.
    \end{equation*}
\end{theorem}

By applying an induction argument, the following corollary of
\Cref{thm:entropic-pde} is immediate.
\begin{corollary}
    Let $u^0 \in K^\circ$ and $\{\alpha_k\}_{k \geq 1}$ be a sequence of
    positive step sizes. Suppose \(f \in L^\infty(\Omega)\), \(g \in
    H^{1/2}(\partial \Omega)\), and \(\dv(A \nabla \phi) \in L^\infty(\Omega)\)
    with \(\esssup_{\partial \Omega} (g - \phi) < \infty\) and
    \(\essinf_{\partial \Omega} (g - \phi) > 0\). Finally, suppose $r$
    satisfies $\dom r = [0,\infty)$ and $\lim_{a \to \infty} r(a)/a \geq
    0$. Then the sequence $\{u^k\}_{k \geq 1}$ defined by
    \eqref{eq:general-bpp-step} is well-defined. Moreover, the iterates $u^k$
    lie in $K^\circ$ and uniquely satisfy \eqref{eq:general-bpp-weak-form-A}.
\end{corollary}

\subsubsection{A priori properties of the iterates}
\label{subsec:apriori-properties}

In this section, we record several properties of the Bregman proximal iterates
that will be used throughout the convergence analysis. Positive invariance and
uniform boundedness from above provide one-sided pointwise control of the
iterates relative to the exact solution, while energy dissipation yields a
uniform \(H^1\)-bound. Together, these estimates provide the order, amplitude,
and Sobolev control needed in both convergence regimes developed below.

For simplicity, from this point onward we restrict to the homogeneous setting
\begin{equation}
\label{eq:homogeneous-setting}
F=0,
\qquad
g=0.
\end{equation}
This homogeneous setting is assumed throughout the a priori and convergence
analysis below.
\begin{remark}[Reduction of nonhomogeneous data]
Let $z\in H_g^1(\Omega)$ solve
\[
    a(z,v)=F(v)
    \qquad \forall v\in H_0^1(\Omega).
\]
Then the change of variables
\[
    \widetilde u=u-z,
    \qquad
    \widetilde\phi=\phi-z
\]
reduces the obstacle problem to the
homogeneous setting. However, the additional assumptions used in \Cref{sec:linear-convergence} must
then be imposed on the transformed obstacle \(\widetilde\phi\).
\end{remark}

The following positive invariance
result is fundamental.

\begin{proposition}[Positive invariance]
\label{prop:main-positive-invariance}
Assume that $u^0\geq u^*$ a.e. in $\Omega$. Then
\begin{equation}
    u^k
    \geq
    u^*
    \qquad
    \text{a.e. in }\Omega,
    \qquad
    \forall k\geq0.
\end{equation}
\end{proposition}

Positive invariance complements the strict feasibility established in
\Cref{thm:entropic-pde}. Indeed, strict feasibility ensures that $u^k-\phi>0$,
whereas \Cref{prop:main-positive-invariance} provides the additional order
relation
\begin{equation}
u^*-\phi\leq u^k-\phi
\qquad
\text{a.e. in }\Omega.
\end{equation}

\begin{remark}[Constructive initial iterate]
    The condition \(u^0\geq u^*\) can be enforced without prior knowledge
    of \(u^*\). A sufficient and directly verifiable construction based on
    supersolutions is provided in \Cref{app:ordered-initialization}.
\end{remark}

The iterates are also uniformly bounded from above as stated in
\Cref{lem:uniform-bound-above}.

\begin{lemma}[Uniform boundedness from above]
\label{lem:uniform-bound-above}
Define
\begin{equation}
\overline
    M_1
    \coloneqq
    \max
    \left\{
        \esssup_\Omega u^0,\,
        \esssup_\Omega \phi,\,
        0
    \right\}
    +1
\end{equation}
and
\begin{equation}
    M_1
    \coloneqq
    \overline M_1-\essinf_\Omega\phi.
\end{equation}
Then $\{u^k\}_{k\geq0}$ is uniformly bounded from above:
\begin{equation}
    u^k
    \leq
    \overline M_1
    \qquad
    \text{a.e. in }\Omega,
    \qquad
    k\geq0.
\end{equation}
Consequently,
\begin{equation}
\label{eq:M1-unfiorm-bound}
    u^k-\phi
    \leq
    M_1
    \qquad
    \text{a.e. in }\Omega,
    \qquad
    k\geq0.
\end{equation}
\end{lemma}

Combining
\Cref{prop:main-positive-invariance} and \Cref{lem:uniform-bound-above} gives
\begin{equation}
u^*-\phi\in[0, M_1],
\qquad
u^k-\phi\in(0, M_1],
\qquad
u^*-\phi\leq u^k-\phi
\qquad
\text{a.e. in }\Omega.
\end{equation}

Moreover, the energy is nonincreasing along the sequence generated by the Bregman
proximal point iteration. We record this basic dissipation property in the
following lemma.

\begin{lemma}[Energy dissipation]
\label{lem:energy-dissipation}
The following energy dissipation law holds for all \(k \geq 0\):
\begin{equation}
\label{eq:energy-dissipation}
    E(u^{k+1}) \leq E(u^k).
\end{equation}
\end{lemma}

Combining \Cref{lem:energy-dissipation} with uniform ellipticity
\eqref{eq:uniform-ellipticity} gives the following uniform $H^1$ bound.

\begin{lemma}[Uniform \(H^1\)-boundedness]
\label{lem:uniform-H1-bound}
$\{u^k\}_{k\geq0}$ is
uniformly bounded in \(H^1(\Omega)\). More precisely,
\begin{equation*}
    \|u^k\|_{H^1(\Omega)}
    \leq M_2
    \qquad
    \forall k \geq0,
\end{equation*}
where
\begin{equation*}
    M_2
    \coloneqq
    \Lambda
    \left(
        C_P^2+1
    \right)^{1/2}
    \|u^0\|_{H^1(\Omega)}.
\end{equation*}
\end{lemma}

\subsection{Convergence rates}
\label{subsec:convergence-rates}

With the preceding a priori control of the iterates in place, we now state the
convergence rate results for the Bregman proximal point iteration
\eqref{eq:general-bpp-step}. We continue to work under the homogeneous setting
\eqref{eq:homogeneous-setting}. The proofs are given in
\Cref{sec:convergence-rate}. We emphasize that both convergence regimes are
dimension-independent, although the linear convergence result requires
additional geometric assumptions on the obstacle and free boundary.

The convergence theory developed below is organized around sequential strict local minimality. Once an admissible order $s$ has been established, its value
determines the resulting convergence rate. The first part uses the one-sided
growth of the Bregman divergence to obtain an admissible order and improve the
sublinear rates for a broad class of Legendre functions including each Legendre
function listed in \Cref{tab:four-legendre-functions}. The second part uses the
additional assumptions on the free boundary and obstacle, together with the
refined estimates for the Shannon and Spence entropies, to obtain the order $s
= 1$ and hence linear convergence, thereby establishing the improvement
anticipated in \cite[Remark 4.18]{keith2024proximal}.

\subsubsection{The convergence mechanism}
The key ingredient of our convergence mechanism is the strength with which the
energy gap controls the Bregman divergence from the energy minimizer. In finite
dimensions, a minimizer $x^* \in \mathbb{R}^n$ is called a strict local
minimizer of order $m$ if, in a neighborhood of $x^*$, the objective gap
controls the $m$-th power of the Euclidean distance to $x^*$; see
\cite[Definition 1.1(b)]{ward1994characterizations}. Related conditions have
appeared in the analysis of finite-dimensional Bregman-type optimization
methods; see, for instance, \cite[Lemma~3.4]{bauschke2019linear} and
\cite[Assumption 7]{zhang2021proximal}. This motivates the following
definition.

\begin{definition}[Sequential strict local minimality of order $s$]
We say that the energy $E$ satisfies \emph{sequential strict local minimality of order
$s\geq1$} at $u^*$ along the iteration $\{u^k\}$ if there exists a constant
$\sigma>0$, independent of $k$, such that
\begin{equation}
\label{eq:energy-bregman-inequality}
    E(u^k)-E(u^*)
    \geq
    \sigma \left( \int_{\Omega} D_R(u^*,u^k)\diff x \right)^s
    \qquad
    \forall k\geq0.
\end{equation}
We call such an $s$ an \emph{admissible order of minimality}.
\end{definition}

\begin{remark}[Interpretation of the terminology]
    The preceding definition is different from classical strict local minimality in two respects. First, the error is
    measured by the Bregman divergence rather than by a Euclidean norm. Second,
    the condition is imposed only along the sequence $\{u^k\}$, rather than at
    every feasible point in a neighborhood of $u^*$. The latter feature motivates the adjective \emph{sequential}.
\end{remark}

For brevity, we denote the Bregman error by
\begin{equation*}
    B_k
    \coloneqq \int_\Omega D_R(u^*,u^k) \diff x.
\end{equation*}
Each admissible order of minimality produces a corresponding convergence
rate, as stated in \Cref{thm:energy-bregman-convergence}. If $s=1$ is
admissible, the iteration converges linearly. If only an admissible order $s>1$
is available, the same mechanism yields a corresponding sublinear convergence
rate. To convert such an estimate into a convergence rate, we combine
\eqref{eq:energy-bregman-inequality} with the following one-step Bregman
descent property.

\begin{lemma}[One-step Bregman descent]
\label{lem:one-step-bregman-descent}
For every $k\geq1$, the Bregman error satisfies
\begin{equation}
\label{eq:one-step-descent}
    B_{k-1}-B_k
    \geq
    \alpha_k\bigl(E(u^k)-E(u^*)\bigr).
\end{equation}
\end{lemma}

Since $E(u^k) \geq E(u^*)$, \Cref{lem:one-step-bregman-descent} immediately
gives $B_{k-1} \geq B_k$, i.e., $B_{k}$ is nonincreasing. Thus
\Cref{lem:energy-dissipation} and \Cref{lem:one-step-bregman-descent}
establish, respectively, that the functionals $E(u)$ and $\int_{\Omega}
D_R(u^*,u)\,\mathrm{d}x$ serve as Lyapunov functionals for the Bregman proximal
point iteration. These monotonicity properties are fundamental ingredients in
the global convergence analysis. Their finite-dimensional counterparts were
established in \cite[Lemma~3.3]{chen1993convergence}.

If \eqref{eq:energy-bregman-inequality} holds with order $s$, then
\eqref{eq:one-step-descent} gives the scalar recursion
\begin{equation}
    B_{k-1}-B_k
    \geq
    \alpha_k\sigma B_k^s.
\end{equation}
The scalar recursion for $B_k$ yields the following convergence theorem.

\begin{theorem}[Bregman error convergence induced by order-$s$ minimality]
\label{thm:energy-bregman-convergence}
Let $s\geq1$ be an admissible order of minimality, with corresponding constant $\sigma>0$.

If $s>1$, then
\begin{equation}
\label{eq:abstract-bregman-rate}
    B_k
    \leq
    \left(
        B_0^{1-s}
        +
        (s-1)
        \sum_{i=1}^k
        \frac{\alpha_i\sigma}
        {
            \left(
                1+\alpha_i\sigma B_0^{s-1}
            \right)^s
        }
    \right)^{-1/(s-1)}.
\end{equation}
In particular, if $\alpha_k\equiv\alpha>0$, then
\begin{equation}
\label{eq:abstract-bregman-fixed-rate}
    B_k
    =
    \mathcal O\left(k^{-1/(s-1)}\right).
\end{equation}

If $s=1$, then
\begin{equation}
\label{eq:abstract-bregman-linear-product}
    B_k
    \leq
    \prod_{i=1}^k
    \frac{1}{1+\alpha_i\sigma}
    B_0.
\end{equation}
In particular, if $\alpha_k\equiv\alpha>0$, then
\begin{equation}
\label{eq:abstract-bregman-linear-rate}
    B_k
    \leq
    \left(1+\alpha\sigma\right)^{-k}B_0.
\end{equation}
\end{theorem}

\begin{corollary}[Energy-gap and $H^1$-error convergence induced by order-$s$ minimality]
\label{cor:energy-gap-H1-error-convergence}
Assume the hypotheses of
\Cref{thm:energy-bregman-convergence}.
Then, for every $0\leq m<k$,
\begin{equation}
E(u^k)-E(u^*)
\leq
\frac{B_m}
{\displaystyle\sum_{j=m+1}^{k}\alpha_j}.
\end{equation}
Moreover,
\begin{equation}
\|u^k-u^*\|_{H^1(\Omega)}^2
\leq
2\Lambda(C_P^2+1)(E(u^k)-E(u^*)).
\end{equation}

In particular, suppose that $\alpha_k\equiv\alpha>0$. If $s>1$,
then
\begin{equation}
E(u^k)-E(u^*)
=
\mathcal O\left(
k^{-s/(s-1)}
\right),
\end{equation}
and
\begin{equation}
\|u^k-u^*\|_{H^1(\Omega)}
=
\mathcal O\left(
k^{-s/(2(s-1))}
\right).
\end{equation}
If $s=1$, then
\begin{equation}
E(u^k)-E(u^*)
=
\mathcal O((1+\alpha\sigma)^{-k}),
\end{equation}
and
\begin{equation}
\|u^k-u^*\|_{H^1(\Omega)}
=
\mathcal O((1+\alpha\sigma)^{-k/2}).
\end{equation}
\end{corollary}

It remains to identify the range of minimality orders. We do so
in two regimes. A one-sided growth condition on the Bregman divergence yields
an exact expression for admissible $s \geq 2$,
whereas direct assumptions on the obstacle and free boundary deliver $s=1$
for the Shannon and Spence entropies.

\subsubsection{Sequential strict local minimality of order \texorpdfstring{$s \geq 2$}{s >= 2}}

We first establish sequential strict local minimality of order $s > 1$ using only the
pointwise growth of the Bregman divergence.

\begin{definition}[One-sided Bregman divergence growth exponent]
\label{def:one-sided-bregman-growth}
Let $M>0$. We say that the Legendre function $R(u)=r(u-\phi)$ has one-sided
Bregman growth exponent $\theta\in(0,1]$ on $[0,M]$ if there exists a constant
$C_M>0$ such that
\begin{equation}
\label{eq:one-sided-bregman-growth}
    D_r(a,b)
    \leq
    C_M(b-a)^\theta
    \qquad
    \forall a\in[0,M],\quad b\in(0,M]
    \quad\text{with }a\leq b.
\end{equation}
\end{definition}

The growth exponents for the four model Legendre functions in
\Cref{sec:problem-setting} are listed in the second column of \Cref{tab:convergence-rates-model-legendre}. Their
verification is deferred to \Cref{app:one-sided-growth-exponents}.

\begin{remark}[Relation to the Legendre exponent]
\Cref{def:one-sided-bregman-growth} is related to the Legendre exponent in
\cite[Definition 4.1]{azizian2024rate}. In both settings, the relevant exponent
quantifies the local growth of the Bregman divergence. The difference is
that, in the present infinite-dimensional setting, we impose the one-sided
bounded restriction,
\begin{equation*}
    a\in[0,M],\quad b\in(0,M]
    \quad\text{and }a\leq b.
\end{equation*}
This restriction is adapted to the positive invariance of the iterates and
is what makes the resulting estimate uniform in $k$.
\end{remark}

\begin{remark}[Dependence of the growth constant]
The growth constants for the Shannon and Tsallis entropies are independent of
$M$. In particular, the optimal constant for the Shannon entropy is $C_M=1$,
while the optimal constant for the Tsallis entropy is
$C_M=1/(2-q)$. For the Spence and Kaniadakis entropies, the exponent $\theta$
remains valid on every bounded interval $[0,M]$, but the corresponding constant
$C_M$ generally depends on $M$.
\end{remark}

By \Cref{prop:main-positive-invariance} and
\Cref{lem:uniform-bound-above}, \Cref{def:one-sided-bregman-growth} can be
applied with
\begin{equation}
    a
    =
    u^*-\phi,
    \qquad
    b
    =
    u^k-\phi,
    \qquad
    M
    =
    M_1.
\end{equation}
This gives the first estimate of sequential strict local minimality of order $s = 2/\theta \geq 2$.

\begin{proposition}[Sequential strict local minimality of order $s \geq 2$]
\label{prop:main-energy-bregman}
Assume that $u^0\geq u^*$ a.e. in $\Omega$, and let $M_1$ be the
constant from \Cref{lem:uniform-bound-above}. Suppose that $R$ satisfies
\eqref{eq:one-sided-bregman-growth} with exponent $\theta\in(0,1]$. Then $2/\theta$ is an admissible minimality order. More precisely, there exists a constant $\sigma>0$, depending only on
$\Omega$, $A$, $C_{M_1}$, and $\theta$, such that
\begin{equation*}
    E(u^k)-E(u^*)
    \geq
    \sigma
    \left(
        \int_\Omega D_R(u^*,u^k)\diff x
    \right)^{2/\theta}
    \qquad
    \forall k\geq0.
\end{equation*}
\end{proposition}

The preceding proposition provides a general minimality order that can be verified solely from the one-sided
Bregman growth condition. Combining \Cref{prop:main-energy-bregman} with
\Cref{thm:energy-bregman-convergence} and \Cref{cor:energy-gap-H1-error-convergence} gives the following sublinear
convergence rates under constant step sizes.

These rates should be understood as worst-case algebraic guarantees under
the general hypotheses of \Cref{prop:main-energy-bregman}. In particular,
the analysis uses only the positive invariance property and one-sided Bregman growth condition and does not
exploit finer properties of the solution. Consequently, it does not preclude faster convergence for particular problem instances.

\begin{theorem}[Sublinear convergence rates]
\label{thm:main-sublinear-convergence}
Assume that the hypotheses of
\Cref{prop:main-energy-bregman} hold and that
$\alpha_k = \alpha>0$. Then
\begin{equation}
\label{eq:main-Ak-rate}
    B_k
    =
    \mathcal O\left(
        k^{-\theta/(2-\theta)}
    \right).
\end{equation}
Moreover,
\begin{equation}
\label{eq:main-energy-rate}
    E(u^k)-E(u^*)
    =
    \mathcal O\left(
        k^{-2/(2-\theta)}
    \right),
\end{equation}
and
\begin{equation}
\label{eq:main-H1-rate}
    \|u^k-u^*\|_{H^1(\Omega)}
    =
    \mathcal O\left(
        k^{-1/(2-\theta)}
    \right).
\end{equation}
\end{theorem}

\Cref{tab:convergence-rates-model-legendre} summarizes the fixed-step sublinear
convergence rates for the four model Legendre functions.

\begin{table}[htbp]
\centering
\renewcommand{\arraystretch}{1.35}
\footnotesize
\begin{tabular}{lccc}
\toprule
Legendre function
&
Growth exponent $\theta$
&
$\displaystyle \int_\Omega D_R(u^*,u^k)\diff x$
&
$\displaystyle \|u^k-u^*\|_{H^1(\Omega)}$
\\
\midrule
Shannon entropy
&
1
&
$\mathcal O(k^{-1})$
&
$\mathcal O(k^{-1})$
\\
Spence entropy
&
1
&
$\mathcal O(k^{-1})$
&
$\mathcal O(k^{-1})$
\\
Tsallis entropy, $1<q<2$
&
$2 - q$
&
$\displaystyle \mathcal O\left(k^{-(2-q)/q}\right)$
&
$\displaystyle \mathcal O\left(k^{-1/q}\right)$
\\
Kaniadakis entropy, $0<\kappa<1$
&
$1 - \kappa$
&
$\displaystyle \mathcal O\left(k^{-(1-\kappa)/(1+\kappa)}\right)$
&
$\displaystyle \mathcal O\left(k^{-1/(1+\kappa)}\right)$
\\
\bottomrule
\end{tabular}
\caption{One-sided Bregman growth exponents and the resulting fixed-step sublinear convergence rates for the model Legendre functions.}
\label{tab:convergence-rates-model-legendre}
\end{table}

\begin{remark}[Sharpness]
    For the Shannon and Tsallis entropies, the sublinear convergence rates in
    \Cref{tab:convergence-rates-model-legendre} are sharp for the families of
    examples constructed in \Cref{app:weak-sharpness-models}.
    Indeed, the boundary-degenerate endpoint admits explicit proximal point
    iterations with exactly the sublinear decay rates predicted by
    \Cref{thm:main-sublinear-convergence}, while the stability results transfer
    these lower bounds to the approximating families with $0<\delta\leq1$.
    Consequently, because the constructed family is a subclass of the
    problems considered in this section, the corresponding sublinear rates
    cannot, in general, be improved.
\end{remark}

\subsubsection{Sequential strict local minimality of order \texorpdfstring{$s=1$}{s = 1}}
\label{sec:linear-convergence}
We now turn from the sublinear convergence theory to the linear convergence
result. For the Shannon and Spence entropies, the one-sided growth condition
has $\theta=1$. Thus \Cref{prop:main-energy-bregman} only shows that $s=2$ is
an admissible minimality order. This estimate yields the
algebraic rates in \Cref{tab:convergence-rates-model-legendre}, but it does not
by itself imply linear convergence.

We use additional information about the exact solution $u^*$ and its free boundary to
improve the minimality order from $s = 2$ to $s
= 1$. The decisive ingredients are the refined Bregman bound for the Shannon
and Spence entropies, and some additional assumptions on the free boundary and
obstacle.

Recall the definition of the active set \(\Omega_0\), the inactive set
\(\Omega_+\) and the free boundary \(\Gamma\),
\begin{equation*}
    \Omega_0 \coloneqq  \{x\in\Omega \mid u^*(x)=\phi(x)\},
    \quad
    \Omega_+ \coloneqq  \Omega \setminus \Omega_0,
    \quad
    \Gamma\coloneqq \partial\Omega_0\cap\Omega.
\end{equation*}
We assume that the multiplier \(E'(u^*)\in H_0^1(\Omega)'\) can be represented by a
function \(\lambda^*\in L^\infty(\Omega)\), namely
\begin{equation}
\label{eq:E-prime-representation}
    \langle E'(u^*),v\rangle
    =
    \int_\Omega \lambda^* v\diff x
    \qquad \forall v\in H_0^1(\Omega).
\end{equation}
As throughout this subsection, we retain the homogeneous setting
\eqref{eq:homogeneous-setting}. We impose some additional assumptions on
the free boundary and obstacle below.

Motivated by \cite[Theorem~2.5]{aleksanyan2024quantitative}, we use the
following notion.
\begin{definition}[Regular free-boundary point]
\label{def:regular-free-boundary-point}
Let $w^* \coloneqq u^*-\phi$,
and set
\begin{equation*}
    q_\phi \coloneqq -\nabla\cdot(A\nabla\phi).
\end{equation*}
A point \(y\in\Gamma\) is called a regular free-boundary point if, up to a sequence of radii,
\begin{equation}
    \frac{w^*(y+rx)}{r^2}
    \to
    \frac{q_\phi(y)}
         {2\,\nu_y^{\mathsf T}A(y)\nu_y}
    \bigl[(x\cdot\nu_y)_+\bigr]^2 \quad\text{in }C^1_{\mathrm{loc}}(\mathbb R^d), \text{as } r\to0,
\end{equation}
for some unit vector $\nu_y\in\mathbb S^{d-1}$, where
$(t)_+\coloneqq\max\{t,0\}$ denotes the positive part of \(t\).
\end{definition}

The distinction between regular and singular free-boundary points is
illustrated schematically in \Cref{fig:regular-singular-free-boundary}. The
configuration on the left depicts a free boundary consisting entirely of
regular points, whereas the configuration on the right contains singular
free-boundary points. For further illustrations of regular and singular
free-boundary points, we refer the interested reader to \cite[Figures~7.2
and~7.3]{figalli2018free}.

For the linear convergence analysis, we restrict our attention to free
boundaries consisting entirely of regular points and impose the following
additional regularity and geometric assumptions.

\begin{assumption}
\label{ass:strict-complementarity-free-boundary}
We assume the following conditions:
\begin{itemize}
    \item The domain $\Omega\subset\mathbb R^d$ is bounded and of class $C^{1,\alpha}$.

    \item The obstacle satisfies \(\phi \in C^{2,\alpha}(\overline
    \Omega)\), and there exists \(c_{\phi} >0\) such that
    \begin{equation}
    \label{eq:strong-super-harmonicity}
        -\nabla \cdot(A(x) \nabla \phi) \geq c_{\phi} \qquad\text{in }\Omega.
    \end{equation}

    \item The free boundary \(\Gamma\) is a compact embedded
    \(C^{2,\alpha}\) submanifold of dimension \(d-1\). Every point \(y\in\Gamma\) is a
    regular free-boundary point; see \Cref{def:regular-free-boundary-point}.
\end{itemize}
\end{assumption}

\begin{remark}[On the regularity of free boundary]
The \(C^{2,\alpha}\)-regularity of \(\Gamma\) is imposed here as an additional
assumption. This assumption is inspired by the classical regularity theory for
the obstacle problem; see \cite{caffarelli1977regularity} and
\cite{figalli2018free}. In particular, when \(A=I\) and the obstacle is
sufficiently smooth, the regular part of the free boundary enjoys higher-order
regularity; see \cite[Theorem 1]{kinderlehrer1977regularity}. Moreover, for
variable-coefficient obstacle problems with \(A\in C^{0,\alpha}\), the regular
part of the free boundary is of class \(C^{1,\beta}\) for some
\(\beta\in(0,1)\); see \cite[Theorem 1.1]{andreucci2023classical}. Thus, the
assumed \(C^{2,\alpha}\)-regularity is attainable in the classical smooth-data
setting and is compatible with the known regularity theory for
variable-coefficient problems. We do not derive a higher-order regularity result
for the variable-coefficient problem considered here. This additional regularity
is required below for the properties given in \Cref{prop:strict-complementarity-free-boundary}.
\end{remark}

\begin{remark}[Geometry near regular free-boundary points]
\label{rem:geometry-near-regular-points}
Since every point of $\Gamma$ is a regular free-boundary point, the
local regularity theory
\cite[Theorem 4.12, in particular (4.25)]{focardi2015calculus}
implies that the active set $\Omega_0$ and the inactive
set $\Omega_+$ lie locally on opposite sides of the free boundary.
In particular, for every $y\in\Gamma$, one
side of a sufficiently small neighborhood of $y$ is contained in
$\operatorname{int}(\Omega_0)$. Consequently, if
$\Gamma\neq\varnothing$, then
\[
    \operatorname{int}(\Omega_0)\neq\varnothing
    \qquad\text{and}\qquad
    |\Omega_0|>0.
\]
This geometric structure will be used in the proofs of
\Cref{prop:strict-complementarity-free-boundary} and
\Cref{prop:main-local-bregman-coercivity} below.
\end{remark}

\begin{proposition}
\label{prop:strict-complementarity-free-boundary}
Let $w^* = u^* - \phi$.
Under
\Cref{ass:strict-complementarity-free-boundary}, the following properties hold:
\begin{itemize}
    \item Strict complementarity:
    \begin{equation}
    \label{eq:strict-complementarity-density}
        \lambda^* \geq c_{\phi} > 0, \quad \text{a.e. in }\; \Omega_0.
    \end{equation}

    \medskip

    \item One-sided tubular parametrization. There exist \(\rho>0\) and a \(C^{1,\alpha}\)-diffeomorphism
    \begin{equation}
    \label{eq:ass-tubular-parametrization}
        \Psi:\Gamma\times(0,\rho)\to \mathcal N_\rho^+(\Gamma),
    \end{equation}
     with
    \begin{equation*}
        \Psi(y,s)=y+s\nu(y),
    \end{equation*}
    where \(\nu\) is the unit normal pointing into \(\Omega_+\) and
    \begin{equation*}
        \mathcal N_\rho^+(\Gamma) = \Omega_+ \cap \{\dist(x,\Gamma) < \rho\}.
    \end{equation*} Moreover,
    \begin{equation}
    \label{eq:tubular-jacobian-comparable}
        0<C_J^{-1}\leq |J_\Psi(y,s)|\leq C_J<\infty
        \qquad\text{for a.e. }(y,s)\in\Gamma\times(0,\rho).
    \end{equation}
    Consequently, for every \(e\in H^1(\Omega)\),
    \begin{equation}
    \label{eq:composition-regularity}
        e\circ\Psi\in L^2(\Gamma;H^1(0,\rho))
    \end{equation}
    and
    \begin{equation}
    \label{eq:gradient-bound}
        |\partial_s(e\circ\Psi)(y,s)|
        \leq |\nabla e|(\Psi(y,s))
        \qquad\text{for a.e. }(y,s).
    \end{equation}

    \item Quadratic separation near the free boundary. The function \(w^*\)
    separates quadratically from zero near \(\Gamma\). More precisely, there
    exists a constant \(0<c_1<\infty\) such that
    \begin{equation*}
        c_1^{-1}\,\dist(x,\Gamma)^2
        \leq
        w^*(x)
        \leq
        c_1\,\dist(x,\Gamma)^2
    \end{equation*}
    for a.e. \(x\in \mathcal{N}^+_\rho(\Gamma)\). Equivalently, in the
    coordinates \eqref{eq:ass-tubular-parametrization}, it holds that
    \begin{equation}
    \label{eq:quadratic-separation-parametrized}
        c_1^{-1} s^2
        \leq
        w^*(\Psi(y,s))
        \leq
        c_1 s^2
        \qquad
        \text{for a.e. }(y,s)\in\Gamma\times(0,\rho).
    \end{equation}

    \medskip

    \item Separation away from the free boundary. The function \(w^*\) is
    uniformly separated from zero away from \(\Gamma\): for every
    \(\rho_0\in(0,\rho)\), there exists \(c_{\rho_0}>0\) such that
    \begin{equation}
    \label{eq:separation-away-free-boundary}
        w^*(x)\geq c_{\rho_0}
        \qquad
        \text{for a.e. }x\in
        \Omega_+\cap\{\dist(x,\Gamma)\geq \rho_0\}.
    \end{equation}
\end{itemize}
\end{proposition}

The four properties stated in \Cref{prop:strict-complementarity-free-boundary} are illustrated in \Cref{fig:free-boundary-properties}. The proof of \Cref{prop:strict-complementarity-free-boundary} is deferred to
\Cref{app:strict-complementarity-free-boundary}, since it is independent of the
convergence mechanism.

\begin{figure}[htbp]
    \centering

    \begin{subfigure}[t]{0.48\textwidth}
        \centering
        \includegraphics[width=0.8\textwidth]{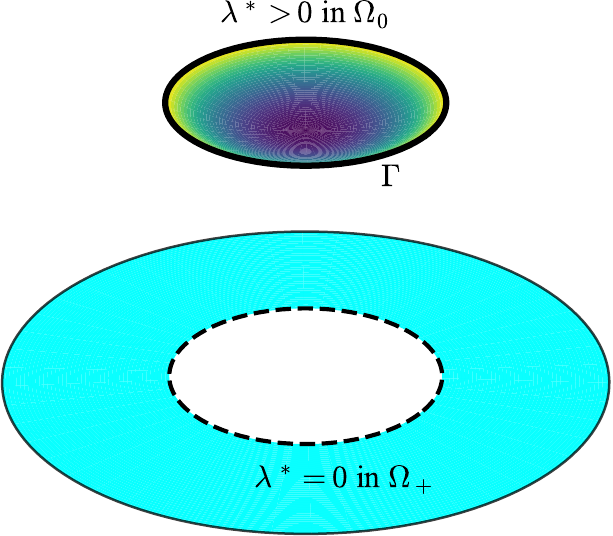}
        \caption{Strict complementarity.}
    \end{subfigure}
    \hfill
    \begin{subfigure}[t]{0.48\textwidth}
        \centering
        \raisebox{1.0cm}
        {\includegraphics[width=\textwidth]{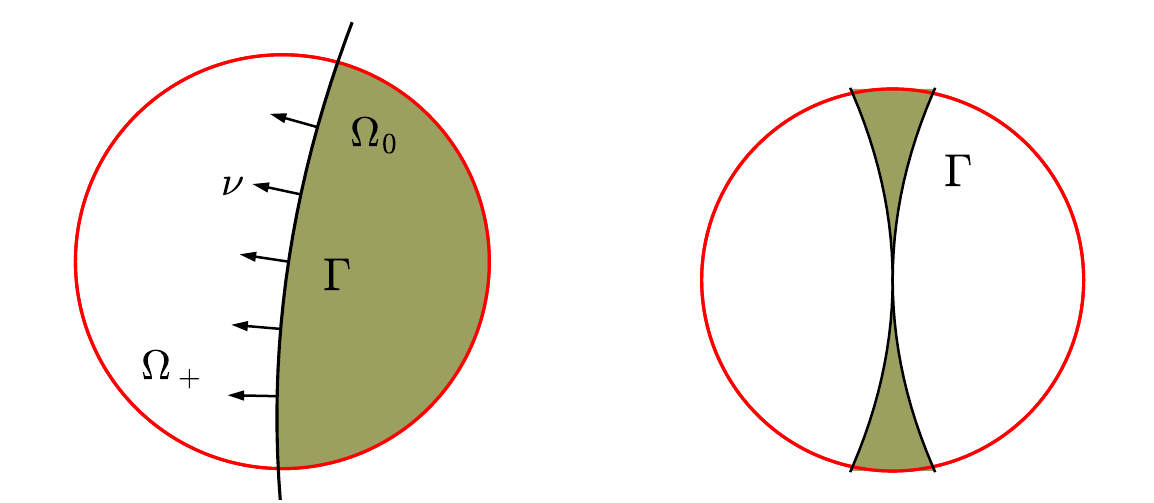}}
        \caption{One-sided tubular parametrization.}
        \label{fig:regular-singular-free-boundary}
    \end{subfigure}

    \vspace{0.5em}

    \begin{subfigure}[t]{0.48\textwidth}
        \centering
        \includegraphics[width=0.9\textwidth]{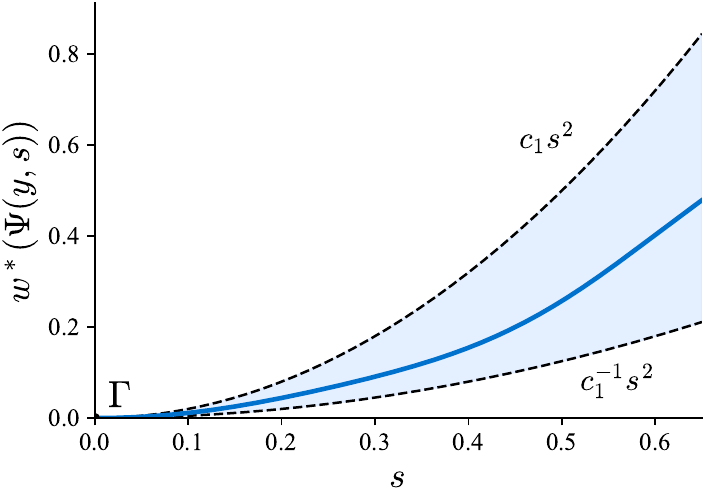}
        \caption{Quadratic separation near free boundary.}
    \end{subfigure}
    \hfill
    \begin{subfigure}[t]{0.48\textwidth}
        \centering
        \includegraphics[width=0.9\textwidth]{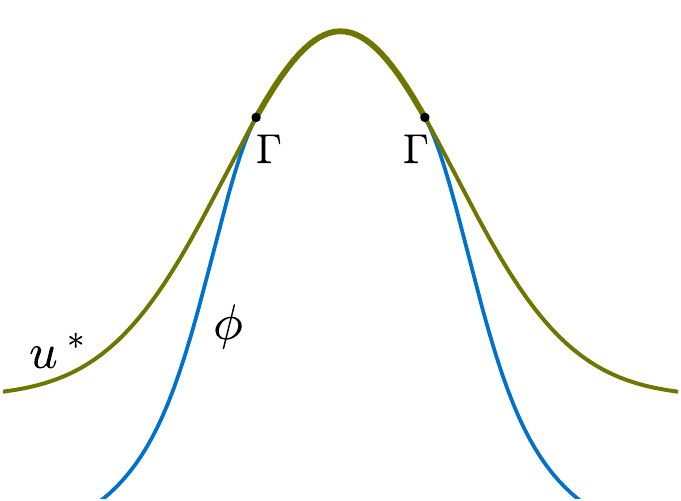}
        \caption{Separation away from free boundary.}
    \end{subfigure}

    \caption{Illustration of the geometric properties near the free boundary in
    \Cref{prop:strict-complementarity-free-boundary}. (a) Strict
    complementarity: \(\lambda^*\) is strictly positive on \(\Omega_0\) and
    vanishes on \(\Omega_+\). (b) One-sided tubular parametrization: near
    \(\Gamma\), points in \(\Omega_+\) can be represented as
    \(x=\Psi(y,s)=y+s\nu(y)\), where \(y\in\Gamma\) and \(s>0\). The left figure
    satisfies this assumption, whereas the right figure depicts a singular
    geometry excluded by the assumption. (c) Quadratic separation: along the
    normal direction into \(\Omega_+\), the gap \(w^*\) locally satisfies
    \(w^*(\Psi(y,s))\sim s^2\). (d) Separation away from the free boundary: the gap $w^*$ is uniformly positive on the portion of the inactive set lying a fixed positive distance away from \(\Gamma\).}
    \label{fig:free-boundary-properties}
\end{figure}

The geometric properties in \Cref{prop:strict-complementarity-free-boundary}
allow the energy gap to control the Bregman error linearly on bounded subsets
lying above $u^*$. Strict complementarity controls the error on the active set.
The tubular parametrization and quadratic separation control the transition
across the free boundary, while separation away from the free boundary prevents
additional degeneracy in the inactive region.

\begin{proposition}[Sequential strict local minimality of order $s = 1$]
\label{prop:main-local-bregman-coercivity}
Assume that \Cref{ass:strict-complementarity-free-boundary} holds and that
$u^0\geq u^*$ a.e. in $\Omega$. Then, for the Shannon and Spence entropies,
$s = 1$ is an admissible minimality order. More precisely, there
exists a constant $\sigma>0$, depending on the problem data and on the
constants $M_1$ and $M_2$ in \Cref{lem:uniform-bound-above} and
\Cref{lem:uniform-H1-bound}, respectively, but independent of $k$, such
that
\begin{equation}
\label{eq:main-linear-energy-bregman}
    E(u^k)-E(u^*)
    \geq \sigma \int_\Omega D_R(u^*,u^k)\diff x
    \qquad \forall k\geq 0.
\end{equation}
\end{proposition}

Combining \Cref{prop:main-local-bregman-coercivity} with
\Cref{thm:energy-bregman-convergence} and the uniform ellipticity of the
quadratic energy gives the following linear convergence rates under constant
step sizes.

\begin{theorem}[Linear convergence rates]
\label{thm:main-linear-convergence}
Assume that the hypotheses of
\Cref{prop:main-local-bregman-coercivity} hold and that
$\alpha_k=\alpha>0$. Let
$\sigma>0$ be the constant in
\eqref{eq:main-linear-energy-bregman}, and set
\begin{equation*}
    \varrho
    \coloneqq
    \frac{1}
    {1+\alpha\sigma}
    \in(0,1).
\end{equation*}
Then
\begin{equation}
\label{eq:main-linear-Ak}
    B_k
    \leq
    \varrho^k B_0
    \qquad
    \forall k\geq0.
\end{equation}
Moreover,
\begin{equation}
\label{eq:main-linear-energy}
    E(u^k)-E(u^*)
    =
    \mathcal O\left(\varrho^k\right),
\end{equation}
and
\begin{equation}
\label{eq:main-linear-H1}
    \|u^k-u^*\|_{H^1(\Omega)}
    =
    \mathcal O\left(\varrho^{k/2}\right).
\end{equation}
\end{theorem}

The linear convergence \eqref{eq:main-linear-H1} should be interpreted as a structural improvement over
the generic algebraic rate \eqref{eq:main-H1-rate}. In the absence of strict complementarity, biactive
behavior may slow the convergence and obstruct minimality order
$s = 1$ required to establish a linear convergence rate; we refer interested
readers to \cite[Definition 12.5]{nocedal2006numerical} and
\cite{keith2024proximal}. Strict complementarity excludes this degeneracy.

\section{Proofs of well-posedness and a priori estimates}
\label{sec:wellposedness-regularity}

\subsection{Proof of existence and uniqueness}
The proof of \Cref{lem:rn-prime-un-l-infty} is based on Stampacchia's truncation method \cite[Theorem 8.19]{brezis2011functional}. Before we state its proof, we require the following result, which extends \cite[Proposition A.7]{keith2024proximal}.

\begin{lemma}
    \label{lem:preimage-l-infty}
    Let $r : \R \to \ER$ be a Legendre function satisfying $\dom r = [0,\infty)$ and $\lim_{a \to \infty} r(a)/a \geq 0$. Define the preimage
    \begin{equation*}
        (r')^{-1}(L^\infty(\Omega))
        = \{w \in L^\infty(\Omega) \mid r' \circ w \in L^\infty(\Omega)\}.
    \end{equation*}
    Then
    \begin{align*}
        (r')^{-1}(L^\infty(\Omega))
        &= (r^*)'(L^\infty(\Omega)) \\
        &= \{w \in L^\infty(\Omega) \mid 1/w \in L^\infty(\Omega) \text{ and } w > 0\} \\
        &= \{w \in L^\infty(\Omega) \mid \textstyle \essinf_\Omega w > 0\}.
    \end{align*}
\end{lemma}
\begin{proof}
    The equality $(r')^{-1}(L^\infty(\Omega)) = (r^*)'(L^\infty(\Omega))$ is immediate by \eqref{eq:legendre-duality}. Thus, we will show
    \begin{align*}
        (r')^{-1}(L^\infty(\Omega))
        &\subseteq \{w \in L^\infty(\Omega) \mid 1/w \in L^\infty(\Omega) \text{ and } w > 0\} \\
        &\subseteq \{w \in L^\infty(\Omega) \mid \textstyle \essinf_\Omega w > 0\} \\
        &\subseteq (r')^{-1}(L^\infty(\Omega)).
    \end{align*}

    First, take $w \in (r')^{-1}(L^\infty(\Omega))$. Then $r'(w) \in L^\infty(\Omega)$, so $-L \leq r'(w)$ a.e.\ in $\Omega$ for some $L > 0$. Moreover, \Cref{lem:weak-superlineraity} asserts that $(-\infty,0) \subset \ran r'$, which in turn implies that $(r')^{-1}(-L)$ is well-defined. As $\dom r = [0,\infty)$, it follows that $\ran (r')^{-1} = \intr(\dom r) = (0,\infty)$, so $(r')^{-1}(-L) > 0$. Hence $w > 0$ and
    \begin{equation*}
        |1/w|
        = 1/w
        \leq 1/(r')^{-1}(-L) \quad \text{a.e.\ in } \Omega,
    \end{equation*}
    and $1/w \in L^\infty(\Omega)$.

    Now take $w \in L^\infty(\Omega)$ with $1/w \in L^\infty(\Omega)$ and $w > 0$. Then $1/w \leq L$ a.e.\ in $\Omega$ for some $L > 0$. Equivalently, $w \geq 1/L$ a.e.\ in $\Omega$, so $\essinf_\Omega w \geq 1/L > 0$.

    Finally, take $w \in L^\infty(\Omega)$ with $\essinf_\Omega w > 0$. Then $w(x) \in (0,\infty) = \intr(\dom r)$ for a.e.\ $x \in \Omega$, so $r'(w)$ is well-defined almost everywhere. Now define $m = \essinf_\Omega w$ and $M = \|w\|_{L^\infty(\Omega)}$. Then $r'(m)$ is well-defined and
    \begin{equation*}
        m \leq w \leq M
        \iff r'(m) \leq r'(w) \leq r'(M)
        \quad \text{a.e.\ in } \Omega,
    \end{equation*}
    from which it follows that $w \in (r')^{-1}(L^\infty(\Omega))$.
\end{proof}

We are now ready to prove \Cref{lem:rn-prime-un-l-infty}.

\begin{proof}[Proof of \Cref{lem:rn-prime-un-l-infty}]
    First, recall \(u^{k-1} \in K^\circ\). Consequently, we can choose \(N\)
    large enough so that \(R_N'(u^{k-1}) = R'(u^{k-1})\) a.e.\ in \(\Omega\).
    From \Cref{lem:preimage-l-infty}, we know that $R'(u^{k-1}) \in
    L^\infty(\Omega)$. Now, define
    \begin{equation*}
        w_N = u_N - \phi,
        \quad
        \widetilde{f}_k = \alpha_k f + R'(u^{k-1}) + \alpha_k \dv(A \nabla \phi),
        \quad
        \widetilde{g} = g - \phi.
    \end{equation*}
    We have \(w_N \in H^1(\Omega)\) with \(\trace w_N = \widetilde{g}\) on
    \(\partial \Omega\). By the structure of \(R_N\), we have the identity
    \begin{equation*}
        R_N'(u_N)
        = r_N'(w_N)
    \end{equation*}
    almost everywhere on \(\Omega\). Moreover, \eqref{eq:entropic-pde-reg-ve}
    reads
    \begin{equation*}
        \alpha_k (A \nabla u_N, \nabla v) + (R_N'(u_N), v)
        = \alpha_k (f, v) + (R'(u^{k-1}),v)
        \quad \text{for all } v \in H^1_0(\Omega).
    \end{equation*}
    Integration by parts gives \((A \nabla \phi, \nabla v) = -(\dv(A \nabla
    \phi), v)\) for \(v \in H^1_0(\Omega)\). Using the decomposition \(\nabla
    u_N = \nabla w_N + \nabla \phi\), we obtain
    \begin{equation}
        \label{eq:shifted-pde}
        \alpha_k (A \nabla w_N, \nabla v) + (r_N'(w_N), v) = (\widetilde{f}_k, v)
        \quad \text{for all } v \in H_0^1(\Omega).
    \end{equation}
    Thus, \(R_N'(u_N) \in L^\infty(\Omega)\) if and only if
    \(r_N'(w_N) \in L^\infty(\Omega)\).

    Now, take any \(H \in C^1(\R)\) such that
    \begin{equation*}
        H(s) = 0 \text{ for all } s \leq 0,
        \quad H(s) > 0 \text{ for all } s > 0,
        \quad 0 \leq H'(s) \leq M \text{ for all } s \in \R
    \end{equation*}
    for some \(M > 0\). Let \(t = \pm 1\), and define \(\widetilde{g}_t\) by
    \(\widetilde{g}_1 = \esssup_{\partial \Omega} \widetilde{g}\) and
    \(\widetilde{g}_{-1} = \essinf_{\partial \Omega} \widetilde{g}\). Define
    \(L_t = \max(t r_N'(\widetilde{g}_t),
    \|\widetilde{f}_k\|_{L^\infty(\Omega)})\) and \(v_t \colon \Omega \to
    \R_{\geq 0}\) by \(v_t = H(t r_N'(w_N) - L_t)\). Through the Lipschitz
    continuity of $H, r_N'$ and repeated applications of \cite[Corollary
    A.6]{kinderlehrer2000introduction}, we know that \(v_t \in H^1(\Omega)\).
    Furthermore, as \(r_N'\) is monotone for all \(N\), and
    \(\widetilde{g}_{-1} \leq \widetilde{g} \leq \widetilde{g}_1\) almost
    everywhere on \(\partial \Omega\), we have
    \begin{equation*}
        t r_N'(\widetilde{g}) - L_t
        \leq t r_N'(\widetilde{g}_t) - L_t
        \leq 0
    \end{equation*}
    almost everywhere on \(\partial \Omega\), which implies
    \begin{equation*}
        \trace v_t
        = H(t r_N'(\trace w_N) - L_t)
        = H(t r_N'(\widetilde{g}) - L_t)
        = 0
        \quad \text{a.e.\ on } \partial \Omega.
    \end{equation*}
    Thus \(v_t \in H_0^1(\Omega)\). By \eqref{eq:shifted-pde}, it follows that
    \begin{equation*}
        \alpha_k (A \nabla w_N, \nabla v_t) + (r_N'(w_N), v_t)
        = (\widetilde{f}_k, v_t).
    \end{equation*}
    Equivalently,
    \begin{equation*}
        t \, \alpha_k (A \nabla w_N, \nabla v_t) + (L_t - t \widetilde{f}_k, v_t)
        = (L_t - t r_N'(w_N), v_t).
    \end{equation*}
    As \(r'' \geq 0\), it follows that \(r_N'' \geq 0\). Furthermore, \(H' \geq
    0\) on \(\R\), and since \(\nabla v_t = t H'(t r_N'(w_N) - L_t)
    r_N''(w_N) \nabla w_N\), it follows that
    \begin{equation*}
        t \, \alpha_k (A \nabla w_N, \nabla v_t)
        = t^2 \, \alpha_k (A \nabla w_N \cdot \nabla w_N, H'(t r_N'(w_N) - L_t) r_N''(w_N))
        \geq 0.
    \end{equation*}
    Moreover, as \((L_t - t \widetilde{f}_k, v_t) \geq 0\), it follows that
    \begin{equation*}
        (L_t - t r_N'(w_N), H(t r_N'(w_N) - L_t))
        = (L_t - t r_N'(w_N), v_t)
        \geq 0.
    \end{equation*}
    Since \(H(s) = 0\) for all \(s \leq 0\), we have that
    \begin{equation*}
        \int_{\{t r_N'(w_N) > L_t\}} (L_t - t r_N'(w_N)) H(t r_N'(w_N) - L_t) \diff x
        \geq 0.
    \end{equation*}
    But on the set \(\{t r_N'(w_N) > L_t\}\), the integrand is strictly
    negative, so it is necessary that \(\{t r_N'(w_N) > L_t\}\) has measure
    zero. Thus \(t r_N'(w_N) \leq L_t\) almost everywhere in \(\Omega\), and
    thus $r_N'(w_N) \in L^\infty(\Omega)$.

    The claim that \(\|R_N'(u_N)\|_{L^\infty(\Omega)}\) is bounded uniformly in
    \(N\) follows if we choose \(N\) large enough so that
    \(\esssup_{\partial \Omega} (g - \phi) \leq N\) and \(\essinf_{\partial
    \Omega} (g - \phi) \geq 1/N\). In particular,
    \begin{equation*}
        N
        \geq \max\!\left(\esssup_{\partial \Omega} (g - \phi),
             1/\essinf_{\partial \Omega} (g - \phi),
             \esssup_\Omega (u^{k-1} - \phi),
             1/\essinf_\Omega (u^{k-1} - \phi)\right)
    \end{equation*}
    suffices.
\end{proof}

\begin{proof}[Proof of \Cref{thm:entropic-pde}]
    There exists a unique \(u_N \in H_g^1(\Omega)\) that solves
    \eqref{eq:entropic-pde-reg-ve} for each \(N\). Furthermore, as $u^{k-1} \in
    K^\circ$, we can choose $N$ large enough so that $R_N'(u^{k-1}) =
    R'(u^{k-1})$ a.e.\ in $\Omega$. Defining
    \begin{equation*}
        w_N = u_N - \phi,
        \quad
        \widetilde{f}_k = \alpha_k f + R'(u^{k-1}) + \alpha_k \dv(A \nabla \phi)
    \end{equation*}
    and applying \eqref{eq:entropic-pde-reg-ve}, $w_N$ satisfies
    \begin{equation}
        \label{eq:shifted-pde-2}
        \alpha_k (A \nabla w_N, \nabla v) + (r_N'(w_N), v)
        = (\widetilde{f}_k,v)
        \quad \text{for all } v \in H_0^1(\Omega).
    \end{equation}

    Through \Cref{lem:rn-prime-un-l-infty}, we know that
    $\|r_N'(w_N)\|_{L^\infty(\Omega)} \leq L_k$ for some $L_k > 0$ independent
    of $N$. Consequently, $r_N'(w_N) \geq -L_k$ a.e.\ in $\Omega$, and as $-L_k \in
    \ran r'$, the quantity $m_k \coloneqq (r')^{-1}(-L_k)$ is well-defined and
    positive. Hence, choose $N$ sufficiently large so that $1/N \leq m_k \leq
    N$. It follows that
    \begin{equation*}
        r_N'(w_N)
        \geq -L_k
        = r'(m_k)
        = r_N'(m_k)
        \geq r_N'(1/N)
        \quad \text{a.e.\ in } \Omega,
    \end{equation*}
    so
    \begin{equation*}
        w_N \geq 1/N
        \quad \text{a.e.\ in } \Omega.
    \end{equation*}

    Define $C_k = \alpha_k^{-1} (\|\widetilde{f}_k\|_{L^\infty(\Omega)} + L_k)
    > 0$. Through \eqref{eq:shifted-pde-2}, we have for every $v \in H_0^1(\Omega)$,
    \begin{equation*}
        (A \nabla w_N, \nabla v)
        = \alpha_k^{-1} (\widetilde{f}_k - r_N'(w_N), v)
        \leq \alpha_k^{-1} (\|\widetilde{f}_k\|_{L^\infty(\Omega)} + L_k) \|v\|_{L^1(\Omega)},
    \end{equation*}
    so
    \begin{equation}
        \label{eq:shifted-pde-bound}
        (A \nabla w_N, \nabla v)
        \leq C_k \|v\|_{L^1(\Omega)}
        \quad \text{for all } v \in H_0^1(\Omega).
    \end{equation}
    Now, let $\psi \in H_0^1(\Omega)$ be the weak solution of
    \begin{equation}
        \label{eq:simple-poisson-pde}
        (A \nabla \psi, \nabla v)
        = (1,v)
        \quad \text{for all } v \in H_0^1(\Omega).
    \end{equation}
                    Through standard elliptic PDE theory (e.g., \cite[Theorem
    2.6]{bensoussan2013regularity}), we know that $\psi \in L^\infty(\Omega)$.
    Hence, define
    \begin{equation*}
        G = \esssup_{\partial \Omega} (g - \phi),
        \quad
        z_N
        = w_N - G - C_k \psi.
    \end{equation*}
    Note that $z_N \in H^1(\Omega)$ and
    \begin{equation*}
        \trace z_N
        = \trace w_N - G - C_k \trace \psi
        = g - \phi - G
        \leq 0
        \quad \text{a.e.\ on } \partial \Omega.
    \end{equation*}
    It follows that $(z_N)_+ \coloneqq \max(z_N,0) \in H_0^1(\Omega)$. Hence,
    applying \eqref{eq:shifted-pde-bound} and \eqref{eq:simple-poisson-pde}
    with $v = (z_N)_+$, we conclude
    \begin{align*}
        (A \nabla (z_N)_+, \nabla (z_N)_+)
        &= (A \nabla z_N, \nabla (z_N)_+) \\
        &= (A \nabla w_N, \nabla (z_N)_+) - C_k (A \nabla \psi, \nabla (z_N)_+) \\
        &= (A \nabla w_N, \nabla (z_N)_+) - C_k (1, (z_N)_+) \\
        &\leq 0.
    \end{align*}
    In conjunction with the uniform ellipticity of $A$,
    \begin{equation*}
    \frac{1}{\Lambda} \|\nabla (z_N)_+\|^2_{L^2(\Omega)}
        \leq (A \nabla (z_N)_+, \nabla (z_N)_+)
        \leq 0.
    \end{equation*}
    Thus $(z_N)_+ = 0$ a.e.\ in $\Omega$, or equivalently, $w_N \leq G + C_k
    \psi$ a.e.\ in $\Omega$. We now choose $N$ sufficiently large so that $G +
    C_k \|\psi\|_{L^\infty(\Omega)} \leq N$, demonstrating that
    \begin{equation*}
        1/N \leq w_N \leq N
        \quad \text{a.e.\ in } \Omega.
    \end{equation*}
    Hence \(R_N'(u_N) = R'(u_N)\) a.e.\ in \(\Omega\), and \(u_N \in K^\circ\)
    solves \eqref{eq:general-bpp-weak-form-A}.

    Lastly, define \(J\colon H_g^1(\Omega)\to \ER\) by
    \[
    J(v)
    \coloneqq
    E(v)
    +
    \alpha_k^{-1}
    \int_\Omega D_R(v,u^{k-1})\diff x.
    \]
    Let \(v\in K\). As \(v - u^k \in H_0^1(\Omega)\), we know that
    \begin{equation*}
        \langle E'(u^k),v - u^k\rangle + \alpha_k^{-1} (R'(u^k),v - u^k)
        = \alpha_k^{-1} (R'(u^{k-1}),v - u^k).
    \end{equation*}
    It follows that
    \begin{equation*}
        \alpha_k^{-1} \int_{\Omega} D_R(v,u^{k-1}) - D_R(u^k,u^{k-1}) \diff x
                        = -\langle E'(u^k), v - u^k\rangle + \alpha_k^{-1} \int_{\Omega} D_R(v,u^k) \diff x
    \end{equation*}
    and thus by the subgradient inequality and non-negativity of $D_R$,
    \begin{align*}
        J(v) - J(u^k)
        &= E(v) - E(u^k) + \alpha_k^{-1} \int_{\Omega} D_R(v,u^{k-1}) - D_R(u^k,u^{k-1}) \diff x \\
        &= E(v) - E(u^k) - \langle E'(u^k), v - u^k\rangle + \alpha_k^{-1} \int_{\Omega} D_R(v,u^k) \diff x \\
        &\geq 0.
    \end{align*}
    Uniqueness of $u^k$ follows from the fact that $J$ is strictly convex on $K$.
\end{proof}

\begin{remark}
    We emphasize that \Cref{thm:entropic-pde} holds even if $r$ is not
    superlinear, as in the case of the Tsallis entropy. If $r$ is superlinear,
    our proof of \Cref{thm:entropic-pde} may be simplified to avoid the use of
    the auxiliary function $\psi$.
\end{remark}

\subsection{A priori properties of the iterates}

We now prove the four a priori estimates stated in
\Cref{subsec:apriori-properties}, i.e., \Cref{prop:main-positive-invariance}, \Cref{lem:uniform-bound-above}, \Cref{lem:energy-dissipation} and \Cref{lem:uniform-H1-bound}, under the homogeneous setting
\eqref{eq:homogeneous-setting}. Together, these estimates provide the
structural control needed to estimate the order of minimality in
the subsequent analysis.

\subsubsection{Proof of \Cref{prop:main-positive-invariance}}

\begin{proof}
We argue by induction on $k$. The base case $k=0$ is the assumption. Suppose
that
\begin{equation*}
    u^{k-1}\geq u^*
    \qquad \text{a.e. in }\Omega.
\end{equation*}
We prove that $u^k\geq u^*$ a.e. in $\Omega$.

\medskip
\noindent\textit{Step 1: Define the test function.}
For each $n \in \mathbb{N}^+$ define
\begin{equation*}
    v_n \coloneqq  \left(u^* - u^k - \frac{1}{n}\right)_+ \in H^1(\Omega),
\end{equation*}
and let
\begin{equation*}
    \Sigma_n \coloneqq \left\{v_n > 0\right\}
    = \left\{u^k < u^* - \frac{1}{n} \right\},
\end{equation*}
so that $v_n>0$ precisely on $\Sigma_n$. We will show that $u^k\geq u^*-
\frac{1}{n}$ a.e. in $\Omega$ for every $n$, and then send $n\to\infty$.

\medskip
\noindent\textit{Step 2: The test function is admissible.}
We claim $v_n\in H_0^1(\Omega)$. The positive truncation preserves $H^1$, so
    $v_n\in H^1(\Omega)$ \cite[Theorem A.1]{kinderlehrer2000introduction}. Since $u^*,u^k\in H^1_0(\Omega)$ share the zero trace
on $\partial\Omega$ and the trace commutes with the positive part,
\begin{equation}
\label{eq:v-n-zero-trace}
    \trace v_n
    = \Bigl(\trace(u^*-u^k) - \tfrac{1}{n}\Bigr)_+
    = \bigl(-\tfrac{1}{n}\bigr)_+ = 0 .
\end{equation}
Hence $v_n\in H_0^1(\Omega)$. Note that $v_n\geq0$.

\medskip
\noindent\textit{Step 3: Variational inequality for $u^*$.}
Recall that $u^*$ solves the obstacle variational inequality
\eqref{eq:abstract-obstacle-vi},
\begin{equation*}
    a(u^*,v-u^*)\geq 0 \qquad\forall v\in K,
\end{equation*}
with $K=\{v\in H^1_0(\Omega):\ v\geq\phi\ \text{a.e.}\}$.

Both $u^*\pm v_n$ are admissible: they lie in $H^1_0(\Omega)$ by
\eqref{eq:v-n-zero-trace}, and
\begin{equation*}
    u^*+v_n\geq u^*\geq\phi,
    \qquad
    u^*-v_n=
      \begin{cases}
        u^*\geq\phi, & \text{on }\{v_n=0\},\\[2pt]
        u^k+\frac{1}{n}\geq\phi+\frac{1}{n}\geq\phi, & \text{on }\{v_n>0\},
      \end{cases}
\end{equation*}
where the second line uses the feasibility $u^k\geq\phi$. Testing the
variational inequality with $v=u^*+v_n$ and with $v=u^*-v_n$ yields
$a(u^*,v_n)\geq 0$ and $a(u^*,v_n)\leq 0$, respectively. Hence
\begin{equation*}
    a(u^*,v_n)=0.
\end{equation*}
Since $\nabla v_n=0$ a.e. outside $\Sigma_n$, we may restrict the integral
to $\Sigma_n$,
\begin{equation}
\label{eq:positive-invariance-eq-ustar}
    \int_{\Sigma_n} A\,\nabla u^*\cdot\nabla v_n \diff x = 0.
\end{equation}

\medskip
\noindent\textit{Step 4: Subtraction.}
By \eqref{eq:general-bpp-weak-form-A}, the $k$-th proximal step reads
\begin{equation}
\label{eq:positive-invariance-uk-eq}
    a(u^k,v)
    +\frac{1}{\alpha_k}\bigl(R'(u^k)-R'(u^{k-1}),v\bigr)
    =0
    \qquad \forall v\in H_0^1(\Omega),
\end{equation}
which applies in particular for $v=v_n\in H_0^1(\Omega)$. Testing
\eqref{eq:positive-invariance-uk-eq} with $v=v_n$, and using that $v_n$ and
$\nabla v_n$ vanish a.e. outside $\Sigma_n$ to restrict every integral to
$\Sigma_n$, gives
\begin{equation}
\label{eq:positive-invariance-u-k}
    \int_{\Sigma_n}A\,\nabla u^k\cdot\nabla v_n \diff x
    +\frac{1}{\alpha_k}\int_{\Sigma_n}R'(u^k)\,v_n \diff x
    =\frac{1}{\alpha_k} \int_{\Sigma_n}R'(u^{k-1})\,v_n \diff x.
\end{equation}
On $\Sigma_n$, feasibility gives $u^*>u^k+\frac{1}{n}\geq\phi+\frac{1}{n}$,
so $u^*-\phi>\frac{1}{n}>0$ and hence $R'(u^*)$ is well defined and bounded
here. Subtracting \eqref{eq:positive-invariance-eq-ustar} from
\eqref{eq:positive-invariance-u-k} and adding and subtracting
$\tfrac{1}{\alpha_k}\int_{\Sigma_n}R'(u^*)v_n \diff x$, we arrive at
\begin{align}
    \int_{\Sigma_n}A\,\nabla (u^k-u^*)\cdot\nabla v_n \diff x
    &+\frac{1}{\alpha_k}\int_{\Sigma_n}\bigl(R'(u^k) - R'(u^*)\bigr)v_n \diff x
    \nonumber \\
    &=\frac{1}{\alpha_k} \int_{\Sigma_n}\bigl(R'(u^{k-1}) - R'(u^*)\bigr)v_n \diff x.
    \label{eq:positive-invariance-diff}
\end{align}

\medskip
\noindent\textit{Step 5: Sign of each term.}
On $\Sigma_n$,
\begin{equation*}
    v_n=u^*-\frac{1}{n}-u^k,
    \qquad
    \nabla v_n=\nabla u^*-\nabla u^k
    \qquad\text{a.e.}
\end{equation*}
Using the uniform ellipticity \eqref{eq:uniform-ellipticity} of $A$, and
\begin{equation*}
    \nabla(u^k-u^*)=-\nabla v_n \quad \text{a.e.\ on} \; \Sigma_n,
\end{equation*}
the bilinear term becomes
\begin{align}
    \int_{\Sigma_n} A\,\nabla(u^k-u^*)\cdot\nabla v_n \diff x \nonumber
    & = -\int_{\Sigma_n} A\,\nabla v_n\cdot\nabla v_n \diff x \\
    & = -\int_{\Omega} A\,\nabla v_n\cdot\nabla v_n \diff x
     \leq -\frac{1}{\Lambda}\|\nabla v_n\|_{L^2(\Omega)}^2.
     \label{eq:bilinear-form}
\end{align}

For the first Bregman term, on $\Sigma_n$ we have $\phi\leq
u^k<u^*-\frac{1}{n}<u^*$. Since $R'$ is nondecreasing,
\begin{equation*}
    R'(u^k)\leq R'(u^*)
    \qquad \text{a.e. on }\Sigma_n,
\end{equation*}
and since $v_n\geq0$,
\begin{equation}
\label{eq:first-bregman-term}
    \int_{\Sigma_n}\bigl(R'(u^k)-R'(u^*)\bigr)v_n \diff x \leq 0.
\end{equation}

For the second Bregman term, by the induction hypothesis $u^{k-1}\geq u^*$
a.e. and monotonicity of $R'$,
\begin{equation*}
    R'(u^{k-1})\geq R'(u^*) \qquad \text{a.e. on }\Sigma_n,
\end{equation*}
and since $v_n\geq0$,
\begin{equation}
\label{eq:second-bregman-term}
    \int_{\Sigma_n}\bigl(R'(u^{k-1}) - R'(u^*)\bigr)v_n \diff x\geq 0.
\end{equation}

\medskip
\noindent\textit{Step 6: Combining.}
Substituting \eqref{eq:bilinear-form}, \eqref{eq:first-bregman-term} and
\eqref{eq:second-bregman-term} into \eqref{eq:positive-invariance-diff} yields
\begin{align}
       0 &\leq \frac{1}{\alpha_k}\int_{\Sigma_n}\bigl(R'(u^{k-1}) - R'(u^*)\bigr)v_n \diff x  \nonumber \\
       &= \int_{\Sigma_n}A\,\nabla (u^k-u^*)\cdot\nabla v_n \diff x
    +\frac{1}{\alpha_k}\int_{\Sigma_n}\bigl(R'(u^k) - R'(u^*)\bigr)v_n \diff x \nonumber \\
       &\leq -\frac{1}{\Lambda}\|\nabla v_n\|_{L^2(\Omega)}^2 \leq 0.
\end{align}
Hence $\|\nabla v_n\|_{L^2(\Omega)}=0$. Since $v_n\in H_0^1(\Omega)$, the
Poincar\'e inequality yields $v_n=0$ a.e., i.e.,
\begin{equation*}
    u^k\geq u^*-\frac{1}{n}\qquad \text{a.e. in }\Omega.
\end{equation*}
As this holds for every $n$, letting $n\to\infty$ gives $u^k\geq u^*$ a.e.
in $\Omega$. The induction is complete.
\end{proof}

\subsubsection{Proof of \Cref{lem:uniform-bound-above}}

\begin{proof}
By the definition of $\overline M_1$, we have
\begin{equation*}
    \overline M_1-\phi\geq 1 \quad \text{and} \quad u^0\leq \overline M_1,
    \qquad\text{a.e.\ in }\Omega.
\end{equation*} We prove by induction that
$u^k\leq \overline M_1$ a.e.\ for every $k\geq0$.

The case $k=0$ follows from the definition of $\overline M_1$. Suppose that $u^{k-1}\leq
\overline M_1$ a.e.\ in $\Omega$. We show that $u^k\leq \overline M_1$. Let
\begin{equation*}
    z \coloneqq  (u^k- \overline M_1)_+ .
\end{equation*}
Since $u^k\in H_0^1(\Omega)$, the trace of $u^k- \overline M_1$ is nonpositive on
$\partial\Omega$. Hence $z\in H_0^1(\Omega)$. Moreover $z\geq0$.

Testing \eqref{eq:general-bpp-weak-form-A} with $v=z$, and using $F=0$, gives
\begin{equation}
    \alpha_k a(u^k,z)
    +
    (R'(u^k)-R'(u^{k-1}),z)
    =
    0.
\end{equation}
Since $\overline M_1$ is constant, $a(\overline M_1,z)=0$. Therefore
\begin{equation}
\label{eq:M1-bar-weak}
    \alpha_k a(u^k-\overline M_1,z)
    =
    -
    (R'(u^k)-R'(u^{k-1}),z).
\end{equation}
On the set $\{z>0\}=\{u^k > \overline M_1\}$, the induction hypothesis implies
\begin{equation}
    u^k-\phi
    >
    \overline M_1-\phi
    \geq
    u^{k-1}-\phi .
\end{equation}
Since $R'(u)=r'(u-\phi)$ and $r'$ is nondecreasing, we have
\begin{equation*}
    R'(u^k)-R'(u^{k-1})\geq0
    \qquad\text{a.e. on }\{z>0\}.
\end{equation*}
As $z=0$ outside $\{z>0\}$, it follows that
\begin{equation}
    (R'(u^k)-R'(u^{k-1}),z)\geq0 .
\end{equation}

It remains to identify the left-hand side of \eqref{eq:M1-bar-weak}. By the truncation chain rule,
\begin{equation}
    \nabla z
    =
    \nabla(u^k- \overline M_1)\,\mathbf 1_{\{u^k> \overline M_1\}}.
\end{equation}
Thus, we obtain
\begin{equation}
    a(u^k-\overline M_1,z)
    =
    a(z,z).
\end{equation}
Using the uniform ellipticity of $A$, we get
\begin{equation}
    \frac{\alpha_k}{\Lambda}\|\nabla z\|_{L^2(\Omega)}^2
    \leq
    \alpha_k a(z,z)
    =
    \alpha_k a(u^k-\overline M_1,z)
    \leq0 .
\end{equation}
Therefore $\nabla z=0$ a.e. in $\Omega$. Since $z\in H_0^1(\Omega)$, it implies
$z=0$ a.e. in $\Omega$. Thus $u^k\leq \overline M_1$ a.e. in $\Omega$.

By induction, $u^k\leq \overline M_1$ a.e. for all $k\geq0$. Therefore
\begin{equation*}
     u^k-\phi
    \leq
    \overline M_1-\essinf_{\Omega}\phi
    \qquad\text{a.e. in }\Omega .
\end{equation*}
This gives the claimed uniform bound \eqref{eq:M1-unfiorm-bound}.
\end{proof}

\subsubsection{Proof of \Cref{lem:energy-dissipation}}

A similar energy dissipation has previously been established in the
infinite-dimensional setting for the Dirichlet energy and the Shannon entropy
in \cite[Theorem 4.13]{keith2024proximal}, with a related finite-dimensional
result given in \cite[Lemma~3.3]{chen1993convergence}. Here we give a different
proof directly from the weak formulation, using only the monotonicity of $R'$.
The argument therefore extends naturally to the more general quadratic energies
and Legendre functions considered in this work.

\begin{proof}
Since $u^k$ and $u^{k+1}$ satisfy the same Dirichlet boundary condition, we
have
\begin{equation*}
    u^k-u^{k+1}\in H_0^1(\Omega).
\end{equation*}
Thus $u^k - u^{k+1}$ is an admissible test function in the $(k+1)$-th step
of \eqref{eq:general-bpp-weak-form-A}. Taking $v=u^k - u^{k+1}$ in
\eqref{eq:general-bpp-weak-form-A}, we obtain
\begin{equation}
    \alpha_{k+1} \left\langle E'(u^{k+1}), u^k-u^{k+1}\right\rangle
    =
    \left( R'(u^{k+1})-R'(u^k),
    u^{k+1}-u^{k}\right).
\end{equation}
Since $R'$ is monotone nondecreasing, we obtain
\begin{equation}
    \left( R'(u^{k+1})-R'(u^k),
    u^{k+1}-u^k\right) \geq 0,
\end{equation}
and therefore
\begin{equation}
    \left\langle E'(u^{k+1}), u^k-u^{k+1}\right\rangle \geq 0.
\end{equation}
On the other hand, by the convexity of $E$, we have
\begin{equation}
    E(u^k)-E(u^{k+1})
    \geq
    \left\langle E'(u^{k+1}), u^k-u^{k+1}\right\rangle \geq 0.
\end{equation}
This proves the energy dissipation property \eqref{eq:energy-dissipation}.
\end{proof}

\subsubsection{Proof of \Cref{lem:uniform-H1-bound}}

\begin{proof}
By \Cref{lem:energy-dissipation},
\begin{equation}
    E(u^k)
    \leq
    E(u^0)
    \qquad
    \forall k\geq0.
\end{equation}
Since $F=0$, the uniform ellipticity \eqref{eq:uniform-ellipticity} and Poincar\'e's inequality give
\begin{equation}
    \frac{1}{2\Lambda(C_P^2+1)}
    \|u^k\|_{H^1(\Omega)}^2
    \leq
    \frac{1}{2\Lambda}
    \|\nabla u^k\|_{L^2(\Omega)}^2
    \leq
    E(u^k).
\end{equation}
On the other hand, the upper ellipticity bound yields
\begin{equation}
    E(u^0)
    \leq
    \frac{\Lambda}{2}
    \|\nabla u^0\|_{L^2(\Omega)}^2
    \leq
    \frac{\Lambda}{2}
    \|u^0\|_{H^1(\Omega)}^2.
\end{equation}
Combining the preceding estimates gives
\begin{equation}
    \|u^k\|_{H^1(\Omega)}
    \leq
    \Lambda
    \left(
        C_P^2+1
    \right)^{1/2}
    \|u^0\|_{H^1(\Omega)}
    =
    M_2.
\end{equation}
This completes the proof.
\end{proof}

\section{Proof of convergence rates}
\label{sec:convergence-rate}

We prove the convergence results in the same order in which they are stated in
\Cref{subsec:convergence-rates}. We first establish the abstract convergence
mechanism. We then verify sequential strict local minimality in the two regimes
$s \geq 2$ and $s=1$, together with the corresponding sublinear and linear
convergence rates.
Throughout this section, we continue to work under the homogeneous setting \eqref{eq:homogeneous-setting}.

\subsection{The convergence mechanism}

The proof relies on two ingredients: the one-step Bregman descent estimate
in \Cref{lem:one-step-bregman-descent} and a scalar discrete inequality
in \Cref{lem:discrete-bihari-lasalle}. Together they yield \Cref{thm:energy-bregman-convergence} for every
admissible minimality order.

\subsubsection{Proof of \Cref{lem:one-step-bregman-descent}}

The proof may be viewed as an infinite-dimensional extension of the
finite-dimensional Bregman proximal descent argument in
\cite[Lemma~3.2]{chen1993convergence}. In both settings, the proximal
optimality condition and the three-point identity yield a telescoping
difference of Bregman errors. The present argument differs from the discrete
analysis in \cite[Theorem~3.3]{keith2025priori}, where a Galerkin projection
introduces an additional consistency term. At the continuous level, $u^k$ is
directly admissible, so no such consistency term appears.

\begin{proof}
    By convexity of $E$,
\begin{equation}
\label{eq:sublinear-convexity}
    E(u^*) \geq E(u^k) + \langle E'(u^k), u^*-u^k\rangle.
\end{equation}
The weak form \eqref{eq:general-bpp-weak-form-A} of the proximal step gives
\begin{equation*}
    \alpha_k\langle E'(u^k),v\rangle
    = \bigl(R'(u^{k-1})-R'(u^k),\,v\bigr)
    \qquad \forall v\in H_0^1(\Omega).
\end{equation*}
Since $u^*-u^k\in H_0^1(\Omega)$, substitution into \eqref{eq:sublinear-convexity}
yields
\begin{equation}
\label{eq:sublinear-after-EL}
    E(u^*)
    \geq E(u^k)
       + \frac{1}{\alpha_k}\bigl(R'(u^{k-1})-R'(u^k),\,u^*-u^k\bigr).
\end{equation}
We now invoke the three-point identity \eqref{eq:three-point-identity}, applied
with $u=u^*$, $v=u^k$, $w=u^{k-1}$:
\begin{equation*}
    \bigl(R'(u^k)-R'(u^{k-1}),\,u^k-u^*\bigr)
    = \int_{\Omega} D_R(u^*,u^k)-D_R(u^*,u^{k-1})+D_R(u^k,u^{k-1}) \diff x.
\end{equation*}
Substituting into \eqref{eq:sublinear-after-EL},
\begin{equation*}
    E(u^*)-E(u^k)
    \geq \frac{1}{\alpha_k}\left(
        B_k - B_{k-1} + \int_\Omega D_R(u^k,u^{k-1}) \diff x
    \right).
\end{equation*}
Since $D_R(u^k,u^{k-1})\geq 0$, rearranging gives exactly \eqref{eq:one-step-descent}:
\begin{equation}
    B_{k-1}-B_k \;\geq\; \alpha_k\bigl(E(u^k)-E(u^*)\bigr).
\end{equation}
This completes the proof.
\end{proof}

\subsubsection{Proof of \Cref{thm:energy-bregman-convergence}}

Before proving \Cref{thm:energy-bregman-convergence}, we record the discrete Bihari--LaSalle inequality. The classical Bihari--LaSalle inequality is a generalization of Gr\"onwall's inequality; see \cite{bihari1956generalization} and \cite{lasalle1949uniqueness}. The following lemma exhibits a discrete counterpart that converts the recursion satisfied by the Bregman error into an explicit convergence rate.

\begin{lemma}[Discrete Bihari--LaSalle inequality]
\label{lem:discrete-bihari-lasalle}
Let $\{B_k\}_{k\geq0}$ be a sequence of nonnegative real numbers satisfying
\begin{equation}
\label{eq:discrete-bihari-hypothesis}
    B_{k-1}-B_k
    \geq
    \lambda_k B_k^\beta,
    \qquad
    \beta>1,
    \quad
    \lambda_k>0,
\end{equation}
for every $k\geq1$. Then $\{B_k\}$ is nonincreasing. If $B_0=0$, then $B_k=0$ for all $k\geq0$. If $B_0>0$, then
\begin{equation}
\label{eq:discrete-bihari-conclusion}
    B_k
    \leq
    \left(
        B_0^{1-\beta}
        +
        (\beta-1)
        \sum_{i=1}^k
        \frac{\lambda_i}
        {\left(1+\lambda_i B_0^{\beta-1}\right)^\beta}
    \right)^{-1/(\beta-1)}.
\end{equation}
In particular, if
$\lambda_k\equiv\lambda>0$, then
\begin{equation}
    B_k
    =
    \mathcal O\left(k^{-1/(\beta-1)}\right).
\end{equation}
\end{lemma}

The proof of \Cref{lem:discrete-bihari-lasalle} is deferred to \Cref{app:discrete-bihari-lasalle}.

\begin{proof}
Combining the sequential strict local minimality inequality
\eqref{eq:energy-bregman-inequality} with the one-step Bregman descent
\eqref{eq:one-step-descent} gives
\begin{equation}
\label{eq:abstract-energy-bregman-recursion}
    B_{k-1}-B_k
    \geq
    \alpha_k\sigma B_k^s
    \qquad
    \forall k\geq1.
\end{equation}

Suppose first that $s>1$. Applying
\Cref{lem:discrete-bihari-lasalle} to
\eqref{eq:abstract-energy-bregman-recursion} with
$\beta=s$ and $\lambda_k=\alpha_k\sigma$ yields
\eqref{eq:abstract-bregman-rate}. When $\alpha_k\equiv\alpha>0$, the
constant-step conclusion of the same lemma gives
\eqref{eq:abstract-bregman-fixed-rate}.

Suppose now that $s=1$. From
\eqref{eq:abstract-energy-bregman-recursion},
\begin{equation}
    \left(1+\alpha_k\sigma\right)B_k
    \leq
    B_{k-1}.
\end{equation}
Iterating this inequality gives
\begin{equation}
    B_k
    \leq
    \prod_{i=1}^k
    \frac{1}{1+\alpha_i\sigma}
    B_0,
\end{equation}
which is \eqref{eq:abstract-bregman-linear-product}. For
$\alpha_k\equiv\alpha>0$, this reduces to
\eqref{eq:abstract-bregman-linear-rate}.
\end{proof}

\subsubsection{Proof of \Cref{cor:energy-gap-H1-error-convergence}}
\begin{proof}
By \Cref{lem:energy-dissipation}, the sequence
$\{E(u^k)-E(u^*)\}$ is nonincreasing. Summing the one-step Bregman
descent estimate \eqref{eq:one-step-descent} from \(j=m+1\) to \(j=k\)
gives
\begin{equation*}
    \sum_{j=m+1}^{k}\alpha_j (E(u^j)-E(u^*))
    \leq
    B_m-B_k
    \leq
    B_m.
\end{equation*}
Since \(E(u^j)-E(u^*)\geq E(u^k)-E(u^*)\) for every \(m<j\leq k\), we obtain
\begin{equation}
    \left(
        \sum_{j=m+1}^{k}\alpha_j
    \right)(E(u^k)-E(u^*))
    \leq
    \sum_{j=m+1}^{k}\alpha_j(E(u^j)-E(u^*))
    \leq
    B_m.
\end{equation}
Therefore,
\begin{equation}
\label{eq:energy-gap-convergence}
    E(u^k)-E(u^*)
    \leq
    \frac{B_m}
    {\displaystyle\sum_{j=m+1}^{k}\alpha_j}.
\end{equation}

We next control the \(H^1(\Omega)\)-error by the energy gap. Expanding the energy difference with $F = 0$ yields
\begin{equation}
\label{eq:expand-energy-diff}
    E(u^k)-E(u^*)
    = \frac12 a(u^k-u^*,u^k-u^*)
      + a(u^*,u^k-u^*).
\end{equation}
Since $u^k\in K$, the variational inequality \eqref{eq:obstacle-vi} applied with $v=u^k$
gives $a(u^*,u^k-u^*)\geq 0$. Combined with the uniform ellipticity of $A$,
\begin{equation}
\label{eq:energy-bregman-step}
    E(u^k)-E(u^*)
    \geq \frac12 a(u^k-u^*,u^k-u^*)
    \geq \frac{1}{2\Lambda}\,\|\nabla(u^k-u^*)\|_{L^2(\Omega)}^2.
\end{equation}
Poincar\'e's inequality then gives
\begin{align}
    \|u^k-u^*\|_{H^1(\Omega)}^2
    &\leq
    (C_P^2+1)
    \|\nabla (u^k-u^*)\|_{L^2(\Omega)}^2 \nonumber
    \\
    &\leq
    2\Lambda(C_P^2+1)
    \bigl(E(u^k)-E(u^*)\bigr).
    \label{eq:energy-controls-H1}
\end{align}
This proves the asserted \(H^1(\Omega)\)-error bound.

It remains to derive the rates under the constant-step assumption
\(\alpha_k\equiv\alpha>0\). First, suppose that \(s>1\). For \(k\geq2\),
choose $m=\left\lfloor\frac{k}{2}\right\rfloor$.
Since $k-\left\lfloor\frac{k}{2}\right\rfloor \geq \frac{k}{2}$,
the preceding energy-gap estimate \eqref{eq:energy-gap-convergence} yields
\begin{equation}
    E(u^k)-E(u^*)
    \leq
    \frac{2}{\alpha k}
    B_{\lfloor k/2\rfloor}.
\end{equation}
By \eqref{eq:abstract-bregman-fixed-rate}, we have
\[
    B_{\lfloor k/2\rfloor}
    =
    \mathcal O\left(
        k^{-1/(s-1)}
    \right).
\]
Consequently, we obtain
\begin{equation*}
    E(u^k)-E(u^*) =
    \mathcal O\left(
        k^{-s/(s-1)}
    \right).
\end{equation*}
Combining this with \eqref{eq:energy-controls-H1} gives
\begin{equation*}
    \|u^k-u^*\|_{H^1(\Omega)}
    =
    \mathcal O\left(
        k^{-s/(2(s-1))}
    \right).
\end{equation*}

Finally, suppose that \(s=1\). Taking \(m=k-1\) in \eqref{eq:energy-gap-convergence} gives
\begin{equation}
    E(u^k)-E(u^*)
    \leq
    \frac{B_{k-1}}{\alpha}.
\end{equation}
By \eqref{eq:abstract-bregman-linear-rate}, we have
\[
    B_{k-1}
    \leq
    (1+\alpha\sigma)^{-(k-1)}B_0.
\]
Hence,
\begin{align*}
    E(u^k)-E(u^*)
    &\leq
    \frac{(1+\alpha\sigma)B_0}{\alpha}
    (1+\alpha\sigma)^{-k} =
    \mathcal O\left(
        (1+\alpha\sigma)^{-k}
    \right).
\end{align*}
Applying \eqref{eq:energy-controls-H1} once more yields
\[
    \|u^k-u^*\|_{H^1(\Omega)}
    =
    \mathcal O\left(
        (1+\alpha\sigma)^{-k/2}
    \right).
\]
\end{proof}

\subsection{Sequential strict local minimality of order \texorpdfstring{$s\geq2$}{s >= 2}}

We now verify the order $s=2/\theta$ under the one-sided Bregman growth
condition \eqref{eq:one-sided-bregman-growth} and then invoke \Cref{thm:energy-bregman-convergence} and \Cref{cor:energy-gap-H1-error-convergence}.

\subsubsection{Proof of \Cref{prop:main-energy-bregman}}

\begin{proof}
By \Cref{prop:main-positive-invariance,lem:uniform-bound-above},
\begin{equation}
u^*-\phi\in[0,M_1],
\qquad
u^k-\phi\in(0,M_1],
\qquad
u^*-\phi\leq u^k-\phi
\qquad
\text{a.e. in }\Omega.
\end{equation}
We split the remainder of the proof into three steps.

\medskip
\noindent\textit{Step 1: Lower bound for the energy gap.}
We recall the energy-gap estimate \eqref{eq:energy-bregman-step}:
\begin{equation}
    E(u^k)-E(u^*)
    \geq
    \frac{1}{2\Lambda}
    \|\nabla(u^k-u^*)\|_{L^2(\Omega)}^2.
    \label{eq:energy-bregman-step1}
\end{equation}

\medskip
\noindent\textit{Step 2: Pointwise control of the Bregman divergence.}
By  \Cref{def:one-sided-bregman-growth}, the one-sided Bregman growth condition
\eqref{eq:one-sided-bregman-growth} yields
\begin{equation}
\label{eq:energy-bregman-pointwise}
    D_R(u^*,u^k)
    = D_r(u^*-\phi,u^k-\phi)
    \leq C_{M_1} (u^k-u^*)^\theta
    \qquad \text{a.e. in }\Omega.
\end{equation}

\medskip
\noindent\textit{Step 3: Integral control.}
Integrating \eqref{eq:energy-bregman-pointwise} and applying H\"older's
inequality with exponents $2/\theta\geq 1$ and its conjugate $2/(2-\theta)$ gives
\begin{align}
    B_k
    = \int_\Omega D_R(u^*,u^k) \diff x
    &\leq C_{M_1} \int_\Omega (u^k-u^*)^\theta \diff x \nonumber \\
    &\leq C_{M_1}\,|\Omega|^{(2-\theta)/2}
        \left(\int_\Omega (u^k-u^*)^2\diff x\right)^{\theta/2} \nonumber \\
    &= C_{ M_1}\,|\Omega|^{(2-\theta)/2}\,\|u^k-u^*\|_{L^2(\Omega)}^\theta.
\end{align}
Since $u^k-u^*\in H_0^1(\Omega)$, Poincar\'e's inequality gives
$\|u^k-u^*\|_{L^2(\Omega)}\leq C_P\|\nabla(u^k-u^*)\|_{L^2(\Omega)}$, so
\begin{equation*}
    B_k
    \leq C_{ M_1}\,|\Omega|^{(2-\theta)/2}\,C_P^{\theta}\,
        \|\nabla(u^k-u^*)\|_{L^2(\Omega)}^{\theta}.
\end{equation*}
Raising both sides to the power $2/\theta$,
\begin{equation}
\label{eq:energy-bregman-step3}
    B_k^{2/\theta}
    \leq \bigl(C_{M_1}\,|\Omega|^{(2-\theta)/2}\,C_P^{\theta}\bigr)^{2/\theta}
        \,\|\nabla(u^k-u^*)\|_{L^2(\Omega)}^2.
\end{equation}
Combining \eqref{eq:energy-bregman-step1} and \eqref{eq:energy-bregman-step3},
\begin{equation*}
    E(u^k)-E(u^*)
    \geq \sigma\, B_k^{2/\theta},
    \qquad
    \sigma
    \coloneqq  \frac{1}{2\Lambda\,
        \bigl(C_{M_1}\,|\Omega|^{(2-\theta)/2}\,C_P^{\theta}\bigr)^{2/\theta}}.
\end{equation*}
The constant $\sigma$ depends only on $\Omega$, $A$, $C_{M_1}$,
and $\theta$, as claimed.
\end{proof}

\subsubsection{Proof of \Cref{thm:main-sublinear-convergence}}

\begin{proof}
By \Cref{prop:main-energy-bregman},
\[
    s=\frac{2}{\theta}
\]
is an admissible minimality order.
Since $\theta\in(0,1]$, one has $s>1$. Therefore,
\Cref{thm:energy-bregman-convergence} yields
\[
    B_k
    =
    \mathcal O\left(k^{-1/(s-1)}\right)
    =
    \mathcal O\left(k^{-\theta/(2-\theta)}\right),
\]
which proves \eqref{eq:main-Ak-rate}.

Applying
\Cref{cor:energy-gap-H1-error-convergence} with
$s=2/\theta$ gives
\[
    E(u^k)-E(u^*)
    =
    \mathcal O\left(k^{-s/(s-1)}\right)
    =
    \mathcal O\left(k^{-2/(2-\theta)}\right)
\]
and
\[
    \|u^k-u^*\|_{H^1(\Omega)}
    =
    \mathcal O\left(k^{-s/(2(s-1))}\right)
    =
    \mathcal O\left(k^{-1/(2-\theta)}\right).
\]
These are \eqref{eq:main-energy-rate} and
\eqref{eq:main-H1-rate}, respectively.
\end{proof}

\subsection{Sequential strict local minimality of order \texorpdfstring{$s=1$}{s = 1}}

Finally, under the free-boundary assumptions, we prove the stronger order
$s=1$ for the Shannon and Spence entropies and therefore derive linear convergence rates.

\subsubsection{Auxiliary estimates}

Before proving \Cref{prop:main-local-bregman-coercivity}, we collect three
auxiliary lemmas. The first is a refined upper bound on the Bregman divergence associated with the Shannon and Spence entropies in the positive invariant regime $u\geq u^*$. The estimate captures two distinct behaviors: on the active set, where $u^*=\phi$, the divergence is controlled by a first-order term in $u-u^*$, whereas in the inactive region, where $u^*-\phi$ is positive, it is controlled by the weighted second-order term $\frac{(u-u^*)^2}{u^*-\phi}$.

\begin{lemma}[Refined Bregman bound for the Shannon and Spence entropies]
\label{lem:delicate-bregman-bound}
Let $R(u)=r(u-\phi)$, where $r$ is either the Shannon entropy
$
    r(s)=s\log s - s
$
or the Spence entropy
$
    r(s)=\tfrac12 s^2+\li_2(e^{-s})-\tfrac{\pi^2}{6}.
$
Then for every $M_1>0$ there exists a constant
$\widetilde C_{M_1}>0$ such that
\begin{equation*}
    D_R(v,u)
    \;\leq\;
    \widetilde C_{M_1}\,\frac{(u-v)^2}{(v-\phi)+(u-v)} = \widetilde C_{M_1}\,\frac{(u-v)^2}{u-\phi}
\end{equation*}
for all $\phi\leq v\leq\phi+ M_1 $, $\phi < u \leq \phi+ M_1$ and $v \leq u$. For Shannon entropy, the
bound holds with $\widetilde C_{M_1}=1$ independent of
$M_1$.
\end{lemma}
The proof of \Cref{lem:delicate-bregman-bound} is deferred to
\Cref{app:delicate-bregman-bound}.

The second lemma in this section is a one-dimensional weighted integral estimate
that will be used to control the Bregman divergence in the tubular neighborhood
of the free boundary.

\begin{lemma}[One-dimensional weighted estimate]
\label{lem:one-dim-weighted}
Let $\rho>0$ and let $w\in H^1(0,\rho)$ with $w\geq0$. Then there exists a
constant $C>0$ such that
\begin{equation*}
    \int_0^\rho \frac{w(s)^2}{s^2+w(s)}\,ds
    \;\leq\;
    C\left(
        |w(0)|^{3/2}
        + \int_0^\rho |w'(s)|^2\,ds
    \right).
\end{equation*}
\end{lemma}

The third lemma is a trace interpolation inequality. It controls the $L^{3/2}$
trace norm by a sum of $L^1$ norm and $H^1$ seminorm, with a constant that
depends on the $H^1$-norm bound.

\begin{lemma}[Trace interpolation]
\label{lem:trace-interpolation}
Let $D\subset\mathbb R^d$ be a bounded Lipschitz open set with finitely many
connected components, and let $M>0$. Then there exists a constant
$C=C(D,M)>0$ such that for every $w\in H^1(D)$ with $w\geq0$ a.e. in $D$
and $\|w\|_{H^1(D)}\leq M$, one has
\begin{equation*}
    \|w\|_{L^{3/2}(\partial D)}^{3/2}
    \leq
    C\left(
        \|w\|_{L^1(D)}
        +
        \|\nabla w\|_{L^2(D)}^2
    \right).
\end{equation*}
\end{lemma}

The proofs of \Cref{lem:one-dim-weighted} and \Cref{lem:trace-interpolation} are
deferred to \Cref{app:one-dim-weighted} and \Cref{app:trace-interpolation},
respectively.

\subsubsection{Proof of \Cref{prop:main-local-bregman-coercivity}}

\begin{proof}
Fix $k\geq0$ and set $u\coloneqq u^k$. By
\Cref{prop:main-positive-invariance}, \Cref{lem:uniform-bound-above} and \Cref{lem:uniform-H1-bound},
\begin{equation}
    u^*\leq u\leq \phi+ M_1
    \qquad\text{a.e. in }\Omega,
    \qquad
    \|u\|_{H^1(\Omega)}\leq M_2.
\end{equation}
Throughout the proof, we write
\begin{equation}
    e \coloneqq u-u^*,
    \qquad
    w^* \coloneqq u^*-\phi.
\end{equation}
Then $e\geq0$ a.e. in $\Omega$ and $e\in H_0^1(\Omega)$.

\medskip
\noindent\textit{Step 1: Lower bound for the energy gap.}
By \eqref{eq:expand-energy-diff}, expanding the quadratic energy yields
\begin{equation*}
    E(u)-E(u^*)
    = \frac12 a(e,e) + a(u^*,e).
\end{equation*}
By the variational inequality for $u^*$ and the representation
\eqref{eq:E-prime-representation}, we obtain
\begin{equation*}
    a(u^*,e)
    = \langle E'(u^*),e\rangle
    = \int_\Omega \lambda^* e \diff x.
\end{equation*}
Decomposing $\Omega = \Omega_0\cup\Omega_+$, on the inactive set $\Omega_+$
we have $\lambda^*=0$ a.e. on $\Omega_+$. On the active set $\Omega_0$,
strict complementarity \eqref{eq:strict-complementarity-density} gives
$\lambda^*\geq c_\phi$. Therefore,
\begin{equation*}
    \int_\Omega \lambda^* e \diff x
    = \int_{\Omega_0} \lambda^* e \diff x
    \geq c_\phi\, \|e\|_{L^1(\Omega_0)}.
\end{equation*}
Combined with the uniform ellipticity $a(e,e)\geq\frac{1}{\Lambda}\|\nabla e\|_{L^2(\Omega)}^2$,
\begin{equation}
\label{eq:coercivity-energy-lower}
    E(u)-E(u^*)
    \;\geq\; \frac{1}{2\Lambda}\|\nabla e\|_{L^2(\Omega)}^2
        + c_\phi\,\|e\|_{L^1(\Omega_0)}.
\end{equation}

\medskip
\noindent\textit{Step 2: Decomposition.}
By \Cref{prop:strict-complementarity-free-boundary}, fix $\rho>0$ as
in the tubular parametrization and set $\rho_0\coloneqq \rho/2$. Decompose
\begin{equation*}
    \Omega = \Omega_0 \;\cup\; \mathcal N_{\rho_0}^+(\Gamma) \;\cup\;
    \Omega_{\mathrm{far}},
\end{equation*}
where
\begin{equation*}
    \Omega_{\mathrm{far}}
    \coloneqq  \Omega_+ \setminus \mathcal N_{\rho_0}^+(\Gamma)
    = \{x\in\Omega_+ \colon \dist(x,\Gamma)\geq \rho_0\}.
\end{equation*}
By \eqref{eq:separation-away-free-boundary} there exists $c_{\rho_0}>0$
such that $w^*\geq c_{\rho_0}$ on $\Omega_{\mathrm{far}}$. We estimate
$\int_\Omega D_R(u^*,u) \diff x$ on each of the three regions separately.

\medskip
\noindent\textit{Step 3: Estimate on $\Omega_0$.}
On $\Omega_0$, $u^*=\phi$. By $\theta=1$ in \Cref{def:one-sided-bregman-growth} and $u - \phi \leq M_1$, we obtain
\begin{equation*}
    D_R(u^*,u) \leq C_{M_1} (u-u^*) = C_{M_1}\, e
    \qquad \text{a.e. on }\Omega_0.
\end{equation*}
Hence
\begin{equation}
\label{eq:coercivity-active}
    \int_{\Omega_0} D_R(u^*,u) \diff x
    \;\leq\; C_{M_1} \,\|e\|_{L^1(\Omega_0)}.
\end{equation}

\medskip
\noindent\textit{Step 4: Estimate on $\Omega_{\mathrm{far}}$.}
By \Cref{lem:delicate-bregman-bound},
\begin{equation*}
    D_R(u^*,u)
    \leq \widetilde C_{M_1}\,\frac{e^2}{w^*+e}
    \leq \widetilde C_{M_1}\,\frac{e^2}{w^*}
    \leq \widetilde C_{M_1}\,c_{\rho_0}^{-1}\, e^2
    \qquad \text{a.e. on }\Omega_{\mathrm{far}}.
\end{equation*}
Therefore
\begin{equation*}
    \int_{\Omega_{\mathrm{far}}} D_R(u^*,u) \diff x
    \leq \widetilde C_{M_1}\,c_{\rho_0}^{-1}\,\|e\|_{L^2(\Omega)}^2
    \leq \widetilde C_{M_1}\,c_{\rho_0}^{-1}\,C_P^2\,\|\nabla e\|_{L^2(\Omega)}^2,
\end{equation*}
where the last step uses Poincar\'e's inequality and $e\in H_0^1(\Omega)$. Thus
\begin{equation}
\label{eq:coercivity-far}
    \int_{\Omega_{\mathrm{far}}} D_R(u^*,u) \diff x
    \;\leq\; \widetilde C_{M_1}\,c_{\rho_0}^{-1}\,C_P^2\,\|\nabla e\|_{L^2(\Omega)}^2.
\end{equation}

\medskip
\noindent\textit{Step 5: Estimate on $\mathcal N_{\rho_0}^+(\Gamma)$.}
This is the delicate step. We first derive a pointwise bound in tubular
coordinates. By \Cref{lem:delicate-bregman-bound} and the quadratic separation
\eqref{eq:quadratic-separation-parametrized},
\begin{equation*}
    D_R(u^*,u)
    \leq \widetilde C_{M_1}\,\frac{e^2}{w^*+e}
    \leq \widetilde C_{M_1}\,\frac{e^2}{c_1^{-1}s^2+e}
    \leq C\,\frac{e^2}{s^2+e},
\end{equation*}
in tubular coordinates $x=\Psi(y,s)$, $(y,s)\in\Gamma\times(0,\rho_0)$,
where $C=\widetilde C_{M_1}/\min\{c_1^{-1},1\}$.

Then, using the tubular parametrization and the Jacobian bound
\eqref{eq:tubular-jacobian-comparable}, with $e(y,s)\coloneqq e(\Psi(y,s))$, we have
\begin{equation*}
    \int_{\mathcal N_{\rho_0}^+(\Gamma)} \frac{e^2}{s^2+e} \diff x
    \leq C_J \int_\Gamma \int_0^{\rho_0}
        \frac{e(y,s)^2}{s^2+e(y,s)} \diff s \diff \sigma(y).
\end{equation*}

By \eqref{eq:composition-regularity}, we obtain that $e\circ\Psi\in
L^2(\Gamma;H^1(0,\rho_0))$ and $e$ is non-negative.
\Cref{lem:one-dim-weighted} gives
\begin{equation*}
    \int_0^{\rho_0} \frac{e(y,s)^2}{s^2+e(y,s)}\,ds
    \leq C\left(|e(y,0)|^{3/2}
        + \int_0^{\rho_0} |\partial_s e(y,s)|^2\,ds\right).
\end{equation*}
Integrating over $\Gamma$ and using \eqref{eq:gradient-bound} after
reverting the change of variables,
\begin{equation}
\label{eq:coercivity-layer-intermediate}
    \int_{\mathcal N_{\rho_0}^+(\Gamma)} D_R(u^*,u) \diff x
    \leq C\left(
        \|e\|_{L^{3/2}(\Gamma)}^{3/2}
        + \|\nabla e\|_{L^2(\Omega)}^2
    \right).
\end{equation}

The trace term can be controlled by \Cref{lem:trace-interpolation}.
Since $\Gamma=\partial\Omega_0\cap\Omega$ and $e\geq 0$, the trace of $e$ on
$\Gamma$ can be estimated by the values of $e$ on $\Omega_0$. Specifically,
$e\in H^1(\Omega_0)$ with
\begin{equation*}
    \|e\|_{H^1(\Omega_0)}
    \leq \|u\|_{H^1(\Omega)}+\|u^*\|_{H^1(\Omega)}
    \leq M_2 + \|u^*\|_{H^1(\Omega)}
    \eqqcolon M',
\end{equation*}
and the trace of $e$ on $\Gamma$ from the $\Omega_0$ side coincides with
the trace from the $\Omega_+$ side. Applying \Cref{lem:trace-interpolation}
on $\intr(\Omega_0)$ to $e\geq 0$ with $H^1$-bound $M'$,
\begin{equation*}
    \|e\|_{L^{3/2}(\Gamma)}^{3/2}
    \leq C\bigl(\|e\|_{L^1(\Omega_0)} + \|\nabla e\|_{L^2(\Omega_0)}^2\bigr)
    \leq C\bigl(\|e\|_{L^1(\Omega_0)} + \|\nabla e\|_{L^2(\Omega)}^2\bigr).
\end{equation*}
Substituting into \eqref{eq:coercivity-layer-intermediate},
\begin{equation}
\label{eq:coercivity-layer}
    \int_{\mathcal N_{\rho_0}^+(\Gamma)} D_R(u^*,u) \diff x
    \leq C\bigl(
        \|e\|_{L^1(\Omega_0)} + \|\nabla e\|_{L^2(\Omega)}^2
    \bigr).
\end{equation}

\medskip
\noindent\textit{Step 6: Conclusion.}
Summing \eqref{eq:coercivity-active}, \eqref{eq:coercivity-far}, and
\eqref{eq:coercivity-layer},
\begin{equation}
\label{eq:coercivity-bregman-upper}
    \int_\Omega D_R(u^*,u) \diff x
    \;\leq\; C^*\bigl(
        \|e\|_{L^1(\Omega_0)} + \|\nabla e\|_{L^2(\Omega)}^2
    \bigr),
\end{equation}
for some constant $C^*=C^*(\Omega,A,u^*,\phi,M_2,C_{M_1},\widetilde C_{M_1},c_\phi,c_1,
C_J)>0$.

On the other hand, from \eqref{eq:coercivity-energy-lower},
\begin{equation*}
    E(u)-E(u^*)
    \geq \min\!\left(\frac{1}{2\Lambda},\,c_\phi\right)
        \bigl(\|\nabla e\|_{L^2(\Omega)}^2 + \|e\|_{L^1(\Omega_0)}\bigr).
\end{equation*}
Combining with \eqref{eq:coercivity-bregman-upper},
\begin{equation*}
    E(u)-E(u^*)
    \;\geq\; \sigma \int_\Omega D_R(u^*,u) \diff x,
    \qquad
    \sigma
    \coloneqq  \frac{\min((2\Lambda)^{-1},\,c_\phi)}{C^*}>0.
\end{equation*}
The constant $\sigma$ depends only on the quantities listed above,
all of which are independent of $u$. This completes the proof.
\end{proof}

\subsubsection{Proof of \Cref{thm:main-linear-convergence}}

\begin{proof}
By \Cref{prop:main-local-bregman-coercivity},
$s=1$ is an admissible minimality order with constant $\sigma$.
Hence \Cref{thm:energy-bregman-convergence} gives
\[
    B_k
    \leq
    (1+\alpha\sigma)^{-k}B_0
    =
    \varrho^kB_0,
\]
which proves \eqref{eq:main-linear-Ak}.

Applying \Cref{cor:energy-gap-H1-error-convergence} gives
\[
    E(u^k)-E(u^*)
    =
    \mathcal O\left((1+\alpha\sigma)^{-k}\right)
    =
    \mathcal O(\varrho^k)
\]
and
\[
    \|u^k-u^*\|_{H^1(\Omega)}
    =
    \mathcal O\left((1+\alpha\sigma)^{-k/2}\right)
    =
    \mathcal O(\varrho^{k/2}).
\]
This proves \eqref{eq:main-linear-energy} and
\eqref{eq:main-linear-H1}.
\end{proof}

\section{Conclusion}
\label{sec:conclusion}

We have developed a continuous-level theory for the Bregman proximal point
method applied to obstacle problems associated with coercive quadratic
energies. The framework accommodates a broad class of Legendre functions,
including the Shannon, Spence, Tsallis, and Kaniadakis entropies, and provides
a unified treatment of the semilinear elliptic subproblems and the resulting
iterative scheme.

For a general class of Legendre functions, we established the well-posedness
and strict feasibility of the proximal subproblems. We also derived several a
priori properties of the iterates, including positive invariance, a uniform
upper bound, energy dissipation, and uniform boundedness in $H^1(\Omega)$.
These estimates provide the structural control required for the convergence
theory.

The central mechanism of the convergence theory is sequential strict local minimality, whose order $s\geq1$ determines the convergence rate. Combined with
the one-step Bregman descent estimate, the case $s>1$ yields sublinear
convergence, whereas $s=1$ yields linear convergence. Assuming that the initial
iterate satisfies $u^0\geq u^*$ and a one-sided Bregman growth condition holds
with exponent $\theta\in(0,1]$, we established sequential strict local minimality
of order $s=2/\theta \geq 2$. Consequently,
\begin{equation*}
\int_\Omega D_R(u^*,u^k) \diff x
=
\mathcal O\left(k^{-\theta/(2-\theta)}\right),
\qquad
\|u^k-u^*\|_{H^1(\Omega)}
=
\mathcal O\left(k^{-1/(2-\theta)}\right).
\end{equation*}
For the Shannon and Tsallis entropies, we further established that the derived
sublinear convergence rates are sharp in a uniform worst-case sense. This was
shown by constructing admissible families of problems whose limiting problems
attain exactly the decay exponents predicted by the convergence theory.

Under the additional assumptions on the free boundary and obstacle, we
established sequential strict local minimality of order $s = 1$ for the Shannon
and Spence entropies. This strengthens the general sublinear estimates to
linear estimates:
\begin{equation*}
\int_\Omega D_R(u^*,u^k) \diff x
=
\mathcal O\left(
\left(1+\alpha\sigma\right)^{-k}
\right),
\qquad
\|u^k-u^*\|_{H^1(\Omega)}
=
\mathcal O\left(
\left(1+\alpha\sigma\right)^{-k/2}
\right).
\end{equation*}

\medskip

The analysis raises several questions that we leave for future work.
\begin{itemize}
    \item \emph{Sharpness of the assumptions.} The geometric assumptions on the free boundary are sufficient but possibly not necessary to permit the  minimality order $s = 1$. Likewise, although the initialization assumption, i.e., $u^0 \geq u^*$, is essential to the present analysis, numerical evidence suggests that it may not be necessary in practice. Identifying minimal geometric conditions and removing, or substantially weakening, the initialization assumption would clarify the full scope of both the sublinear and linear convergence regimes.

    \item \emph{Intermediate minimality orders.}
    The one-sided Bregman growth argument developed in this paper
    yields an admissible minimality order $s\geq2$,
    whereas additional assumptions yield sequential strict local minimality of order $s=1$. It remains open to identify conditions that yield admissible minimality orders in the intermediate regime
    $1<s<2$. Understanding this intermediate regime would help clarify the transition between sublinear and linear convergence.

    \item \emph{Discretization.} The present analysis is carried out at the
    continuous level. For the fully discrete proximal Galerkin scheme
    \cite{keith2024proximal}, an abstract a priori error analysis was recently
    established in \cite{keith2025priori}, combining the optimization error with
    the finite element discretization error in a unified framework. The
    continuous convergence estimates developed here suggest that these fully
    discrete rates may be further sharpened. In particular, establishing optimal
    mesh-independent sublinear convergence, linear convergence under some additional assumptions, and a rigorous
    discrete analogue of the sublinear-to-linear transition remain important
    directions for future work.

    \item \emph{Non-symmetric and non-energy variational inequalities.} Our
    analysis exploits the energy structure throughout. Extension to
    variational inequalities arising from non-symmetric bilinear forms
    \cite{fu2026proximal} or to variational inequalities without an
    underlying energy principle is a natural direction.
\end{itemize}

\appendix \crefalias{section}{appendix}

\section{Construction of an ordered initial iterate}
\label{app:ordered-initialization}

This section shows that the condition $u^0\geq u^*$ can be
enforced without prior knowledge of $u^*$. We first prove that every
admissible supersolution lies above $u^*$, and then construct such a
supersolution using only the obstacle $\phi$.

\begin{lemma}
\label{lem:supersolution-comparison}
Suppose that \(u^0\in K\) is a weak supersolution \cite[Definition 5.6]{kinderlehrer2000introduction}, that is,
\begin{equation}
\label{eq:super-solution}
    a(u^0,v)\geq 0
    \qquad
    \forall v\in H_0^1(\Omega),\quad v\geq 0,
\end{equation}
then
\[
    u^0\geq u^*
    \qquad \text{a.e. in }\Omega.
\]
\end{lemma}

\begin{proof}

Set \(z=(u^*-u^0)_+ \in H_0^1(\Omega)\). Since
    $u^*-z=\min\{u^*,u^0\}\in K$,
testing the variational inequality \eqref{eq:obstacle-vi} with $u^* - z$ and using $F = 0$ gives
\begin{equation}
\label{eq:VI-conseq}
    a(u^*,z)\leq 0.
\end{equation}
On the other hand, taking $v=z$ in \eqref{eq:super-solution} yields
\begin{equation}
\label{eq:super-solution-conseq}
    a(u^0,z)\geq 0.
\end{equation}
Therefore, combining \eqref{eq:VI-conseq} and \eqref{eq:super-solution-conseq} gives
\begin{equation}
    a(u^*- u^0,z)\leq 0.
\end{equation}
By \eqref{eq:uniform-ellipticity}, we obtain
\begin{equation}
    a(u^*-u^0,z)=a(z,z)\geq \frac{1}{\Lambda} \|\nabla z\|_{L^2(\Omega)}^2.
\end{equation}
Since $z\in H_0^1(\Omega)$, we obtain \(z=0\). Hence
\(u^*\leq u^0\) almost everywhere.
\end{proof}

\begin{proposition}[Construction of an ordered initial iterate]
Let
\(w\in C^2(\Omega)\cap C(\overline{\Omega})\) be the strong solution of
\begin{equation}
\label{eq:torsion-problem}
\begin{cases}
    -\nabla\!\cdot\!\bigl(A(x)\nabla w\bigr)=1
        &  \text{in }\Omega,\\
    w=0
        & \text{on }\partial\Omega.
\end{cases}
\end{equation}
Assume that \(\phi\in C(\overline{\Omega})\) satisfies \eqref{eq:boundary-compatibility}.
Then there exists \(C_0\geq1\) such that, for every
\(C\geq C_0\),
\[
    u^0 \coloneqq Cw\in K^\circ
    \qquad\text{and}\qquad
    u^0\geq u^*
    \quad\text{a.e. in }\Omega.
\]
\end{proposition}

\begin{proof}
By the strong maximum principle applied to
\eqref{eq:torsion-problem}, we obtain
\[
    w>0 \quad\text{in }\Omega.
\]
Since \(w=0\) on \(\partial\Omega\), we have
\[
    w-\phi > \delta_0
    \qquad\text{on }\partial\Omega.
\]
By continuity, there exists an open neighborhood \(U\) of
\(\partial\Omega\) such that
\[
    w-\phi\geq\frac{\delta_0}{2}
    \qquad\text{in }U\cap\Omega.
\]
Since \(w\geq0\), it follows that, for every \(C\geq1\),
\begin{equation}
\label{eq:estimate-U}
    Cw-\phi
    \geq
    w-\phi
    \geq\frac{\delta_0}{2}
    \qquad\text{in }U\cap\Omega.
\end{equation}

On the compact set
\begin{equation*}
    \Omega_U \coloneqq \overline{\Omega}\setminus U,
\end{equation*}
the strong positivity and continuity of \(w\) imply
\begin{equation*}
    m_U \coloneqq \min_{\Omega_U}w>0.
\end{equation*}
Consequently, choosing
\begin{equation*}
    C_0
    \coloneqq
    \max\left\{
        1,\,
        \frac{\max_{\Omega_U}(\phi+\frac{\delta_0}{2})_+}{m_U}
    \right\} < \infty,
\end{equation*}
we obtain, for every \(C\geq C_0\),
\begin{equation}
\label{eq:estimate-minus-U}
    Cw-\phi\geq\frac{\delta_0}{2}
    \qquad\text{in }\Omega_U.
\end{equation}
Combining \eqref{eq:estimate-U} and \eqref{eq:estimate-minus-U} gives
\begin{equation}
    Cw\geq\phi+\frac{\delta_0}{2}
    \qquad\text{in }\Omega,
\end{equation}
and hence \(u^0=Cw\in K^\circ\).

Moreover, for every nonnegative \(v\in H_0^1(\Omega)\),
\[
    a(u^0,v)
    =
    C\int_\Omega v\,dx
    \geq0.
\]
Thus \(u^0\) is a supersolution, and \Cref{lem:supersolution-comparison} yields
\[
    u^0\geq u^*
    \qquad\text{a.e. in }\Omega.
\]
This finishes the proof.
\end{proof}

\section{Sharpness}
\label{app:weak-sharpness-models}

In this appendix, we prove the sharpness of the sublinear convergence rates for the Shannon and Tsallis entropies
in \Cref{thm:main-sublinear-convergence}. We construct one-parameter families of strictly boundary-compatible
obstacle problems, indexed by $0<\delta\leq1$, whose convergence behavior
approaches that of the boundary-degenerate endpoint as
$\delta\downarrow0$.

The proof proceeds by first analyzing the boundary-degenerate endpoint $\delta=0$, where the proximal iterations admit explicit representations with
the predicted sublinear decay rates. We then show that, for every fixed finite number of iterations, the compatible problems with $\delta>0$ converge to the endpoint iteration as $\delta\downarrow0$.
Monotonicity properties of the proximal map provide the necessary conditions to pass to the limit.

Throughout this appendix, let $\Omega\subset\mathbb R^d$ be a bounded,
connected domain with smooth boundary, and fix $\alpha,\mu>0$.
We use the special case of \eqref{eq:energy-general-A} given by
\begin{equation*}
    A=I,
    \qquad
    F=0,
    \qquad
    g=0,
    \qquad
    \phi_\delta=-\delta,
    \qquad
    0\leq\delta\leq1.
\end{equation*}
Set
\begin{align*}
    & K_\delta
    =
    \{v\in H_0^1(\Omega) \mid
    v\geq-\delta\text{ a.e.\ in }\Omega\}, \\
    & K_\delta^\circ = \{v \in H_0^1(\Omega) \cap L^\infty(\Omega) \mid \essinf(v + \delta) > 0\}.
\end{align*}
The exact obstacle solution is
\begin{equation}
\label{eq:app-exact-zero}
    u_\delta^*=0
    \qquad
    \forall\delta\in[0,1].
\end{equation}
For every $\delta>0$, the strict boundary
compatibility condition in \eqref{eq:boundary-compatibility} holds.
Only the obstacle varies within each model family; in particular, the
initial iterate and the step size are fixed.

For either entropy considered below, write
\begin{equation*}
    \operatorname{Prox}_\delta(\tilde u)
    \coloneqq
    \operatorname*{argmin}_{v\in K_\delta}
    \left\{
        E(v)
        +
        \frac1\alpha
        \int_\Omega
        D_r(v+\delta,\tilde u+\delta)\diff x
    \right\},
\end{equation*}
where $\delta \in (0,1]$ and $\tilde {u} \in K_\delta^\circ$.
By \Cref{thm:entropic-pde}, the subproblem is well-posed and its solution satisfies
\begin{equation}
\label{eq:app-compatible-EL}
    \alpha(\nabla u,\nabla v)
    +
    \bigl(
        r'(u+\delta)-r'(\tilde u+\delta),
        v
    \bigr)
    =0
    \qquad
    \forall v\in H_0^1(\Omega).
\end{equation}
We set
\begin{equation}
\label{eq:app-compatible-iteration}
    u_\delta^0=u^0,
    \qquad
    u_\delta^k
    =
    \operatorname{Prox}_\delta(u_\delta^{k-1}),
    \qquad
    k\geq1.
\end{equation}
The choice of $r$ is understood from the model under consideration.
We use $d(x)=\operatorname{dist}(x,\partial\Omega)$ to denote the distance function to the boundary.

\subsection{Monotonicity properties}

Call $\widetilde u\in H_0^1(\Omega)$ a weak supersolution if
\begin{equation*}
    (\nabla \widetilde u,\nabla v)\geq0
    \qquad
    \forall v\in H_0^1(\Omega),\quad v\geq0.
\end{equation*}

\begin{lemma}[Three comparison properties]
\label{lem:app-prox-comparisons}
The following properties hold:

\begin{enumerate}
\item
If $\widetilde u_1,\widetilde u_2\in K_\delta^\circ$ and
$\widetilde u_1\leq\widetilde u_2$, then
\begin{equation}
\label{eq:app-prox-monotone-b}
    \operatorname{Prox}_\delta(\widetilde u_1)
    \leq
    \operatorname{Prox}_\delta(\widetilde u_2).
\end{equation}

\item
If $\widetilde u \in K_\delta^\circ$ is a weak supersolution and $0<\delta\leq1$, then
\begin{equation}
\label{eq:app-prox-below-b}
    0
    \leq
    \operatorname{Prox}_\delta(\widetilde u)
    \leq
    \widetilde u,
\end{equation}
and $\operatorname{Prox}_\delta(\widetilde u)$ is also a weak
supersolution.

\item
If $\widetilde u \in K_{\delta_1}^\circ \cap K_{\delta_2}^\circ$ is a weak supersolution and
$0<\delta_1\leq\delta_2\leq1$, then
\begin{equation}
\label{eq:app-prox-monotone-delta}
    \operatorname{Prox}_{\delta_1}(\widetilde u)
    \geq
    \operatorname{Prox}_{\delta_2}(\widetilde u).
\end{equation}
\end{enumerate}
\end{lemma}

\begin{proof}
For the first property, set
\[
    u_j=\operatorname{Prox}_\delta(\widetilde u_j),
    \qquad
    z=(u_1-u_2)_+ \in H_0^1(\Omega).
\]
Subtracting the two equations in \eqref{eq:app-compatible-EL} and testing
with $z$ gives
\begin{align}
    \alpha\|\nabla z\|_{L^2(\Omega)}^2
    +
    \int_\Omega
    \bigl[
        r'(u_1+\delta)-r'(u_2+\delta)
    \bigr]z\diff x
    =
    \int_\Omega
    \bigl[
        r'(\widetilde u_1+\delta)
        -
        r'(\widetilde u_2+\delta)
    \bigr]z\diff x.
\label{eq:app-comparison-b-proof}
\end{align}
The right-hand side in \eqref{eq:app-comparison-b-proof} is nonpositive because
$\widetilde u_1\leq\widetilde u_2$. On the support of $z$, one has
$u_1>u_2$, so the second term on the left in \eqref{eq:app-comparison-b-proof} is nonnegative. Hence
\[
    \alpha\|\nabla z\|_{L^2(\Omega)}^2\leq0,
\]
which implies $z=0$ almost everywhere and proves \eqref{eq:app-prox-monotone-b}.

For the second property, let
\[
    u=\operatorname{Prox}_\delta(\widetilde u),
    \qquad
    z=(u-\widetilde u)_+ \in H_0^1(\Omega).
\]
Testing \eqref{eq:app-compatible-EL} with $z$ gives
\begin{equation*}
    \alpha(\nabla u,\nabla z)
    =
    -
    \int_\Omega
    \bigl[
        r'(u+\delta)-r'(\widetilde u+\delta)
    \bigr]z\diff x
    \leq0.
\end{equation*}
On the other hand,
\begin{equation}
    (\nabla u,\nabla z)
    =
    (\nabla \widetilde u,\nabla z)
    +
    \|\nabla z\|_{L^2(\Omega)}^2
    \geq
    \|\nabla z\|_{L^2(\Omega)}^2,
\end{equation}
since $z\geq0$ and $\widetilde u$ is a weak supersolution. Thus $z=0$ almost everywhere, and
\[
    u\leq\widetilde u \quad \text{a.e.}
\]
The nonnegativity of $u$ follows directly from \Cref{prop:main-positive-invariance}.
Finally, for every nonnegative $v\in H_0^1(\Omega)$,
\eqref{eq:app-compatible-EL} and $u\leq\widetilde u$ give
\[
    \alpha(\nabla u,\nabla v)
    =
    \int_\Omega
    \bigl[
        r'(\widetilde u+\delta)-r'(u+\delta)
    \bigr]v\diff x
    \geq0.
\]
Thus $u$ is also a weak supersolution.

For the third property, set
\[
    u_j=\operatorname{Prox}_{\delta_j}(\widetilde u),
    \qquad
    z=(u_2-u_1)_+.
\]
By \eqref{eq:app-prox-below-b}, we have
\begin{equation}
    0\leq u_j\leq\widetilde u.
\end{equation}
For fixed $t$, define
\begin{equation}
    G(s,\eta)
    =
    r'(s+\eta)-r'(t+\eta).
\end{equation}
For $0\leq s\leq t$, since $r$ is convex, we have
\[
    \partial_sG(s,\eta)
    =
    r''(s+\eta) \geq 0,
\]
while
\[
    \partial_\eta G(s,\eta)
    =
    r''(s + \eta)
    -
    r''(t + \eta)
    \geq0,
\]
because $r''$ is decreasing for the Shannon and Tsallis entropies.

Subtracting \eqref{eq:app-compatible-EL} for $u_1$ and $u_2$ and testing with $z$
yields
\begin{equation}
\label{eq:app-delta-comparison-proof}
    \alpha\|\nabla z\|_{L^2(\Omega)}^2
    +
    \int_\Omega
    \bigl[
        G(u_2,\delta_2)
        -
        G(u_1,\delta_1)
    \bigr]z\diff x
    =0.
\end{equation}
On the support of $z$,
\[
    u_2>u_1,
    \qquad
    \delta_2\geq\delta_1,
\]
and therefore
\[
    G(u_2,\delta_2)
    \geq
    G(u_1,\delta_2)
    \geq
    G(u_1,\delta_1).
\]
The second term in \eqref{eq:app-delta-comparison-proof} is nonnegative, so
$z=0$. This proves \eqref{eq:app-prox-monotone-delta}.
\end{proof}

\begin{theorem}[Double monotonicity]
\label{thm:app-double-monotonicity}
Let $u^0 \in \displaystyle\bigcap_{\delta\in (0,1]} K_\delta^\circ$ be a weak
supersolution. Then
\begin{equation}
\label{eq:app-k-monotonicity}
    0
    \leq
    u_\delta^{k+1}
    \leq
    u_\delta^k
    \leq
    u^0
    \qquad
    \forall k\geq0,\quad 0<\delta\leq1.
\end{equation}
Moreover, if
\[
    0<\delta_1\leq\delta_2\leq1,
\]
then
\begin{equation}
\label{eq:app-delta-monotonicity}
    u_{\delta_1}^k
    \geq
    u_{\delta_2}^k
    \qquad
    \forall k\geq0.
\end{equation}
Thus the iterates are decreasing in $k$ and increasing as
$\delta\downarrow0$.
\end{theorem}

\begin{proof}
By \eqref{eq:app-prox-monotone-b} and \eqref{eq:app-prox-below-b} of \Cref{lem:app-prox-comparisons}, we inductively obtain
\[
    0\leq u_\delta^k\leq u_\delta^{k-1}
\]
and every $u_\delta^k$ is a weak supersolution. This proves \eqref{eq:app-k-monotonicity}.

To prove monotonicity in $\delta$, we argue by induction in $k$. At
$k=0$, the iterates agree. Suppose that
\[
    u_{\delta_1}^{k-1}
    \geq
    u_{\delta_2}^{k-1}.
\]
Then \eqref{eq:app-prox-monotone-b} and \eqref{eq:app-prox-monotone-delta} of
\Cref{lem:app-prox-comparisons} give
\begin{align*}
    u_{\delta_2}^k
    &=
    \operatorname{Prox}_{\delta_2}
        (u_{\delta_2}^{k-1})
    \leq
    \operatorname{Prox}_{\delta_2}
        (u_{\delta_1}^{k-1})
    \leq
    \operatorname{Prox}_{\delta_1}
        (u_{\delta_1}^{k-1})
    =
    u_{\delta_1}^k.
\end{align*}
This proves \eqref{eq:app-delta-monotonicity}.
\end{proof}

\subsection{The Shannon entropy}

In this section, we take $r = r_{\mathrm{sh}}$. Let $w$ solve
\begin{equation}
\label{eq:app-shannon-profile}
    -\Delta w=\mu
    \quad\text{in }\Omega,
    \qquad
    w=0
    \quad\text{on }\partial\Omega.
\end{equation}
Standard elliptic PDE theory \cite[Theorem~3.5 and Theorem~6.14]{gilbarg2001elliptic} gives
\begin{equation}
\label{eq:app-shannon-profile-bound}
    w\in C^1(\overline\Omega)\cap H_0^1(\Omega),
    \qquad
    w > 0 \quad \text{in} \; \Omega.
\end{equation}
Nevertheless, the boundary singularity of $\log w$ is mild enough that
\begin{equation}
    \log w\in L^p(\Omega) \quad \forall 1 \leq p < \infty,
\end{equation}
see, for example, \cite[Lemma~3.4 and Theorem~6.14]{gilbarg2001elliptic}.
For the remainder of this subsection, we choose the initial iterate $u^0=c_0w$ for some constant $c_0 > 0$, which means $u^0_\delta = c_0w$ for all $\delta \in [0,1]$.

\begin{proposition}[Boundary-degenerate Shannon model]
\label{thm:app-shannon-model}
There exists a unique positive sequence $\{c_k\}_{k\geq0}$, with the
prescribed $c_0>0$, satisfying
\begin{equation}
\label{eq:app-shannon-recurrence}
    \alpha\mu c_k
    +
    \log c_k
    -
    \log c_{k-1}
    =0,
    \qquad
    k\geq1.
\end{equation}
The functions
\begin{equation*}
    u_0^k=c_kw
\end{equation*}
satisfy the endpoint equation, which is \eqref{eq:app-compatible-EL} with $\delta = 0$,
\begin{equation}
\label{eq:app-shannon-endpoint-EL}
    \alpha(\nabla u_0^k,\nabla v)
    +
    \bigl(
        \log u_0^k-\log u_0^{k-1},
        v
    \bigr)
    =0
    \qquad
    \forall v\in H_0^1(\Omega).
\end{equation}
Moreover, it holds that
\begin{align}
    \label{eq:app-shannon-bregman-c-k}
    & c_k\sim(\alpha\mu k)^{-1}, \\
    \label{eq:app-shannon-asymptotic}
    & k\|u_0^k\|_{H^1(\Omega)}
    \to
    L_{\rm sh}
    \coloneqq
    \frac{\|w\|_{H^1(\Omega)}}{\alpha\mu}
    >0,\\
    \label{eq:app-shannon-bregman-asymptotic}
    & k
    \int_\Omega
    D_r(u_0^*,u_0^k)
    \diff x
    \to
    L_{\rm sh}^{D}
    \coloneqq
    \frac{1}{\alpha\mu}
    \int_\Omega w\diff x
    >0.
\end{align}
\end{proposition}

\begin{proof}
For fixed $c_{k-1}>0$, the function $g(c) = \alpha\mu c+\log c-\log c_{k-1}$
is strictly increasing. Hence
\eqref{eq:app-shannon-recurrence} determines a unique $c_k>0$.

Using \eqref{eq:app-shannon-profile}, we have
\begin{equation}
    (\nabla w,\nabla v)
    =
    \mu\int_\Omega v\diff x.
\end{equation}
Thus, for every $v\in H_0^1(\Omega)$,
\begin{align*}
&\alpha(\nabla(c_kw),\nabla v)
+
\bigl(
    \log(c_kw)-\log(c_{k-1}w),
    v
\bigr)
\\
&\qquad
=
\left[
    \alpha\mu c_k
    +
    \log c_k-\log c_{k-1}
\right]
\int_\Omega v\diff x
=
0.
\end{align*}
This proves \eqref{eq:app-shannon-endpoint-EL}.

The recurrence \eqref{eq:app-shannon-recurrence} is equivalent to $c_{k-1}
    =
    c_ke^{\alpha\mu c_k}$.
Hence
\[
    0<c_k<c_{k-1} \quad \text{and} \quad c_k \to 0.
\]
Furthermore,
\begin{align*}
\frac1{c_k}-\frac1{c_{k-1}}=
\frac1{c_k}
\left(
    1-e^{-\alpha\mu c_k}
\right)
\to
\alpha\mu.
\end{align*}
Therefore, by the Stolz--Ces\`aro theorem \cite[Chapter 3, Theorem 1.22]{muresan2009concrete}, we obtain
\[
    \frac{c_k^{-1}}{k} \to
    \alpha\mu.
\]
This proves
\eqref{eq:app-shannon-bregman-c-k}. Since $u_0^* = 0$, \eqref{eq:app-shannon-asymptotic} and \eqref{eq:app-shannon-bregman-asymptotic} follow immediately.
\end{proof}

\begin{remark}[Variational interpretation of the endpoint iterate]
    Although the boundary-degenerate case $\delta=0$ falls outside the
    hypotheses of \Cref{thm:entropic-pde}, the variational interpretation of
    the proximal step remains valid. In particular, for every $k\geq1$,
    the functions $\{u_0^k\}$ constructed in
    \Cref{thm:app-shannon-model} is the unique minimizer of
    \begin{equation*}
        v\mapsto
        E(v)
        +
        \frac1\alpha
        \int_\Omega
        D_{r}(v,u_0^{k-1})
        \diff x \quad \text{over } K_0.
    \end{equation*}
    Indeed, the endpoint equation
    \eqref{eq:app-shannon-endpoint-EL} is precisely the first-order optimality
    condition, while strict convexity gives uniqueness.
\end{remark}

\begin{lemma}[Convergence of the Shannon entropy variables]
\label{lem:app-shannon-entropy-limit}
For every fixed $k\geq0$, the pointwise limit
\begin{equation}
\label{eq:app-shannon-pointwise-limit}
    \bar u^k
    \coloneqq \lim_{\delta\downarrow0} u_\delta^k
\end{equation}
exists a.e. in $\Omega$. There is also $a_k>0$, independent of
$0<\delta\leq1$, such that
\begin{equation}
\label{eq:app-uniform-boundary-barrier}
    a_kd
    \leq
    u_\delta^k
    \leq
    u^0,
    \qquad
    a_kd
    \leq
    \bar u^k
    \leq
    u^0.
\end{equation}
Consequently,
\begin{equation}
\label{eq:app-shannon-entropy-limit}
    \log(u_\delta^k+\delta)
    \to
    \log\bar u^k
    \qquad
    \text{in }L^2(\Omega).
\end{equation}
\end{lemma}

\begin{proof}
The existence of the limit in
\eqref{eq:app-shannon-pointwise-limit} follows directly from the monotone convergence theorem and \eqref{eq:app-delta-monotonicity}.

For $k=0$, the lower bound in
\eqref{eq:app-uniform-boundary-barrier} follows from
\eqref{eq:app-shannon-profile-bound}. Let $k\geq1$. Since
$0<\delta\leq1$, \eqref{eq:app-delta-monotonicity} gives
\begin{equation}
\label{eq:app-u-delta-above-u1}
    u_\delta^k\geq u_1^k.
\end{equation}
By \eqref{eq:app-compatible-EL},
\begin{equation}
\label{eq:app-shannon-u1-weak}
    (\nabla u_1^k,\nabla v)
    =
    (f_k,v)
    \qquad
    \forall v\in H_0^1(\Omega),
\end{equation}
where
\begin{equation*}
    f_k
    \coloneqq
    \frac1\alpha
    \left[
        \log(1+u_1^{k-1})
        -
        \log(1+u_1^k)
    \right].
\end{equation*}
From \eqref{eq:app-k-monotonicity}, we have $0 \leq u_1^k \leq u_1^{k-1} \leq u^0$. Hence
\begin{equation*}
    0
    \leq f_k
    \leq \frac{1}{\alpha} \log\!\left(1+\|u^0\|_{L^\infty(\Omega)}\right),
\end{equation*}
so $f_k\in L^\infty(\Omega)$.

Moreover, $f_k\not\equiv0$. Otherwise
\eqref{eq:app-shannon-u1-weak} would give $u_1^k=0$, which implies $u_1^{k-1}=0$. Thus, $u^0_1 = u^0=0$, which gives a
contradiction.

Hence, by global elliptic regularity \cite[Theorem~9.15]{gilbarg2001elliptic}, we obtain
\[
    u_1^k\in W^{2,p}(\Omega)
    \qquad
    \forall 1<p<\infty,
\]
and therefore
\begin{equation}
\label{eq:app-shannon-u1-strong}
    -\Delta u_1^k=f_k
    \qquad\text{a.e. in }\Omega.
\end{equation}
Taking $p>d$ gives
\[
    u_1^k\in C^{1,\gamma}(\overline\Omega)
\]
for some $\gamma>0$. The strong maximum principle \cite[Theorem 8.19]{gilbarg2001elliptic} and the Hopf boundary
point lemma \cite[Lemma 3.4]{gilbarg2001elliptic} applied to
\eqref{eq:app-shannon-u1-strong} yield
\[
    u_1^k>0
    \quad\text{in }\Omega,
    \qquad
    -\partial_\nu u_1^k>0
    \quad\text{on }\partial\Omega.
\]
Since $\partial\Omega$ is compact and $\nabla u_1^k$ is continuous, we obtain
\[
    u_1^k(x)\geq a_kd(x)
    \qquad
    x\in\Omega
\]
for some $a_k>0$. Together with
\eqref{eq:app-u-delta-above-u1}, this proves
\eqref{eq:app-uniform-boundary-barrier}.

Finally,
\[
    a_kd
    \leq
    u_\delta^k+\delta
    \leq
    \|u^0\|_{L^\infty(\Omega)}+1.
\]
Therefore
\begin{equation*}
    |\log(u_\delta^k+\delta)|
    + |\log\bar u^k|
    \leq C_k (1+|\log d|).
\end{equation*}
Since $|\log d|\in L^2(\Omega)$ and
\[
    \log(u_\delta^k+\delta)
    \to
    \log\bar u^k
    \qquad
    \text{a.e.},
\]
the dominated convergence theorem proves
\eqref{eq:app-shannon-entropy-limit}.
\end{proof}

\begin{lemma}[$H^1$ and Bregman stability for Shannon entropy]
\label{lem:app-shannon-finite-horizon}
The limit $\bar u^k$ defined in
\Cref{lem:app-shannon-entropy-limit} satisfies
$\bar u^k=u_0^k$ for every $k\geq0$. Moreover,
for every fixed integer $K\geq0$,
\begin{equation}
\label{eq:app-shannon-finite-horizon}
    \max_{0\leq k\leq K}
    \|u_\delta^k-u_0^k\|_{H^1(\Omega)}
    \to 0
    \qquad
    \text{as }\delta\downarrow0,
\end{equation}
and
\begin{equation}
\label{eq:app-shannon-finite-horizon-bregman}
    \max_{0\leq k\leq K}
    \left|
    \int_\Omega
    D_r(u_\delta^*+\delta,u_\delta^k+\delta)
    \diff x
    -
    \int_\Omega
    D_r(u_0^*,u_0^k)
    \diff x
    \right|
    \to 0
    \qquad
    \text{as }\delta\downarrow0.
\end{equation}
\end{lemma}

\begin{proof}
By the energy dissipation law in \Cref{lem:energy-dissipation}, we obtain
\begin{equation}
\label{eq:app-energy-bound}
E(u_\delta^k)\le E(u_\delta^{k-1})
\le \cdots \le E(u^0),
\end{equation}

For each fixed $k$, the monotone convergence in
\Cref{lem:app-shannon-entropy-limit} and the bound
$0\leq u_\delta^k\leq u^0$ imply
\begin{equation*}
    u_\delta^k
    \to \bar{u}^k
    \qquad \text{in } L^2(\Omega).
\end{equation*}
The energy bound in \eqref{eq:app-energy-bound} gives
\begin{equation}
\label{eq:app-shannon-weak-H1}
    u_\delta^k
    \rightharpoonup
    \bar u^k
    \qquad
    \text{in }H^1_0(\Omega).
\end{equation}

We then prove the identity $\bar u^k=u_0^k$ by induction. The identity is immediate for
$k=0$. Suppose $\bar u^{k-1}=u_0^{k-1}$.
Passing to the limit in \eqref{eq:app-compatible-EL} using
\eqref{eq:app-shannon-entropy-limit} and \eqref{eq:app-shannon-weak-H1} gives
\begin{equation}
\label{eq:app-shannon-limit-equation}
    \alpha(\nabla\bar u^k,\nabla v)
    +
    \bigl(
        \log\bar u^k-\log u_0^{k-1},
        v
    \bigr)
    =0
    \qquad
    \forall v\in H_0^1(\Omega).
\end{equation}
On the other hand, $u_0^k=c_kw$ satisfies
\eqref{eq:app-shannon-endpoint-EL}. Subtracting
\eqref{eq:app-shannon-endpoint-EL} from
\eqref{eq:app-shannon-limit-equation} and taking
    $v=\bar u^k-u_0^k$
gives
\begin{align*}
\alpha
\|\nabla(\bar u^k-u_0^k)\|_{L^2(\Omega)}^2
&+
\int_\Omega \left(\log \bar{u}^k-\log u_0^k \right) \left(\bar{u}^k-u_0^k \right) \diff x
=0.
\end{align*}
Hence, we obtain $\bar u^k=u_0^k$.

It remains to upgrade the convergence to strong $H^1$. Let $e_\delta^k = u_\delta^k-u_0^k$.
Subtracting \Cref{eq:app-shannon-endpoint-EL} from \eqref{eq:app-compatible-EL} and testing with $e_\delta^k$ gives
\begin{align}
\alpha \|\nabla e_\delta^k\|_{L^2}^2
+
\left(
\log(u_\delta^k+\delta)-\log u_0^k,
e_\delta^k
\right)
=
\left(
\log(u_\delta^{k-1}+\delta)-\log u_0^{k-1},
e_\delta^k
\right).
\label{eq:app-stability-dual-estimate}
\end{align}
Therefore,
\begin{equation*}
\alpha \|\nabla e_\delta^k\|_{L^2}^2
\leq
\Bigl(
\|\log(u_\delta^k+\delta)-\log u_0^k\|_{L^2(\Omega)}
+ \|\log(u_\delta^{k-1}+\delta)-\log u_0^{k-1}\|_{L^2(\Omega)}
\Bigr)
\|e_\delta^k\|_{L^2(\Omega)}.
\end{equation*}
By \Cref{lem:app-shannon-entropy-limit} and Poincar\'e's inequality, taking the maximum over the finite set
$\{0,\ldots,K\}$ proves
\eqref{eq:app-shannon-finite-horizon}.

It remains to prove \eqref{eq:app-shannon-finite-horizon-bregman}. Since
\(u_\delta^*=0\), we have
\begin{equation}
\int_\Omega
D_r(u_\delta^*+\delta,u_\delta^k+\delta)
\diff x
=
|\Omega|r(\delta)
-
\int_\Omega
r(u_\delta^k+\delta)
\diff x
+
\bigl(
    \log(u_\delta^k+\delta),
    u_\delta^k
\bigr).
\label{eq:app-shannon-bregman-representation}
\end{equation}
The continuity of $r$ at zero gives $r(\delta)
    \to
    r(0)$.
By \Cref{lem:app-shannon-entropy-limit} and
\eqref{eq:app-shannon-finite-horizon}, we have
\[
    u_\delta^k+\delta
    \to
    u_0^k
    \qquad
    \text{a.e. in }\Omega.
\]
Moreover, we have $0\leq u_\delta^k+\delta
    \leq
    \|u^0\|_{L^\infty(\Omega)}+1$.
Since $r$ is continuous, the dominated convergence theorem gives
\[
    r(u_\delta^k+\delta)
    \to
    r(u_0^k)
    \qquad
    \text{in }L^1(\Omega).
\]
For the last term in
\eqref{eq:app-shannon-bregman-representation}, we have
\begin{align*}
\left|
\bigl(
    \log(u_\delta^k+\delta),
    u_\delta^k
\bigr)
-
\bigl(
    \log u_0^k,
    u_0^k
\bigr)
\right|
\leq&
\|\log(u_\delta^k+\delta)-\log u_0^k\|_{L^2(\Omega)}
\|u_\delta^k\|_{L^2(\Omega)}
\\
&+
\|\log u_0^k\|_{L^2(\Omega)}
\|u_\delta^k-u_0^k\|_{L^2(\Omega)}
\to 0,
\end{align*}
where we used
\eqref{eq:app-shannon-entropy-limit},
\eqref{eq:app-shannon-finite-horizon}, and the uniform
\(L^2(\Omega)\)-boundedness of \(u_\delta^k\).
Passing to the limit in
\eqref{eq:app-shannon-bregman-representation}, we conclude that
\[
\int_\Omega
D_r(u_\delta^*+\delta,u_\delta^k+\delta)
\diff x
\to
\int_\Omega
D_r(u_0^*,u_0^k)
\diff x.
\]
Since this holds for every \(k\in\{0,\ldots,K\}\), taking the maximum
over this finite set proves
\eqref{eq:app-shannon-finite-horizon-bregman}.
\end{proof}

\begin{theorem}[Sharpness for Shannon entropy]
\label{thm:app-shannon-uniform-sharpness}
For $\{u_\delta^k\}$ defined by
\eqref{eq:app-compatible-iteration}, it holds that
\begin{equation}
\label{eq:app-shannon-uniform-sharpness}
    \liminf_{k\to\infty}
    k
    \sup_{0<\delta\leq1}
    \|u_\delta^k-u_\delta^*\|_{H^1(\Omega)}
    \geq
    L_{\rm sh}
    >0,
\end{equation}
and
\begin{equation}
\label{eq:app-shannon-bregman-sharpness}
    \liminf_{k\to\infty}
    k
    \sup_{0<\delta\leq1}
    \int_\Omega
        D_r(u_\delta^*+\delta,u_\delta^k+\delta)
    \diff x
    \geq
    L^D_{\rm sh}
    >0.
\end{equation}
\end{theorem}

\begin{proof}
For every fixed $k$, \eqref{eq:app-exact-zero} and
\eqref{eq:app-shannon-finite-horizon} give
\begin{align*}
\sup_{0<\delta\leq1}
\|u_\delta^k-u_\delta^*\|_{H^1(\Omega)}
=
\sup_{0<\delta\leq1}
\|u_\delta^k\|_{H^1(\Omega)}
\geq
\lim_{\delta\downarrow0}
\|u_\delta^k\|_{H^1(\Omega)}
=
\|u_0^k\|_{H^1(\Omega)}.
\end{align*}
By \eqref{eq:app-shannon-finite-horizon-bregman}, for every fixed \(k\),
\begin{align*}
\sup_{0<\delta\leq1}
\int_\Omega
    D_r(u_\delta^*+\delta,u_\delta^k+\delta)
\diff x
\geq
\lim_{\delta\downarrow0}
\int_\Omega
    D_r(u_\delta^*+\delta,u_\delta^k+\delta)
\diff x
=
\int_\Omega D_r(u_0^*,u_0^k)\diff x.
\end{align*}
Multiplying by $k$, taking the lower limit, and using
\eqref{eq:app-shannon-asymptotic} and \eqref{eq:app-shannon-bregman-asymptotic} gives
\eqref{eq:app-shannon-uniform-sharpness} and \eqref{eq:app-shannon-bregman-sharpness}.
\end{proof}

In particular, as $\delta\downarrow0$, the constructed iterates converge on
every fixed finite horizon to the explicitly solvable boundary-degenerate
endpoint, whose $H^1$ error and Bregman error both have the exact algebraic
order $k^{-1}$. Consequently, \Cref{thm:app-shannon-uniform-sharpness} shows
that the $\mathcal O(k^{-1})$ rates are sharp in the uniform worst-case sense
over this family. More precisely, there do not exist a sequence
$\rho_k=o(k^{-1})$ and a constant $C>0$, independent of $\delta$, such that
\begin{equation*}
    \|u_\delta^k-u_\delta^*\|_{H^1(\Omega)}
    \leq C\rho_k
\end{equation*}
for all $0<\delta\leq1$ and all sufficiently large $k$. The same conclusion
holds for the corresponding Bregman error.

\subsection{The Tsallis entropy}

Let $1<q<2$ and take $r = r_q$, for which $r_q'(s)=\frac{s^{1-q}-1}{1-q}$. For
the remainder of this subsection, because the entropy variable $r'(u_0^k)$
generally belongs only to \(H^{-1}(\Omega)\), the Bregman divergence is
interpreted via the \(H^{-1}(\Omega)\)--\(H_0^1(\Omega)\) duality pairing,
i.e.,
\begin{equation}
\label{eq:app-tsallis-dual-bregman}
    \int_\Omega D_{r}(v,u) \diff x
    \coloneqq
    \int_\Omega
        r(v)-r(u)
    \diff x
    -
    \left\langle
        r'(u),v-u
    \right\rangle_{H^{-1}(\Omega),H_0^1(\Omega)}.
\end{equation}
The initial iterate is $u^0=c_0w$, for some constant $c_0 > 0$,
where $w$ is the positive solution of
\begin{equation}
\label{eq:weak-tsallis-profile}
    -\Delta w
    =
    \mu w^{1-q}
    \quad\text{in }\Omega,
    \qquad
    w=0
    \quad\text{on }\partial\Omega.
\end{equation}
The singular elliptic theory for
\eqref{eq:weak-tsallis-profile} gives
\[
    w\in H_0^1(\Omega)\cap L^\infty(\Omega),
    \qquad
    w\geq a_0d;
\]
see \cite{crandall1977dirichlet,lazer1991singular}.
We can show that
\begin{equation*}
    r'(w) = \frac{w^{1-q} - 1}{1-q}\in H^{-1}(\Omega).
\end{equation*}
Indeed, for every $v\in H_0^1(\Omega)$,
\begin{align*}
\left|
    \int_\Omega w^{1-q}v\diff x
\right|
&\leq
C
\int_\Omega d^{1-q}|v|\diff x
\\
&\leq
C
\|d^{2-q}\|_{L^2(\Omega)}
\left\|
    \frac{v}{d}
\right\|_{L^2(\Omega)}
\\
&\leq
C
\|\nabla v\|_{L^2(\Omega)},
\end{align*}
where Hardy's inequality is applied in the last step.
Since $u^0$ is a weak supersolution,
\Cref{thm:app-double-monotonicity} also applies in this section.

\begin{proposition}[Boundary-degenerate Tsallis model]
\label{thm:app-tsallis-model}
There exists a unique positive sequence $\{c_k\}_{k\geq0}$ satisfying
\begin{equation}
\label{eq:app-tsallis-recurrence}
    c_k^{1-q}
    -
    c_{k-1}^{1-q}
    =
    (q-1)\alpha\mu c_k.
\end{equation}
The functions
\begin{equation*}
    u_0^k=c_kw
\end{equation*}
satisfy
\begin{equation}
\label{eq:app-tsallis-endpoint-EL}
    \alpha(\nabla u_0^k,\nabla v)
    +
    \left\langle
        r_q'(u_0^k)-r_q'(u_0^{k-1}),
        v
    \right\rangle_{H^{-1},H_0^1}
    =0, \quad \forall v \in H_0^1(\Omega).
\end{equation}
Moreover, it holds
\begin{align}
    \label{eq:app-tsallis-asymptotic-c-k}
    & c_k
    \sim
    (q\alpha\mu k)^{-1/q},\\
    \label{eq:app-tsallis-asymptotic}
    & k^{1/q}
    \|u_0^k\|_{H^1(\Omega)}
    \to
    L_q
    \coloneqq
    (q\alpha\mu)^{-1/q} \|w\|_{H^1(\Omega)}
    >0,\\
    \label{eq:app-tsallis-bregman-asymptotic}
    & k^{(2-q)/q}
    \int_\Omega
    D_r(u_0^*,u_0^k)
    \diff x
    \to
    L_q^D
    \coloneqq
    \frac{(q\alpha\mu)^{-(2-q)/q}}{2-q}
    \int_\Omega w^{2-q}\diff x
    >0.
\end{align}
\end{proposition}

\begin{proof}
For fixed $c_{k-1}>0$, since $F(c)=c^{1-q}-(q-1)\alpha\mu c-c_{k-1}^{1-q}$ is decreasing, $F(0^+) = + \infty$ and $F(c_{k-1}) < 0$, \eqref{eq:app-tsallis-recurrence} determines a unique $c_k$. Moreover, it holds that
\[
    0 < c_k < c_{k-1} \quad \text{and} \quad c_k \to 0.
\]

Using \eqref{eq:weak-tsallis-profile} and the identity
\begin{align*}
r_q'(c_kw)-r_q'(c_{k-1}w)
&=
\frac{
    c_k^{1-q}-c_{k-1}^{1-q}
}{1-q}
w^{1-q} \in H^{-1}(\Omega),
\end{align*}
we obtain
\begin{align*}
&\alpha(\nabla(c_kw),\nabla v)
+
\left\langle
    r_q'(c_kw)-r_q'(c_{k-1}w),
    v
\right\rangle_{H^{-1},H_0^1}
\\
&\qquad
=
\left[
    \alpha\mu c_k
    +
    \frac{
        c_k^{1-q}-c_{k-1}^{1-q}
    }{1-q}
\right]
\int_\Omega w^{1-q}v\diff x
=
0 \quad \forall v \in H_0^1(\Omega),
\end{align*}
which proves \eqref{eq:app-tsallis-endpoint-EL}.

The recurrence \eqref{eq:app-tsallis-recurrence} implies that
\begin{align}
    c_k^{-q} - c_{k-1}^{-q} = c_{k}^{-q} \left( 1-(1-(q-1)\alpha \mu c_{k}^q)^{\frac{-q}{1-q}} \right) \to q\alpha \mu.
\end{align}
Therefore, by the Stolz--Ces\`aro theorem, we obtain
\begin{equation*}
    \frac{c_k^{-q}}{k} \to q\alpha \mu,
\end{equation*}
which proves \eqref{eq:app-tsallis-asymptotic-c-k}. Since $u_0^* = 0$, \eqref{eq:app-tsallis-asymptotic} and \eqref{eq:app-tsallis-bregman-asymptotic} follow immediately.
\end{proof}

As in the Shannon case, although $\delta=0$ is not covered by
\Cref{thm:entropic-pde}, the function $u_0^k$ in \Cref{thm:app-tsallis-model}
is still the unique minimizer of the corresponding boundary-degenerate proximal
functional over $K_0$, in the $H^{-1}(\Omega)$--$H_0^1(\Omega)$ sense.

\begin{lemma}[Convergence of the Tsallis entropy variables]
\label{lem:app-tsallis-entropy-limit}
For every fixed $k \geq 0$,
the pointwise limit
\begin{equation}
\label{eq:app-tsallis-pointwise-limit}
    \bar u^k
    \coloneqq
    \lim_{\delta\downarrow0}u_\delta^k
\end{equation}
exists a.e. in $\Omega$. There is also $a_k>0$, independent of
$0<\delta\leq1$, such that
\begin{equation}
\label{eq:app-tsallis-uniform-boundary-barrier}
    a_kd
    \leq
    u_\delta^k
    \leq
    u^0,
    \qquad
    a_kd
    \leq
    \bar u^k
    \leq
    u^0.
\end{equation}
Consequently,
\begin{equation}
\label{eq:app-tsallis-entropy-limit}
    r_q'(u_\delta^k+\delta)
    \to
    r_q'(\bar u^k)
    \qquad
    \text{in }H^{-1}(\Omega).
\end{equation}
\end{lemma}

\begin{proof}
The existence of the limit in
\eqref{eq:app-tsallis-pointwise-limit} follows directly from the monotone convergence theorem and \eqref{eq:app-delta-monotonicity}.

The proof of \eqref{eq:app-tsallis-uniform-boundary-barrier} follows the same argument as the proof of \eqref{eq:app-uniform-boundary-barrier} in
\Cref{lem:app-shannon-entropy-limit}. Indeed,
for $k\geq1$,
\begin{equation}
\label{eq:app-tsallis-u1-equation}
    -\Delta u_1^k
    =
    \frac1\alpha
    \left[
        r_q'(1+u_1^{k-1})
        -
        r_q'(1+u_1^k)
    \right].
\end{equation}
Because $0 \leq u_1^k \leq u_1^{k-1} \leq u^0$,
the right-hand side of
\eqref{eq:app-tsallis-u1-equation} is bounded, nonnegative, and nonzero.
The same argument yields
\[
    a_kd \leq u_1^k \leq u_\delta^k \leq u^0.
\]

It remains to prove $H^{-1}$ convergence. First
\[
    r_q'(u_\delta^k+\delta)
    -
    r_q'(\bar u^k) \to 0
    \qquad
    \text{a.e. in }\Omega.
\]
By \eqref{eq:app-tsallis-uniform-boundary-barrier}, we get
\begin{equation}
\label{eq:app-tsallis-h-bound}
    |r_q'(u_\delta^k+\delta)
    -
    r_q'(\bar u^k)|
    \leq
    C_k
    \bigl(
        1+d^{1-q}
    \bigr).
\end{equation}
For every $v\in H_0^1(\Omega)$, Hardy's inequality gives
\begin{align*}
|\langle r_q'(u_\delta^k+\delta)
    -
    r_q'(\bar u^k),v\rangle|
&\leq
\int_\Omega
d|r_q'(u_\delta^k+\delta)
    -
    r_q'(\bar u^k)|
\frac{|v|}{d}
\diff x
\nonumber\\
&\leq
\|d (r_q'(u_\delta^k+\delta)
    -
    r_q'(\bar u^k))\|_{L^2(\Omega)}
\left\|
    \frac{v}{d}
\right\|_{L^2(\Omega)}
\nonumber\\
&\leq
C_H
\|d (r_q'(u_\delta^k+\delta)
    -
    r_q'(\bar u^k))\|_{L^2(\Omega)}
\|\nabla v\|_{L^2(\Omega)}.
\end{align*}
Thus
\[
    \|r_q'(u_\delta^k+\delta)
    -
    r_q'(\bar u^k)\|_{H^{-1}(\Omega)}
    \lesssim
    \|d (r_q'(u_\delta^k+\delta)
    -
    r_q'(\bar u^k))\|_{L^2(\Omega)}.
\]
By \eqref{eq:app-tsallis-h-bound}, we obtain
\[
    |d (r_q'(u_\delta^k+\delta)
    -
    r_q'(\bar u^k))|
    \leq
    C_k
    \bigl(
        d+d^{2-q}
    \bigr) \in L^2(\Omega),
\]
because
$2-q>0$. The dominated convergence theorem therefore gives
\[
    \|d (r_q'(u_\delta^k+\delta)
    -
    r_q'(\bar u^k))\|_{L^2(\Omega)}
    \to 0,
\]
which proves \eqref{eq:app-tsallis-entropy-limit}.
\end{proof}

\begin{lemma}[$H^1$ and Bregman stability for Tsallis entropy]
The limit $\bar u^k$ defined in
\Cref{lem:app-tsallis-entropy-limit} satisfies
$\bar u^k=u_0^k$ for every $k\geq0$. Moreover, for every fixed integer $K\geq0$,
\begin{equation}
\label{eq:app-tsallis-finite-horizon}
    \max_{0\leq k\leq K}
    \|u_\delta^k-u_0^k\|_{H^1(\Omega)}
    \to 0
    \qquad
    \text{as }\delta\downarrow0,
\end{equation}
and
\begin{equation}
\label{eq:app-tsallis-finite-horizon-bregman}
    \max_{0\leq k\leq K}
    \left|
    \int_\Omega
    D_r(u_\delta^*+\delta,u_\delta^k+\delta)
    \diff x
    -
    \int_\Omega
    D_r(u_0^*,u_0^k)
    \diff x
    \right|
    \to 0
    \qquad
    \text{as }\delta\downarrow0.
\end{equation}
\end{lemma}

\begin{proof}
The compactness part is identical to
\Cref{lem:app-shannon-finite-horizon}. In particular, for each fixed
$k$,
\[
    u_\delta^k
    \to
    \bar u^k
    \quad\text{in }L^2(\Omega),
    \qquad
    u_\delta^k
    \rightharpoonup
    \bar u^k
    \quad\text{in }H^1_0(\Omega).
\]
Using \Cref{lem:app-tsallis-entropy-limit}, we may pass to the limit in
the compatible equation. Assuming inductively that
$\bar u^{k-1}=u_0^{k-1}$ gives
\begin{equation}
\label{eq:app-tsallis-limit-equation}
    \alpha(\nabla\bar u^k,\nabla v)
    +
    \left\langle
        r_q'(\bar u^k)-r_q'(u_0^{k-1}),
        v
    \right\rangle_{H^{-1},H_0^1}
    =0, \quad \forall v \in H_0^1(\Omega).
\end{equation}

On the other hand, $u_0^k=c_kw$ satisfies
\eqref{eq:app-tsallis-endpoint-EL}. Subtracting \eqref{eq:app-tsallis-endpoint-EL} from \eqref{eq:app-tsallis-limit-equation} and  testing with $v=\bar u^k-u_0^k$
gives
\begin{align}
\alpha
\|\nabla(\bar u^k-u_0^k)\|_{L^2(\Omega)}^2
&+
\left\langle
    r_q'(\bar u^k)-r_q'(u_0^k),
    \bar u^k-u_0^k
\right\rangle_{H^{-1},H^1_0}
=0.
\label{eq:app-tsallis-endpoint-uniqueness}
\end{align}
Since $r_q$ is convex, both terms in
\eqref{eq:app-tsallis-endpoint-uniqueness} are nonnegative, and hence
    $\bar u^k=u_0^k$.

To obtain $H^1$ convergence, the same argument as in
\eqref{eq:app-stability-dual-estimate} gives
\begin{align*}
\alpha
\|\nabla(u_\delta^k-u_0^k)\|_{L^2}^2
\leq
\Bigl(
&
\|r_q'(u_\delta^k+\delta)-r_q'(u_0^k)\|_{H^{-1}(\Omega)}
\\
&+
\|r_q'(u_\delta^{k-1}+\delta)-r_q'(u_0^{k-1})\|_{H^{-1}(\Omega)}
\Bigr)
\|u_\delta^k-u_0^k\|_{H_0^1(\Omega)}.
\end{align*}
By \Cref{lem:app-tsallis-entropy-limit} and Poincar\'e's inequality, we obtain
\[
\|u_\delta^k-u_0^k\|_{H^1(\Omega)}
\to0.
\]
Taking the maximum proves
\eqref{eq:app-tsallis-finite-horizon}.

The proof of \eqref{eq:app-tsallis-finite-horizon-bregman} follows the same argument
as in \Cref{lem:app-shannon-finite-horizon}, with two modifications.
First, by \eqref{eq:app-tsallis-dual-bregman}, the Bregman divergence is represented as
\begin{align*}
\int_\Omega
D_r(u_\delta^*+\delta,u_\delta^k+\delta)
\diff x
=
|\Omega|r_q(\delta)
-
\int_\Omega r_q(u_\delta^k+\delta)\diff x
+
\left\langle
    r_q'(u_\delta^k+\delta),
    u_\delta^k
\right\rangle_{H^{-1},H_0^1}.
\end{align*}
Since \(2-q\in(0,1)\) and \(u_\delta^k+\delta\to u_0^k\) in \(L^2(\Omega)\), we have
\[
    r_q(u_\delta^k+\delta)
    \to
    r_q(u_0^k)
    \qquad
    \text{in }L^1(\Omega).
\]
Second, the convergence of the last term follows from
\begin{equation*}
    r_q'(u_\delta^k+\delta)
    \to
    r_q'(u_0^k)
    \quad\text{in }H^{-1}(\Omega)
\quad \text{and} \quad
    u_\delta^k\to u_0^k
    \quad\text{in }H_0^1(\Omega).
\end{equation*}
Consequently,
\[
\int_\Omega
D_r(u_\delta^*+\delta,u_\delta^k+\delta)
\diff x
\to
\int_\Omega
D_r(u_0^*,u_0^k)
\diff x.
\]
Taking the maximum over \(0\leq k\leq K\) proves
\eqref{eq:app-tsallis-finite-horizon-bregman}.
\end{proof}

\begin{theorem}[Sharpness for Tsallis entropy]
For $\{u_\delta^k\}$ defined by \eqref{eq:app-compatible-iteration}, it holds
\begin{equation}
\label{eq:app-tsallis-uniform-sharpness}
    \liminf_{k\to\infty}
    k^{1/q}
    \sup_{0<\delta\leq1}
    \|u_\delta^k-u_\delta^*\|_{H^1(\Omega)}
    \geq
    L_q
    >0,
\end{equation}
and
\begin{equation}
\label{eq:app-tsallis-uniform-bregman-sharpness}
    \liminf_{k\to\infty}
    k^{(2-q)/q}
    \sup_{0<\delta\leq1}
    \int_\Omega
    D_r(u_\delta^*+\delta,u_\delta^k+\delta)
    \diff x
    \geq
    L_q^D
    >0.
\end{equation}
\end{theorem}

\begin{proof}
By \eqref{eq:app-exact-zero} and \eqref{eq:app-tsallis-finite-horizon}, for every fixed
$k$,
\begin{align*}
\sup_{0<\delta\leq1}
\|u_\delta^k-u_\delta^*\|_{H^1(\Omega)} = \sup_{0 < \delta \leq 1} \|u_\delta^k\|_{H^1(\Omega)}
\geq
\lim_{\delta\downarrow0}
\|u_\delta^k\|_{H^1(\Omega)}
=
\|u_0^k\|_{H^1(\Omega)}.
\end{align*}
By \eqref{eq:app-tsallis-finite-horizon-bregman}, for fixed $k$,
\begin{align*}
\sup_{0<\delta\leq1}
\int_\Omega
D_r(u_\delta^*+\delta,u_\delta^k+\delta)
\diff x
\geq
\lim_{\delta\downarrow0}
\int_\Omega
D_r(u_\delta^*+\delta,u_\delta^k+\delta)
\diff x
=
\int_\Omega
D_r(u_0^*,u_0^k)
\diff x.
\end{align*}
Multiplying by $k^{1/q}$ and $k^{(2-q)/q}$, respectively, taking the lower limit, and using
\eqref{eq:app-tsallis-asymptotic} and \eqref{eq:app-tsallis-bregman-asymptotic} proves
\eqref{eq:app-tsallis-uniform-sharpness} and \eqref{eq:app-tsallis-uniform-bregman-sharpness}.
\end{proof}

For the Tsallis entropy, the one-sided Bregman growth exponent is
$\theta=2-q$, so \Cref{thm:main-sublinear-convergence} gives the sublinear rates
\begin{equation*}
    \int_\Omega D_R(u^*,u^k) \diff x
    = \mathcal O\left(k^{-(2-q)/q}\right), \quad \|u^k-u^*\|_{H^1(\Omega)}
    = \mathcal O(k^{-1/q}).
\end{equation*}
The family constructed above attains these rates in the uniform worst-case
sense. Hence the exponents $(2-q)/q$ and $1/q$ cannot
be improved while retaining a constant which is uniform over $0<\delta\leq1$.

\section{Verification of \Cref{tab:convergence-rates-model-legendre}}
\label{app:one-sided-growth-exponents}

In this appendix, we verify the exponents in \Cref{tab:convergence-rates-model-legendre}.
Throughout, \(M>0\) is fixed and $0\leq s \leq M$, $0 < t\leq M$. We recall the
definition of the Bregman divergence,
\begin{equation*}
    D_r(s,t)\coloneqq r(s)-r(t)-r'(t)(s-t).
\end{equation*}

For $0 < s < t$, Taylor's formula gives an integral-type expression of the Bregman divergence,
\begin{equation}
\label{eq:bregman-integral-formula}
    D_r(s,t)
    =
    \int_s^t \bigl(r'(t)-r'(\xi)\bigr)\diff \xi
    =
    \int_s^t (\eta-s) r''(\eta)\diff \eta .
\end{equation}
The same estimate for \(s=0\) follows by taking the limit
\(s\downarrow0\), since the model entropies have finite continuous
extensions at the origin.

\begin{itemize}
    \item For the Shannon entropy, we have
    $r_{\rm sh}(s)=s\log s - s$ and $r_{\rm sh}''(s)=\frac1s$.
Using \eqref{eq:bregman-integral-formula},
\begin{equation*}
    D_{r_{\rm sh}}(s,t)
    =
    \int_s^t \frac{\eta-s}{\eta}\diff \eta
    \leq
    \int_s^t 1\diff \eta
    =
    t-s .
\end{equation*}
Thus the Shannon entropy has the one-sided Bregman growth exponent $\theta=1$.

\item
For the Spence entropy, we have $r_{\rm sp}(s)
    =
    \frac12 s^2+\li_2(e^{-s})-\frac{\pi^2}{6}$ and
    $r_{\rm sp}''(s)
    =
    \frac{e^s}{e^s-1}$.
Using \eqref{eq:bregman-integral-formula}, we obtain
\begin{equation*}
    D_{r_{\rm sp}}(s,t)
    =
    \int_s^t
    (\eta-s)\frac{1}{1-e^{-\eta}}\diff \eta
    \leq
    \int_s^t
    \frac{\eta}{1-e^{-\eta}}\diff \eta .
\end{equation*}
Since \(\eta/(1-e^{-\eta})\) is continuous on \([0,M]\) after defining
its value at \(0\) to be \(1\), we may take
    $C_M=\frac{M}{1-e^{-M}}$.
Consequently,
\begin{equation*}
    D_{r_{\rm sp}}(s,t)
    \leq
    C_M\int_s^t 1\diff \eta
    =
    C_M(t-s).
\end{equation*}
Hence its one-sided Bregman growth exponent is $\theta=1$.

\item Let \(1<q<2\). The Tsallis entropy satisfies $r_q''(s)=s^{-q}$.
By \eqref{eq:bregman-integral-formula},
\begin{equation*}
    D_{r_q}(s,t)
    =
    \int_s^t (\eta-s)\eta^{-q}\diff \eta .
\end{equation*}
Since $0<\eta-s\leq\eta$ and $t\mapsto t^{-q}$ is decreasing for
$q>0$, we obtain $\eta^{-q}\leq(\eta-s)^{-q}$.
Then,
\begin{equation*}
    D_{r_q}(s,t)
    \leq \int_s^t  (\eta - s)^{1-q} \diff  \eta
    = \frac{1}{2-q}(t-s)^{2-q}.
\end{equation*}
Therefore the Tsallis entropy has one-sided Bregman growth exponent $\theta=2-q$.

\item
Let \(0<\kappa<1\). The Kaniadakis entropy satisfies $r_\kappa''(z)
    =
    \frac12\left(z^{\kappa-1}+z^{-\kappa-1}\right)$.
Then we get
\begin{equation*}
    D_{r_\kappa}(s,t)
    =
    \frac12\int_s^t(\eta-s)
    \left(\eta^{\kappa-1}+\eta^{-\kappa-1}\right)\diff \eta.
\end{equation*}
Since \(\eta \geq \eta - s\), we have
\begin{align*}
    &\int_s^t (\eta - s)\eta^{-\kappa-1}\diff  \eta
    \leq
    \int_s^t (\eta - s)^{-\kappa}\diff \eta
    =
    \frac{(t-s)^{1-\kappa}}{1-\kappa}, \\
    &\int_s^t (\eta - s)\eta^{\kappa-1}\diff  \eta
    \leq
    \int_s^t (\eta - s)^{\kappa}\diff  \eta
    =
    \frac{(t-s)^{1+\kappa}}{1+\kappa}
    \leq
    \frac{M^{2\kappa}}{1+\kappa}(t-s)^{1-\kappa}.
\end{align*}
Consequently, there exists \(C_{M}>0\) such that
\begin{equation*}
    D_{r_\kappa}(s,t)
    \leq
    C_{M}(t-s)^{1-\kappa}.
\end{equation*}
Hence the Kaniadakis entropy has one-sided Bregman growth exponent $\theta=1-\kappa$.
\end{itemize}

\section{Proof of \Cref{lem:discrete-bihari-lasalle}}
\label{app:discrete-bihari-lasalle}

\begin{proof}
The recursion \eqref{eq:discrete-bihari-hypothesis} implies that $\{B_k\}$ is nonincreasing. The cases
$B_0=0$ or $B_j=0$ for some $j$ are immediate, so assume $B_k>0$.
Set $x_k=\lambda_kB_k^{\beta-1}$. Then
\begin{align*}
B_k^{1-\beta}-B_{k-1}^{1-\beta}
&\ge B_k^{1-\beta}\bigl[1-(1+x_k)^{1-\beta}\bigr]\\
&\ge
B_k^{1-\beta}\frac{(\beta-1)x_k}{(1+x_k)^\beta}\\
&\ge
\frac{(\beta-1)\lambda_k}
     {(1+\lambda_kB_0^{\beta-1})^\beta}.
\end{align*}
Here the second inequality follows from the elementary inequality
\begin{equation*}
    (1+x)^{1-\beta} \leq 1 - \frac{\beta-1}{(1+x_0)^{\beta}}\,x
    \qquad
    \forall\, 0\leq x\leq x_0,
\end{equation*} and the
last follows from $B_k\le B_0$. Summing from $1$ to $k$ proves \eqref{eq:discrete-bihari-conclusion}. The constant-step rate follows immediately.
\end{proof}

\section{Proof of \Cref{prop:strict-complementarity-free-boundary}}
\label{app:strict-complementarity-free-boundary}

In \Cref{lem:regularity-u-star}, we collect the regularity properties of \(u^*\)
and \(w^*=u^*-\phi\) that will be used in the proof of
\Cref{prop:strict-complementarity-free-boundary}.

\begin{lemma}[Regularity of \(u^*\) and \(w^*\)]
\label{lem:regularity-u-star}

Suppose that
\Cref{ass:strict-complementarity-free-boundary} holds. Then the following
regularity properties hold.

\begin{itemize}
    \item
    Interior \(C^{1,\alpha}\)-regularity.
    The solution satisfies
    \begin{equation}
        \label{eq:w-star-local-c1alpha}
        u^*\in C_{\mathrm{loc}}^{1,\alpha}(\Omega) \quad \text{and} \quad w^*\in C_{\mathrm{loc}}^{1,\alpha}(\Omega).
    \end{equation}

    \item
    Uniform one-sided \(C^{2,\alpha}\)-regularity. There exist constants
    \(\rho_{\mathrm{reg}}>0\) and \(C_{\mathrm{reg}}>0\) such that
    \begin{equation}
    \label{eq:w-star-one-sided-c2alpha}
        w^*
        \in
        C^{2,\alpha}
        \bigl(
            \overline{
                \mathcal N_{\rho_{\mathrm{reg}}}^+(\Gamma)
            }
        \bigr)
    \end{equation}
    and
    \begin{equation}
    \label{eq:w-star-one-sided-c2alpha-bound}
        \|w^*\|_{
            C^{2,\alpha}
            (
                \overline{
                    \mathcal N_{\rho_{\mathrm{reg}}}^+(\Gamma)
                }
            )
        }
        \leq C_{\mathrm{reg}},
    \end{equation}
    where
        $\mathcal N_{\rho_{\mathrm{reg}}}^+(\Gamma)
        =
        \Omega_+ \cap
        \left\{
            \dist(x,\Gamma)
            <
            \rho_{\mathrm{reg}}
        \right\}$.

    \item
    Continuity up to the boundary. The solution \(u^*\) admits a continuous
    representative on \(\overline\Omega\), and
    \begin{equation}
        \label{eq:w-star-boundary-continuity}
        u^*\in C(\overline\Omega) \quad \text{and} \quad  w^*\in C(\overline\Omega).
    \end{equation}
\end{itemize}
\end{lemma}

\begin{proof}[Proof of \Cref{lem:regularity-u-star}]
We prove the three properties separately.

\medskip \noindent \textbf{Part 1. Interior \(C^{1,\alpha}\)-regularity.}

By \cite[Theorem 1.1]{andreucci2023classical}, we have $u^*\in C_{\mathrm{loc}}^{1,\alpha}(\Omega)$.
Since \(\phi\in C^{2,\alpha}(\overline\Omega)\), we also have $w^*=u^*-\phi \in
    C_{\mathrm{loc}}^{1,\alpha}(\Omega)$.

\medskip \noindent \textbf{Part 2. Uniform one-sided
\(C^{2,\alpha}\)-regularity.}

On the inactive set \(\Omega_+\), we have
\begin{equation*}
    \nabla\cdot(A\nabla w^*)
    =
    -\nabla\cdot(A\nabla\phi)
    =
    q_\phi
    \qquad\text{in } H^{-1}(\Omega_+),
\end{equation*}
with \(w^*=0\) on \(\Gamma\). Moreover,
\begin{equation*}
    A\in C^{1,\alpha}(\overline\Omega),
    \qquad
    q_\phi\in C^{0,\alpha}(\overline\Omega).
\end{equation*}
Since \(\Gamma\) is a compact embedded \(C^{2,\alpha}\) submanifold and
\(\Omega_+\) lies locally on one side of \(\Gamma\), the boundary regularity
theory for weak solutions, followed by the boundary Schauder estimates, yields
one-sided \(C^{2,\alpha}\)-regularity of \(w^*\) up to \(\Gamma\); see
\cite[Theorems 8.12 and 9.19]{gilbarg2001elliptic}. By compactness of
\(\Gamma\), there exist \(\rho_{\mathrm{reg}}>0\) and \(C_{\mathrm{reg}}>0\)
such that \eqref{eq:w-star-one-sided-c2alpha} and
\eqref{eq:w-star-one-sided-c2alpha-bound} hold.

\medskip \noindent \textbf{Part 3. Continuity up to the boundary.}

We first prove that \(u^*\geq 0\) a.e. in \(\Omega\).
Let
\begin{equation*}
    (u^*)_-\coloneqq \max\{-u^*,0\}.
\end{equation*}
Since \(g=0\), we have \((u^*)_-\in H_0^1(\Omega)\). Moreover,
\begin{equation*}
    v\coloneqq u^*+(u^*)_-=(u^*)_+
\end{equation*}
belongs to \(K\). Indeed, on \(\{u^*\geq0\}\) one has \(v=u^*\geq\phi\),
whereas on \(\{u^*<0\}\),
\begin{equation*}
    \phi\leq u^*<0=v.
\end{equation*}
Since \(F = 0\), the obstacle variational inequality \eqref{eq:obstacle-vi} with
this choice of \(v\) gives
\begin{equation*}
    0
    \leq
    a\bigl(u^*,v-u^*\bigr)
    =
    a\bigl(u^*,(u^*)_-\bigr)
    =
    -a\bigl((u^*)_-,(u^*)_-\bigr).
\end{equation*}
By uniform ellipticity \eqref{eq:uniform-ellipticity} of \(A\),
\begin{equation*}
    0
    \leq
    -\frac{1}{\Lambda}
    \|\nabla (u^*)_-\|_{L^2(\Omega)}^2
    \leq 0.
\end{equation*}
Consequently, \((u^*)_-=0\) a.e. in \(\Omega\), and hence
\begin{equation*}
    u^*\geq0
    \qquad\text{a.e. in }\Omega.
\end{equation*}

Since \(g=0\), the strict boundary compatibility condition
\eqref{eq:boundary-compatibility} gives
\begin{equation*}
    \trace\phi\leq -\delta_0
    \qquad\text{on }\partial\Omega.
\end{equation*}
By the continuity of \(\phi\) on \(\overline\Omega\), there exists
\(\varepsilon>0\) such that
\begin{equation*}
    \phi\leq-\frac{\delta_0}{2}
    \qquad\text{in }
    U_\varepsilon
    \coloneqq
    \left\{
        x\in\Omega \mid
        \dist(x,\partial\Omega)<\varepsilon
    \right\}.
\end{equation*}
Combining this estimate with \(u^*\geq0\), we obtain
\begin{equation*}
    u^*-\phi
    \geq
    \frac{\delta_0}{2}
    \qquad\text{a.e. in }U_\varepsilon.
\end{equation*}
Thus the obstacle constraint is uniformly inactive in \(U_\varepsilon\).
Therefore, the variational equality holds,
\begin{equation*}
    -\nabla\cdot(A\nabla u^*)=0
    \quad\text{in } H^{-1}(U_\varepsilon), \quad \text{with} \quad
    \trace u^*=0
    \quad\text{on }\partial\Omega.
\end{equation*}

Boundary regularity for weak solutions of uniformly elliptic divergence-form
equations on \(C^{1,\alpha}\) domains therefore gives continuity up to
\(\partial\Omega\); see
\cite[Corollary~8.36]{gilbarg2001elliptic}.
Together with Part 1, this yields $u^*\in C(\overline\Omega)$.
Since \(\phi\in C(\overline\Omega)\), it follows that $w^*=u^*-\phi\in C(\overline\Omega)$.

\end{proof}

\begin{proof}[Proof of \Cref{prop:strict-complementarity-free-boundary}]
Set
    $q_\phi
    \coloneqq
    -\nabla\cdot(A\nabla\phi)$.
By assumption, we have $q_\phi\geq c_\phi>0$ in $\Omega$.
We divide the proof into four parts.

\medskip \noindent \textbf{Part 1. Strict complementarity on the active set.}

The obstacle multiplier is given by
    $\lambda^*
    =
    -\nabla\cdot(A\nabla u^*)$
in the sense of distributions. On \(\intr(\Omega_0)\), we have \(u^* = \phi\) by
definition. It follows that
\begin{equation*}
    \lambda^*
    =
    -\nabla\cdot(A\nabla u^*)
    =
    -\nabla\cdot(A\nabla\phi)
    =
    q_\phi \geq c_{\phi}
    \qquad\text{a.e. on }\intr(\Omega_0).
\end{equation*}
Since \(\Gamma = \partial \Omega_0 \cap \Omega\) is of zero measure, it holds that
\begin{equation*}
    \lambda^* \geq c_{\phi} \qquad\text{a.e. on }\Omega_0.
\end{equation*}

\medskip \noindent \textbf{Part 2. One-sided tubular coordinates.}

By \Cref{ass:strict-complementarity-free-boundary}, $\Gamma$ is a compact
embedded $C^{2,\alpha}$ submanifold of dimension $d-1$, all of whose points
are regular free-boundary points. By
\Cref{rem:geometry-near-regular-points}, the active and inactive sets lie
locally on opposite sides of $\Gamma$. Hence, on each connected component
of $\Gamma$, the side occupied by $\Omega_+$ determines a consistent choice
of unit normal field. Since $\Gamma$ is a compact manifold, it has finitely
many connected components. The following construction is performed on each
component separately. For notational simplicity, we present the argument
for a single connected component and suppress the component index.

By the tubular neighborhood theorem, there exists \(\rho_1>0\) such
that, after choosing the orientation of the unit normal field \(\nu\)
to point into \(\Omega_+\), the map
\begin{equation*}
    \Psi:
    \Gamma\times(0,\rho_1)
    \to
    \mathcal N_{\rho_1}^+(\Gamma),
    \qquad
    \Psi(y,s)=y+s\nu(y),
\end{equation*}
is a \(C^{1,\alpha}\)-diffeomorphism; see
\cite[Theorem 6.24]{lee2003smooth} and \cite{li2005regularity}.
Moreover,
\begin{equation*}
    \dist(\Psi(y,s),\Gamma)=s
    \qquad
    \text{for all }
    (y,s)\in\Gamma\times(0,\rho_1).
\end{equation*}

Let $\kappa_1(y),\ldots,\kappa_{d-1}(y)$
denote the principal curvatures of \(\Gamma\) at \(y\). The Jacobian of the
normal coordinate map satisfies
\begin{equation}
    |J_\Psi(y,s)|
    =
    \prod_{i=1}^{d-1}
    |1-s\kappa_i(y)|.
\end{equation}
Since \(\Gamma\) is compact and of class \(C^{2,\alpha}\), its principal
curvatures are uniformly bounded. Hence, there exist constants \(0 <\rho_2 \leq \rho_1\) and \(C_J\geq1\) such that
\begin{equation}
    C_J^{-1}
    \leq
    |J_\Psi(y,s)|
    \leq
    C_J
    \qquad
    \forall (y,s)\in\Gamma\times(0,\rho_2).
\end{equation}

Since \(\Psi\) is a \(C^{1,\alpha}\)-diffeomorphism, composition with \(\Psi\)
preserves \(H^1\)-regularity. Hence, for every \(e\in H^1(\Omega)\),
\begin{equation*}
    e\circ\Psi
    \in
    H^1\bigl(\Gamma\times(0,\rho_2)\bigr).
\end{equation*}
In particular, for almost every \(y\in\Gamma\), the function $s\mapsto e(\Psi(y,s))$
belongs to \(H^1(0,\rho_2)\), and
\begin{equation*}
    e\circ\Psi
    \in
    L^2\bigl(\Gamma;H^1(0,\rho_2)\bigr).
\end{equation*}
The chain rule gives
\begin{equation*}
    \partial_s(e\circ\Psi)(y,s)
    =
    \nabla e(\Psi(y,s))\cdot\nu(y)
\end{equation*}
for a.e. \((y,s)\in\Gamma\times(0,\rho_2)\). Hence
\begin{equation*}
    |\partial_s(e\circ\Psi)(y,s)|
    \leq
    |\nabla e|(\Psi(y,s))
\end{equation*}
for a.e. \((y,s)\).

\medskip \noindent \textbf{Part 3. Quadratic separation.}

On \(\Omega_+\), we have
    $-\nabla\cdot(A\nabla u^*)=0$ in $\Omega_+$.
It follows that
    $\nabla\cdot(A\nabla w^*)
    =
    -\nabla\cdot(A\nabla\phi)
    =
    q_\phi$
    in $\Omega_+$.

By \eqref{eq:w-star-local-c1alpha}, the function \(w^*\) is \(C^{1,\alpha}\) in
a neighborhood of \(\Gamma\). Since $w^*\geq0$ in $\Omega$, $w^*=0$ on $\Gamma$,
every point of \(\Gamma\) is a local minimum of \(w^*\). Therefore, $\nabla w^*=0$ on $\Gamma$.

By \eqref{eq:w-star-one-sided-c2alpha}, there exists
    $0 < \rho_3 < \min(\rho_2, \rho_{\mathrm{reg}})$
such that
\begin{equation}
\label{eq:w-schauder-tube}
    w^*
    \in
    C^{2,\alpha}
    \bigl(
        \overline{\mathcal N_{\rho_3}^+(\Gamma)}
    \bigr).
\end{equation}

The following proof follows the classical tangential-normal argument from the
regularity theory of obstacle problems; see \cite{kinderlehrer1977regularity}.
Since $w^*=0$ and $\nabla w^*=0$ on $\Gamma$,
differentiating the identity \(\nabla w^*=0\) along the tangential direction
$\tau\in T_y\Gamma$, where \(T_y\Gamma\) denotes the tangent
space of \(\Gamma\) at \(y\), gives
\begin{equation*}
    D^2w^*(y)\tau=0,
    \qquad
    \tau\in T_y\Gamma.
\end{equation*}
Since \(T_y\Gamma\) has codimension one and \(D^2w^*(y)\) is symmetric,
it follows that
\begin{equation}
\label{eq:hessian-normal-form-proof}
    D^2w^*(y)
    =
    \partial_{\nu\nu}w^*(y)\,
    \nu(y)\otimes\nu(y).
\end{equation}

Using \(\nabla w^*(y)=0\) and \eqref{eq:hessian-normal-form-proof}, we deduce that
\begin{equation*}
    q_\phi(y)
    =
    \bigl(\nu(y)^\top A(y)\nu(y)\bigr)
    \partial_{\nu\nu}w^*(y), \quad \forall y \in \Gamma.
\end{equation*}
Therefore,
\begin{equation*}
    \partial_{\nu\nu}w^*(y)
    =
    \frac{q_\phi(y)}
    {\nu(y)^\top A(y)\nu(y)}.
\end{equation*}

Uniform ellipticity \eqref{eq:uniform-ellipticity} and the assumption on the obstacle \eqref{eq:strong-super-harmonicity} imply
\begin{equation}
\label{eq:bound-partial-nunu}
    \frac{c_\phi}{\Lambda}
    \leq
    \partial_{\nu\nu}w^*(y)
    \leq
    \Lambda\|q_\phi\|_{L^\infty(\Omega)}
    \qquad
    \text{for all }y\in\Gamma.
\end{equation}

By \eqref{eq:w-schauder-tube}, Taylor's theorem in the normal
direction gives
\begin{equation*}
    w^*(\Psi(y,s))
    =
    a(y)s^2+r(y,s),
\end{equation*}
where
    $a(y)
    \coloneqq
    \frac{q_\phi(y)}
    {2\,\nu(y)^\top A(y)\nu(y)}$
and by \eqref{eq:w-star-one-sided-c2alpha-bound}, there exists a constant
\(C_T>0\), independent of \(y\in\Gamma\), such that
\begin{equation*}
    |r(y,s)|
    \leq
    C_Ts^{2+\alpha}
    \qquad
    \text{for all }
    (y,s)\in\Gamma\times(0,\rho_3).
\end{equation*}

Define
\begin{equation*}
    m_0
    \coloneqq
    \frac{c_\phi}{2\Lambda},
    \qquad
    M_0
    \coloneqq
    \frac{\Lambda\|q_\phi\|_{L^\infty(\Omega)}}{2}.
\end{equation*}
Then by \eqref{eq:bound-partial-nunu},
\begin{equation*}
    0<m_0\leq a(y)\leq M_0
    \qquad
    \text{for all }y\in\Gamma.
\end{equation*}

Choose $0<\rho_4\leq\rho_3$
sufficiently small that
\begin{equation*}
    C_T\rho_4^\alpha
    \leq
    \frac{m_0}{2}.
\end{equation*}
For \((y,s)\in\Gamma\times(0,\rho_4)\), we then have
\begin{align}
    w^*(\Psi(y,s))
    &\geq
    m_0s^2-C_Ts^{2+\alpha}
    \geq
    \frac{m_0}{2}s^2,
    \nonumber \\
    w^*(\Psi(y,s))
    &\leq
    M_0s^2+C_Ts^{2+\alpha}
    \leq
    \left(M_0+\frac{m_0}{2}\right)s^2.
\end{align}
Hence, with
\begin{equation*}
    c_1
    \coloneqq
    \max\left\{
        1,\,
        \frac{2}{m_0},\,
        M_0+\frac{m_0}{2}
    \right\},
\end{equation*}
we obtain
\begin{equation*}
    c_1^{-1}s^2
    \leq
    w^*(\Psi(y,s))
    \leq
    c_1s^2
\end{equation*}
for all \((y,s)\in\Gamma\times(0,\rho_4)\).

Finally, set $\rho\coloneqq \rho_4$.

Because $\dist(\Psi(y,s),\Gamma)=s$,
the parametrized estimate is equivalent to
\begin{equation*}
    c_1^{-1}\dist(x,\Gamma)^2
    \leq
    w^*(x)
    \leq
    c_1\dist(x,\Gamma)^2
\end{equation*}
for every $x\in \mathcal{N}^+_{\rho}(\Gamma)$.

\medskip \noindent \textbf{Part 4. Separation away from the free boundary.}

By \eqref{eq:w-star-boundary-continuity}, we have
\begin{equation*}
    w^*=u^*-\phi\in C(\overline\Omega).
\end{equation*}

Fix \(\rho_0\in(0,\rho)\). Suppose, by contradiction, that no positive constant
\(c_{\rho_0}\) exists. Then there is a sequence \(\{x_n\}\subset\Omega_+\) such
that
\begin{equation*}
    \dist(x_n,\Gamma)\geq\rho_0,
    \qquad
    w^*(x_n)\to 0.
\end{equation*}
Since \(\overline{\Omega_+}\) is compact, after passing to a subsequence,
there exists \(\bar x\in\overline{\Omega_+}\) such that $x_n\to\bar x$. The continuity of \(w^*\) gives $w^*(\bar x)=0$.

The strict boundary compatibility condition \eqref{eq:boundary-compatibility}
implies
\begin{equation*}
    w^*=u^*-\phi
    =
    g-\phi
    \geq\delta_0
    \qquad\text{on }\partial\Omega.
\end{equation*}
Therefore,
    $\bar x\notin\partial\Omega$.
It follows that
    $\bar x\in \Gamma$.
But continuity of the distance function gives
\begin{equation*}
    \dist(\bar x,\Gamma)
    =
    \lim_{n\to\infty}
    \dist(x_n,\Gamma)
    \geq\rho_0,
\end{equation*}
which contradicts \(\bar x\in\Gamma\).

Therefore, there exists \(c_{\rho_0}>0\) such that
\begin{equation*}
    w^*(x)\geq c_{\rho_0}
\end{equation*}
for every
    $x\in
    \Omega_+
    \cap
    \{\dist(x,\Gamma)\geq\rho_0\}$.
This completes the proof.
\end{proof}

\section{Proof of \Cref{lem:delicate-bregman-bound}}
\label{app:delicate-bregman-bound}

\begin{proof}
Set $a=v-\phi$ and $b=u-\phi$, so $a \in [0,M_1]$, $b \in (0,M_1]$ and $a \leq b$.
For the Shannon entropy, using $\log t\ge 1-t^{-1}$, we have
\[
D_{r_{\rm sh}}(a,b)
=b-a-a\log(b/a)
\le \frac{(b-a)^2}{b},
\]
with the case $a=0$ understood by continuity.

For the Spence entropy,
\[
r_{\rm sp}''(\eta)
=
\frac{1}{1-e^{-\eta}}
\leq
\frac{M_1}{1-e^{-M_1}}\frac1\eta
=
K_{M_1}r_{\rm sh}''(\eta),
\qquad 0<\eta\leq M_1.
\]
Therefore, by the integral formula of Bregman divergence \eqref{eq:bregman-integral-formula}, we obtain
\[
D_{r_{\rm sp}}(a,b)
\leq
K_{M_1}D_{r_{\rm sh}}(a,b)
\leq
K_{M_1}\frac{(b-a)^2}{b}.
\]
with the case \(a=0\) understood by letting \(a\downarrow0\).
\end{proof}

\section{Proof of \Cref{lem:one-dim-weighted}}
\label{app:one-dim-weighted}

\begin{proof}
Set
$a\coloneqq w(0)$ and
$q\coloneqq \|w'\|_{L^2(0,\rho)}$.
If \(a=0\), the conclusion follows directly from the
one-dimensional Hardy inequality and
\[
\frac{w(s)^2}{s^2+w(s)}
\leq
\frac{w(s)^2}{s^2}.
\]
Suppose \(a>0\) and set \(\sigma\coloneqq\sqrt a\).
By extending $w$ constantly when $\rho<\sigma$, we may assume
$\sigma\le\rho$.
Since \(w\geq0\), for \(s\in(0,\rho]\),
\[
\frac{w(s)^2}{s^2+w(s)}
\leq
\min\left\{w(s),\frac{w(s)^2}{s^2}\right\},
\qquad
w(s)\leq a+s^{1/2}q.
\]
Consequently, applying Hardy's inequality to \(w-a\) and using Young's inequality, we obtain
\begin{align*}
\int_0^\rho\frac{w(s)^2}{s^2+w(s)}\diff s
&\leq
\int_0^\sigma w(s)\diff s
+
\int_\sigma^\rho\frac{w(s)^2}{s^2}\diff s\\
&\leq
a^{3/2}+\frac23a^{3/4}q
+2\int_0^\rho\frac{(w(s)-a)^2}{s^2}\diff s
+2a^2\int_\sigma^\infty s^{-2}\diff s\\
&\leq
3a^{3/2}+\frac23a^{3/4}q+8q^2\\
&\leq
\frac{10}{3}a^{3/2}+\frac{25}{3}q^2.
\end{align*}
Thus the claim
holds, for example, with \(C=25/3\).
\end{proof}

\section{Proof of \Cref{lem:trace-interpolation}}
\label{app:trace-interpolation}

\begin{proof}
Write
    $D=\displaystyle\bigsqcup_{j=1}^J D_j$,
where each \(D_j\) is a bounded connected Lipschitz domain. We prove the
estimate on each \(D_j\). Fix \(j\in\{1,\ldots,J\}\). Since $\|w\|_{H^1(D_j)}\leq\|w\|_{H^1(D)}\leq M$,
the same uniform \(H^1\)-bound is available on every component.
We divide the argument into three steps.

\medskip
\noindent\textit{Step 1: Trace of \(w^{3/2}\).}
Set \(v\coloneqq w^{3/2}\). Since \(w\in H^1(D_j)\) and \(w\geq0\),
we have \(v\in W^{1,1}(D_j)\), with
\begin{equation*}
    |\nabla v|
    =
    \frac{3}{2}\,w^{1/2}|\nabla w|
    \qquad\text{a.e. in }D_j.
\end{equation*}
By the \(W^{1,1}\) trace theorem on the Lipschitz domain \(D_j\),
there exists \(C_{1,j}=C_{1,j}(D_j)>0\) such that
\begin{equation*}
    \|v\|_{L^1(\partial D_j)}
    \leq
    C_{1,j}
    \left(
        \|v\|_{L^1(D_j)}
        +
        \|\nabla v\|_{L^1(D_j)}
    \right).
\end{equation*}
Rewriting this estimate in terms of \(w\), we obtain
\begin{equation}
\label{eq:trace-step1}
    \|w\|_{L^{3/2}(\partial D_j)}^{3/2}
    \leq
    C_{1,j}
    \left(
        \|w\|_{L^{3/2}(D_j)}^{3/2}
        +
        \frac{3}{2}
        \int_{D_j}w^{1/2}|\nabla w| \diff x
    \right).
\end{equation}

\medskip
\noindent\textit{Step 2: Interpolation of the two terms.}
For the first term, the Cauchy--Schwarz inequality gives
\begin{equation}
\label{eq:trace-interp1}
    \|w\|_{L^{3/2}(D_j)}^{3/2}
    =
    \int_{D_j}w^{1/2}w \diff x
    \leq
    \|w\|_{L^1(D_j)}^{1/2}
    \|w\|_{L^2(D_j)}.
\end{equation}
For the second term, again by the Cauchy--Schwarz inequality,
\begin{equation}
\label{eq:trace-interp2}
    \int_{D_j}w^{1/2}|\nabla w| \diff x
    \leq
    \|w\|_{L^1(D_j)}^{1/2}
    \|\nabla w\|_{L^2(D_j)}.
\end{equation}
Substituting \eqref{eq:trace-interp1} and \eqref{eq:trace-interp2} into
\eqref{eq:trace-step1}, we obtain
\begin{equation}
\label{eq:trace-step2}
    \|w\|_{L^{3/2}(\partial D_j)}^{3/2}
    \leq
    C_{2,j}\,
    \|w\|_{L^1(D_j)}^{1/2}
    \left(
        \|w\|_{L^2(D_j)}
        +
        \|\nabla w\|_{L^2(D_j)}
    \right),
\end{equation}
where \(C_{2,j}=C_{2,j}(D_j)>0\).

\medskip
\noindent\textit{Step 3: Closing the estimate.}
We bound \(\|w\|_{L^2(D_j)}\) using Poincar\'e's inequality. Define
\begin{equation*}
    \overline w_j
    \coloneqq
    \frac{1}{|D_j|}
    \int_{D_j}w \diff x.
\end{equation*}
Then
\begin{align}
    \|w\|_{L^2(D_j)}
    &\leq
    \|w-\overline w_j\|_{L^2(D_j)}
    +
    \overline w_j|D_j|^{1/2}
    \nonumber\\
    &\leq
    C_{P,j}\|\nabla w\|_{L^2(D_j)}
    +
    |D_j|^{-1/2}\|w\|_{L^1(D_j)},
\end{align}
where \(C_{P,j}\) is the Poincar\'e constant of \(D_j\). Substituting this
estimate into \eqref{eq:trace-step2} yields
\begin{equation}
\label{eq:trace-step3}
    \|w\|_{L^{3/2}(\partial D_j)}^{3/2}
    \leq
    C_{3,j}\,
    \|w\|_{L^1(D_j)}^{1/2}
    \left(
        \|\nabla w\|_{L^2(D_j)}
        +
        \|w\|_{L^1(D_j)}
    \right),
\end{equation}
for some \(C_{3,j}=C_{3,j}(D_j)>0\). We estimate the two products separately.

\medskip
\noindent\textit{First product.}
By Young's inequality,
\begin{equation}
\label{eq:trace-young}
    \|w\|_{L^1(D_j)}^{1/2}
    \|\nabla w\|_{L^2(D_j)}
    \leq
    \frac{1}{2}\|w\|_{L^1(D_j)}
    +
    \frac{1}{2}\|\nabla w\|_{L^2(D_j)}^2.
\end{equation}

\medskip
\noindent\textit{Second product.}
Using the \(H^1\)-bound,
\begin{equation*}
    \|w\|_{L^1(D_j)}
    \leq
    |D_j|^{1/2}\|w\|_{L^2(D_j)}
    \leq
    |D_j|^{1/2}\|w\|_{H^1(D_j)}
    \leq
    |D_j|^{1/2}M.
\end{equation*}
Consequently,
\begin{equation}
\label{eq:trace-smallness}
    \|w\|_{L^1(D_j)}^{3/2}
    \leq
    |D_j|^{1/4}M^{1/2}
    \|w\|_{L^1(D_j)}.
\end{equation}

Combining \eqref{eq:trace-step3}, \eqref{eq:trace-young}, and
\eqref{eq:trace-smallness}, we conclude that
\begin{equation*}
    \|w\|_{L^{3/2}(\partial D_j)}^{3/2}
    \leq
    C_j
    \left(
        \|w\|_{L^1(D_j)}
        +
        \|\nabla w\|_{L^2(D_j)}^2
    \right),
\end{equation*}
where \(C_j=C_j(D_j,M)>0\).

Finally, since \(J<\infty\), we may set
    $C\coloneqq\max_{1\leq j\leq J}C_j$.
Using
$\partial D\subseteq \displaystyle\bigcup_{j=1}^J\partial D_j$,
and summing the componentwise estimates, we obtain
\begin{align}
    \|w\|_{L^{3/2}(\partial D)}^{3/2}
    \leq
    \sum_{j=1}^J
    \|w\|_{L^{3/2}(\partial D_j)}^{3/2}
    &\leq
    C\sum_{j=1}^J
    \left(
        \|w\|_{L^1(D_j)}
        +
        \|\nabla w\|_{L^2(D_j)}^2
    \right)
    \nonumber\\
    &=
    C
    \left(
        \|w\|_{L^1(D)}
        +
        \|\nabla w\|_{L^2(D)}^2
    \right).
\end{align}
This proves the desired estimate.
\end{proof}

\section*{Acknowledgments}
The authors thank Thomas M.~Surowiec for many interesting discussions and comments on the text.
BK and NRR were supported in part by the U.S. Department of Energy, Office of Science Early Career Research Program under Award Number DE-SC0024335 and by the Center for Information Geometric Mechanics and Optimization (CIGMO), a PSAAP-IV Focused Investigatory Center funded by the U.S. Department of Energy, National Nuclear Security Administration under Award Number DE-NA0004261. BK and HQ were also supported in part by the Alfred P. Sloan Foundation via a Sloan Research Fellowship in Mathematics.

\section*{Data Availability Statement}
The authors declare that the data supporting the findings of this study are available within the paper.

\printbibliography

\end{document}